\documentclass[10pt, reqno]{amsart}
\usepackage{amsfonts}
\usepackage{appendix}
\usepackage[a4paper, centering]{geometry}
\usepackage{xcolor}
\usepackage{amsthm}

\usepackage[utf8]{inputenc}
\usepackage[colorlinks=true,hyperindex,pagebackref,linktocpage=true]{hyperref}
\usepackage{cleveref}
\usepackage{amsfonts}
\usepackage{amssymb}
\usepackage{tikz-cd}
\usepackage[mathscr]{euscript}
\hypersetup{
	colorlinks,
	linkcolor={violet!60!black},
	citecolor={cyan!60!black},
	urlcolor={orange!60!black}
}
\usepackage{stmaryrd}
\usepackage{varwidth}
\theoremstyle{plain}
\newtheorem{theorem}{Theorem}[subsection]
\newtheorem{proposition}[theorem]{Proposition}
\newtheorem{lemma}[theorem]{Lemma}
\newtheorem{corollary}[theorem]{Corollary}

\newtheorem{construction}[theorem]{Construction}

\theoremstyle{definition}
\newtheorem{definition}[theorem]{Definition}
\newtheorem{example}[theorem]{Example}
\newtheorem{remark}[theorem]{Remark}

\newtheorem{assumption}[theorem]{Assumption}

\newcommand{\nc}{\newcommand}
\nc{\on}{\operatorname}

\makeatletter
\newcommand*{\da@rightarrow}{\mathchar"0\hexnumber@\symAMSa 4B }
\newcommand*{\da@leftarrow}{\mathchar"0\hexnumber@\symAMSa 4C }
\newcommand*{\xdashrightarrow}[2][]{%
	\mathrel{%
		\mathpalette{\da@xarrow{#1}{#2}{}\da@rightarrow{\,}{}}{}%
	}%
}
\newcommand{\xdashleftarrow}[2][]{%
	\mathrel{%
		\mathpalette{\da@xarrow{#1}{#2}\da@leftarrow{}{}{\,}}{}%
	}%
}
\newcommand*{\da@xarrow}[7]{%
	\sbox0{$\ifx#7\scriptstyle\scriptscriptstyle\else\scriptstyle\fi#5#1#6\m@th$}%
	\sbox2{$\ifx#7\scriptstyle\scriptscriptstyle\else\scriptstyle\fi#5#2#6\m@th$}%
	\sbox4{$#7\dabar@\m@th$}%
	\dimen@=\wd0 %
	\ifdim\wd2 >\dimen@
	\dimen@=\wd2 %
	\fi
	\count@=2 %
	\def\da@bars{\dabar@\dabar@}%
	\@whiledim\count@\wd4<\dimen@\do{%
		\advance\count@\@ne
		\expandafter\def\expandafter\da@bars\expandafter{%
			\da@bars
			\dabar@ 
		}%
	}%
	\mathrel{#3}%
	\mathrel{%
		\mathop{\da@bars}\limits
		\ifx\\#1\\%
		\else
		_{\copy0}%
		\fi
		\ifx\\#2\\%
		\else
		^{\copy2}%
		\fi
	}%
	\mathrel{#4}%
}
\makeatother

\nc{\Q}{\mathbb{Q}}
\nc{\Z}{\mathbb{Z}}
\nc{\cl}{\mathrm{cl}}

\nc{\fraka}{{\mathfrak a}} \nc{\bba}{{\mathbf a}}
\nc{\frakb}{{\mathfrak b}}
\nc{\frakc}{{\mathfrak c}}
\nc{\frakd}{{\mathfrak d}}
\nc{\frake}{{\mathfrak e}}
\nc{\frakf}{{\mathfrak f}}
\nc{\frakg}{{\mathfrak g}}
\nc{\frakh}{{\mathfrak h}}
\nc{\fraki}{{\mathfrak i}}
\nc{\frakj}{{\mathfrak j}}
\nc{\frakk}{{\mathfrak k}}
\nc{\frakl}{{\mathfrak l}}
\nc{\frakm}{{\mathfrak m}}
\nc{\frakn}{{\mathfrak n}}
\nc{\frako}{{\mathfrak o}}
\nc{\frakp}{{\mathfrak p}}
\nc{\frakq}{{\mathfrak q}}
\nc{\frakr}{{\mathfrak r}}
\nc{\fraks}{{\mathfrak s}}
\nc{\frakt}{{\mathfrak t}}
\nc{\fraku}{{\mathfrak u}}
\nc{\frakv}{{\mathfrak v}}
\nc{\frakw}{{\mathfrak w}}
\nc{\frakx}{{\mathfrak x}}
\nc{\fraky}{{\mathfrak y}}
\nc{\frakz}{{\mathfrak z}}
\nc{\frakA}{{\mathfrak A}}
\nc{\frakB}{{\mathfrak B}}
\nc{\frakC}{{\mathfrak C}}
\nc{\frakD}{{\mathfrak D}}
\nc{\frakE}{{\mathfrak E}}
\nc{\frakF}{{\mathfrak F}}
\nc{\frakG}{{\mathfrak G}}
\nc{\frakH}{{\mathfrak H}}
\nc{\frakI}{{\mathfrak I}}
\nc{\frakJ}{{\mathfrak J}}
\nc{\frakK}{{\mathfrak K}}
\nc{\frakL}{{\mathfrak L}}
\nc{\frakM}{{\mathfrak M}}
\nc{\frakN}{{\mathfrak N}}
\nc{\frakO}{{\mathfrak O}}
\nc{\frakP}{{\mathfrak P}}
\nc{\frakQ}{{\mathfrak Q}}
\nc{\frakR}{{\mathfrak R}}
\nc{\frakS}{{\mathfrak S}}
\nc{\frakT}{{\mathfrak T}}
\nc{\frakU}{{\mathfrak U}}
\nc{\frakV}{{\mathfrak V}}
\nc{\frakW}{{\mathfrak W}}
\nc{\frakX}{{\mathfrak X}}
\nc{\frakY}{{\mathfrak Y}}
\nc{\frakZ}{{\mathfrak Z}}
\nc{\bbA}{{\mathbb A}}
\nc{\bbB}{{\mathbb B}}
\nc{\bbC}{{\mathbb C}}
\nc{\bbD}{{\mathbb D}}
\nc{\bbE}{{\mathbb E}}
\nc{\bbF}{{\mathbb F}} \nc{\bbf}{{\mathbf f}}
\nc{\bbG}{{\mathbb G}}
\nc{\bbH}{{\mathbb H}}
\nc{\bbI}{{\mathbb I}}
\nc{\bbJ}{{\mathbb J}}
\nc{\bbK}{{\mathbb K}}
\nc{\bbL}{{\mathbb L}}
\nc{\bbM}{{\mathbb M}}
\nc{\bbN}{{\mathbb N}}
\nc{\bbO}{{\mathbb O}}
\nc{\bbP}{{\mathbb P}}
\nc{\bbQ}{{\mathbb Q}}
\nc{\bbR}{{\mathbb R}}
\nc{\bbS}{{\mathbb S}}
\nc{\bbT}{{\mathbb T}}
\nc{\bbU}{{\mathbb U}}
\nc{\bbV}{{\mathbb V}}
\nc{\bbW}{{\mathbb W}}
\nc{\bbX}{{\mathbb X}}
\nc{\bbY}{{\mathbb Y}}
\nc{\bbZ}{{\mathbb Z}}
\nc{\calA}{{\mathcal A}}
\nc{\calB}{{\mathcal B}}
\nc{\calC}{{\mathcal C}}
\nc{\calD}{{\mathcal D}}
\nc{\calE}{{\mathcal E}}
\nc{\calF}{{\mathcal F}}
\nc{\calG}{{\mathcal G}}
\nc{\calH}{{\mathcal H}}
\nc{\calI}{{\mathcal I}}
\nc{\calJ}{{\mathcal J}}
\nc{\calK}{{\mathcal K}}
\nc{\calL}{{\mathcal L}}
\nc{\calM}{{\mathcal M}}
\nc{\calN}{{\mathcal N}}
\nc{\calO}{{\mathcal O}}
\nc{\calP}{{\mathcal P}}
\nc{\calQ}{{\mathcal Q}}
\nc{\calR}{{\mathcal R}}
\nc{\calS}{{\mathcal S}}
\nc{\calT}{{\mathcal T}}
\nc{\calU}{{\mathcal U}}
\nc{\calV}{{\mathcal V}}
\nc{\calW}{{\mathcal W}}
\nc{\calX}{{\mathcal X}}
\nc{\calY}{{\mathcal Y}}
\nc{\calZ}{{\mathcal Z}}

\nc{\scrA}{{\mathscr A}}
\nc{\scrB}{{\mathscr B}}
\nc{\scrC}{{\mathscr C}}
\nc{\scrD}{{\mathscr{D} }}
\nc{\scrE}{{\mathscr E}}
\nc{\scrF}{{\mathscr F}}
\nc{\scrG}{{\mathscr G}}
\nc{\scrH}{{\mathscr H}}
\nc{\scrI}{{\mathscr J}}
\nc{\scrJ}{{\mathscr I}}
\nc{\scrK}{{\mathscr K}}
\nc{\scrL}{{\mathscr L}}
\nc{\scrM}{{\mathscr M}}
\nc{\scrN}{{\mathscr N}}
\nc{\scrO}{{\mathscr O}}
\nc{\scrP}{{\mathscr P}}
\nc{\scrQ}{{\mathscr Q}}
\nc{\scrR}{{\mathscr R}}
\nc{\scrZ}{{\mathscr{Z}}}
\nc{\sfF}{{\mathsf{F}}}
\nc{\sfT}{{\mathsf{T}}}
\nc{\sfw}{{\mathsf{w}}}

\nc{\bB}{{\breve{B}}}
\nc{\bcB}{{\breve{\calB}}}
\nc{\bE}{{\breve{E}}}
\nc{\bF}{{\breve{F}}}
\nc{\bG}{{\breve{G}}}
\nc{\bcG}{{\breve{\calG}}}
\nc{\bO}{{\breve{O}}}
\nc{\bT}{{\breve{T}}}
\nc{\bcT}{{\breve{\calT}}}
\nc{\bU}{{\breve{U}}}
\nc{\bcU}{{\breve{\calU}}}
\nc{\bOE}{{\breve{O}_E}}
\nc{\bOF}{{\breve{O}_F}}
\nc{\OF}{{O_F}}
\nc{\kE}{{k_E}}
\nc{\kF}{{k_F}}
\nc{\piE}{{\pi_E}}
\nc{\Ebar}{{\overline{E}}}

\nc{\D}{{\on{D}}}
\nc{\F}{{\on{F}}}

\nc{\bbFbar}{{\overline{\bbF}}}
\nc{\bbFpbar}{{\overline{\bbF}_p}}
\nc{\Fpbar}{{\overline{\bbF}_p}}
\nc{\bbQbar}{{\overline{\bbQ}}}
\nc{\ibar}{{\overline{i}}}
\nc{\jbar}{{\overline{j}}}
\nc{\kbar}{{\overline{k}}}
\nc{\kEbar}{{\overline{k}_E}}
\nc{\kFbar}{{\overline{k}_F}}
\nc{\Sbar}{{\overline{S}}}
\nc{\etab}{{\breve{\eta}}}
\nc{\etabar}{{\overline{\eta}}}
\nc{\TE}{{\prod_{i=1}^m\Spd \bE_i}}

\nc{\T}{{\breve{T}}}
\nc{\B}{{\breve{B}}}
\nc{\G}{{\breve{G}}}
\nc{\U}{{\breve{U}}}
\nc{\N}{{\breve{N}}}

\nc{\ad}{{\on{ad}}}
\nc{\aff}{{\on{af}}}
\nc{\Bun}{{\on{Bun}}}
\nc{\can}{{\on{can}}}
\nc{\cInd}{{\on{cInd}}}
\nc{\coev}{{\on{coev}}}
\nc{\cosupp}{{\on{cosupp}}}
\nc{\CHT}{{\on{CH}_{\Lambda}(\mathsf{T})}}
\nc{\CT}{{\on{CT}}}
\nc{\Def}{{\on{Def}}}
\nc{\Div}{{\on{Div}}}
\nc{\DX}{{\on{Div}^1_X}}
\nc{\DY}{{\on{Div}^1_{\mathcal{Y}}}}
\newcommand{\Dwchi}{\Delta_{w, \chi}}
\nc{\emb}{{\on{emb}}}
\nc{\et}{\textup{\'et}}
\nc{\ev}{{\on{ev}}}
\nc{\Fl}{{\mathrm{Fl}}}
\nc{\Fr}{{\on{Fr}}}
\nc{\Frob}{{\mathrm{Frob}}}
\nc{\Gal}{{\on{Gal}}}
\nc{\glob}{{\on{glob}}}
\nc{\Gm}{{\on{\mathbb{G}_m}}}
\nc{\Gr}{{\on{Gr}}}
\nc{\Hk}{{\on{Hk}}}
\nc{\HV}{{H_{\sigma, b, I, \overline{s}_0, V}}}
\nc{\cHV}{{\calH_{\sigma, b, I, \overline{s}_0, V}}}
\nc{\Hom}{{\on{Hom}}}
\nc{\id}{{\on{id}}}
\nc{\Ind}{{\on{Ind}}}
\nc{\Irr}{{\on{Irr}}}
\newcommand{\Iw}{\on{Iw}}
\nc{\lc}{{\on{lc}}}
\nc{\Lam}{{\underline{\Lambda}}}

\nc{\loc}{{\on{loc}}}
\nc{\Perf}{{\on{Perf}}}
\nc{\Perv}{{\on{Perv}}}
\nc{\pr}{{\on{pr}}}
\nc{\pre}{{\on{pre}}}
\nc{\Qlbar}{{\overline{\bbQ}_{\ell}}}
\nc{\rad}{{\on{rad}}}
\nc{\Rank}{{\on{Rank}}}
\nc{\Rep}{{\on{Rep}}}
\nc{\Reg}{{\on{Reg}}}
\nc{\res}{{\on{res}}}
\nc{\sbar}{{\overline{s}_0}}
\nc{\Sat}{{\on{Sat}}}
\nc{\Sht}{{\on{Sht}}}
\nc{\Spa}{{\on{Spa}}}
\nc{\Spd}{{\on{Spd}}}
\nc{\Stab}{{\on{Stab}}}
\nc{\supp}{{\on{supp}}}
\nc{\tw}{{\on{tw}}}

\numberwithin{equation}{section}

\begin{document}
	
	\title[Depth-zero local Langlands via shtukas in mixed characteristic]{Depth-zero local Langlands via shtukas in mixed characteristic}
	
	\author[Chenji Fu]{Chenji Fu}
	
	\begin{abstract}
		Assuming a smoothness property for mixed-characteristic shtuka spaces and 
		hyperbolic localization for twisted sheaves, we determine the inertial restriction
		of the Fargues-Scholze $L$-parameter for irreducible depth-zero representations. Let $E$ be a
		non-archimedean local field of residue characteristic $p$, let $G$ be an unramified connected
		reductive group over $E$, and let $\pi$ be an irreducible smooth depth-zero representation of
		$G(E)$ over $\overline{\bbF}_{\ell}$ where $\ell$ is a prime number which does not divide $p|\pi_0Z(G)|$. Let $\varphi_{\pi}^{\on{FS}}$ be the Fargues-Scholze $L$-parameter associated to $\pi$. We describe the restriction of $\varphi_{\pi}^{\on{FS}}$ to the inertia subgroup in terms of Deligne-Lusztig theory. We prove this by computing the inertia action on the cohomology of shtuka spaces with depth-zero level structure. As an application, we prove that  the Fargues--Scholze $L$-parameters agree with DeBacker--Reeder $L$-parameters over $\overline{\bbQ}_{\ell}$.
	\end{abstract}	
	
	\maketitle

	\setcounter{tocdepth}{2}
	\tableofcontents

	\section{Introduction}
	
	The local Langlands conjectures predict a map from irreducible smooth representations of a $p$-adic group to $L$-parameters. Many cases have been constructed by
	different methods. Using the cohomology of shtuka spaces, the paper \cite{fargues2021geometrization} constructed such a map for every connected reductive groups and every irreducible smooth representation. One goal of our paper is to compute the corresponding (semisimplified) Fargues-Scholze $L$-parameter $\varphi_{\pi}^{\on{FS}}$ associated with a depth-zero irreducible representation $\pi$ for unramified groups. More precisely, we will show that $\varphi_{\pi}^{\on{FS}}$ is trivial on the wild inertia, and we will describe the restriction of $\varphi_{\pi}^{\on{FS}}$ to the inertia subgroup in terms of Deligne-Lusztig theory (Theorem \ref{Thm_L_param_intro}). Another related goal is to control the cohomology of mixed-characteristic shtuka spaces with depth-zero level structure. Inspired by the paper \cite{eteve2023monodromiques}, assuming the smoothness property in Assumption \ref{Assum_sm}, we reduce the question to control the inertia action on twisted central sheaves on the local Hecke stack for pro-$p$ Iwahori level. Then, following \cite{anschutz2023gaitsgory} and assuming the properties in Assumption \ref{Assum_CT}, we use twisted hyperbolic localization to reduce to the torus case, where the inertia action can be
	computed directly. Our method might shed light on cohomology of locally symmetric spaces with deeper level structures.

	Let $p$ be a prime number. Let $E$ be a non-archimedean local field with residue field $k_E \cong \bbF_q$, where $q=p^r$ is a power of $p$. Let $O_E$ be the ring of integers of $E$.  We fix a uniformizer $\pi_E$ of $O_E$. Let $\bOE$ be the completion of the maximal unramified extension of $O_E$. 
	
	We will mainly work in the world of diamonds and $v$-stacks created by Scholze \cite{scholze2017ecd, scholze2020berkeley}. Let $\ell$ be a prime number which does not divide $p|\pi_0Z(G)|$. Let $\Lambda = \overline{\bbF}_{\ell}$ be the coefficients of sheaves. For $X$ a small $v$-stack, we will write
	$$\scrD_{\et}(X, \Lambda)$$
	for the infinity-category of \'etale sheaves on $X$ with $\Lambda$-coefficients. Sometimes we will omit the coefficient $\Lambda$ and simply denote it by $\scrD_{\et}(X)$.
	
	Let $\breve{S}_0$ denote the $v$-stack $\Spd \bOE$. The $v$-stack $\breve{S}_0$ is also the moduli space of degree $1$ closed Cartier divisors on the curve $\calY_S$ (\cite[Proposition VI.1.2]{fargues2021geometrization}). The $v$-stack $\breve{S}_0$ has two points, the generic point $\breve{\eta}=\Spd \bE$ and the closed point $\overline{s}_0=\Spd \overline{k}_E$. Choose a geometric point $\overline{\eta}$ of $\breve{\eta}$.
	
	Let $\calQ$ be a smooth affine group scheme over $\bOE$. For $I$ a finite set, let
	$$\Sht_{\calQ, (\DY)^I}=\Sht_{\calQ, (\Spd \bOE)^I} \to (\Spd \bOE)^I$$
	be the stack of $\calQ$-shtukas over $(\Spd \bOE)^I$ (see Definition \ref{Def_Sht}). Let 
	$$\Sht_{\calQ, (\DY)^I, \sbar} \to (\Spd \bOE - \sbar)^I=(\Spd \bE)^I$$
	be the stack of $\calQ$-shtukas with $I$ legs and level structure $\sbar$ (see Definition \ref{Def_Sht_depth0}).

	Let $\calG$ be a connected reductive group over $O_E$. We denote $G:=\calG_E$ its generic fiber. Let $\sigma$ be a polysimplex in the Bruhat-Tits building $\calB(G, E)$. Let $\calG_{\sigma}$ be the parahoric group scheme over $O_E$ corresponding to $\sigma$. We have the associated shtuka space $\Sht_{\calG_{\sigma, \bOE}}$, which we simply denote by $\Sht_{\calG_{\sigma}}$.

	Let $N$ be either the empty set or $\sbar$. We use the convention that
	$$\Sht_{\calG_{\sigma}, (\DY)^I, N}=\Sht_{\calG_{\sigma}, (\DY)^I}$$
	when $N$ is the empty set.
	For every $\sigma, I$, we have the leg map
	$$p_{\sigma, I, N}: \Sht_{\calG_{\sigma}, (\DY)^I, N} \to (\Spd \bOE-N)^I$$
	and the map
	$$\epsilon_{\sigma, I, N}: \Sht_{\calG_{\sigma}, (\DY)^I, N} \to \Hk^{\loc}_{\calG_{\sigma}, (\DY)^I} \times_{(\DY)^I} (\Spd \bOE-N)^I$$
	recording the local modification, where 
	$$\Hk^{\loc}_{\calG_{\sigma}, (\DY)^I}$$
	is the local Hecke stack for $\calG_{\sigma}$ with $I$ legs (see Subsection \ref{Subsection_Sht}).
	
	Before continuing, we state our assumptions.
	
	\begin{assumption}\label{Assum} 
		\begin{enumerate}
			\item For  some variant of the map 
			$$\epsilon: \Sht_{\calQ, (\Spd \bOE)^I} \to \Hk^{\loc}_{\calQ, (\Spd \bOE)^I},$$
			we assume that $\epsilon^*$ commutes with iterated nearby cycles over $(\Spd \bOE)^I$. See Assumption \ref{Assum_sm} for details.
			\item We assume that there is a good theory of hyperbolic localizations for twisted sheaves. See Assumption \ref{Assum_CT} for details. 
		\end{enumerate}
	\end{assumption}

	\begin{definition}
		\begin{enumerate}
			\item Choose $\sigma, I$ as before. Let $N:=\sbar$. For $V \in \Rep(\widehat{G}^I)$, the \textit{cohomology of the shtuka space with depth-zero level structure with respect to $V$} is 
			$$\calH_{\sigma, I, \overline{s}_0, V}:=p_{\sigma, I, \sbar, !} \epsilon_{\sigma, I, \sbar}^*\Sat V \in \scrD_{\et}((\Spd \bE)^I, \Lambda),$$
			where $\Sat V \in \scrD_{\et}(\Hk^{\loc}_{\calG_{\sigma}, \breve{\eta}^I}, \Lambda)$ is the Satake sheaf.
			\item We write
			$$H_{\sigma, I, \overline{s}_0, V}:=(\calH_{\sigma, I, \overline{s}_0, V})|_{\Delta(\overline{\eta})} \in \scrD_{\et}(\overline{\eta}, \Lambda)$$
			for the fiber of $\calH_{\sigma, I, \overline{s}_0, V}$ at the generic geometric point $\overline{\eta}$.
		\end{enumerate}
	\end{definition}
	
	The complex $H_{\sigma, I, \overline{s}_0, V}$ carries an action of $(I_E)^I$. The formalism of excursion endows $\calH_{\sigma, I, \overline{s}_0, V}$ with an action of $\calO(Z^1(W_E, \widehat{G})/\widehat{G})$, the ring of functions on the stack of $L$-parameters (Lemma \ref{Lemma_Sht_Exc}).

	We can also choose an element $b$ in the Kottwitz set $B(G)$ and consider the versions of cohomology of shtuka space with fixed $b$, denoted by $\cHV$, etc (see Subsection \ref{Subsection_Sht}). By changing the level structure, there is a $\calG_{\sigma}(k_E)$-action  on $\Sht_{\calG_{\sigma}, (\DY)^I, \sbar}$, hence on $\cHV$ and $\HV$. Let $V_{\sigma}(k_E)$ be the unipotent radical of $\calG_{\sigma}(k_E)$ and $M_{\sigma}(k_E)$ the reductive quotient. We consider the invariants
	$(\HV)^{V_{\sigma}(k_E)}$, which then carries an action of $M_{\sigma}(k_E)$.
	
	\begin{theorem}
		Assume Assumption \ref{Assum_sm} and \ref{Assum_CT}. For every $I, V, \sigma, b$, the $(I_E)^I$-module $(\HV)^{V_{\sigma}(k_E)}$ is tame, i.e., the $(I_E)^I$-action is trivial on the wild inertia $(P_E)^I$.
	\end{theorem}
	
	Moreover, we can describe the $(I_E/P_E)^I$-action on $(\HV)^{V_{\sigma}(k_E)}$. Let $T_{M_{\sigma}}$ be a Frobenius-stable maximal torus of $M_{\sigma}$ and $W_{M_{\sigma}}$ the corresponding Weyl group. Recall that from Deligne-Lusztig theory, for every $s \in (\widehat{T_{M_{\sigma}}}\sslash W_{M_{\sigma}})^{\widehat{F}}(\Lambda)$, there is an associated idempotent element $e_s$ in the group algebra $\Lambda[M_{\sigma}(k_E)]$ (see Subsection \ref{Subsection_finite_red} and \ref{Subsection_padic group}). Using the $M_{\sigma}(k_E)$-action on $(\HV)^{V_{\sigma}(k_E)}$, we can define
	$$e_s(\HV)^{V_{\sigma}(k_E)},$$ 
	the summand of cohomology of shtuka with respect to $s$. On the other hand, the element $s \in (\widehat{T_{M_{\sigma}}}\sslash W_{M_{\sigma}})^{\widehat{F}}(\Lambda)$ corresponds to a character $\chi$ of some Frobenius-stable maximal torus $\sfT^{wF}$ of $M_{\sigma}$, for some $w \in W_{M_{\sigma}}$.
	
	A convenient way to encode the data $\chi$ is as follows.  We choose a nice integral model $\calT$ of some maximal torus in $G$ (see Subsection \ref{section-notation}). Let $\sfT:=\calT_{\overline{k}_E}$ be the geometric special fiber of $\calT$. We define
	\begin{equation}
		\on{CH}_{\Lambda}(\sfT):=\Hom(\pi_1^{\on{t}}(\sfT)_{\ell'}, \Lambda^*)^{\on{tors}},
	\end{equation}
	the group of finite-order characters of the prime-to-$\ell$ tame \'etale fundamental group of $\sfT$ (Definition \ref{Def_CH}). 
	
	We can view $\chi$ as a rank-one tame \'etale local system on $\sfT$ (denoted $\calL_{\chi}$) via pullback along the (variant of) Lang's map
	\begin{equation}
		\sfT \to [\sfT/_{Ad_{wF}} \sfT] \cong [*/\sfT^{wF}].
	\end{equation}
	
	Note that we have canonical isomorphisms
	\begin{equation}
		I_E/P_E \cong \widehat{\bbZ}^{(p)}(1) \cong I_{k_E((t))}/P_{k_E((t))}.
	\end{equation}
	We also have a map
	\begin{equation}
		a: [*/I_{k_E((t))}] \cong \Spd \overline{k}_E((t)) \to \bbG_{m, \overline{k}_E}=\Spd \overline{k}_E[u, u^{-1}]
	\end{equation}
	sending $u$ to $t$.
	Thus, for every cocharacter $\nu: \bbG_m \to \sfT$, we can view $\calL_{\chi}$ as a character of $(I_E/P_E)$ via pullback along $\nu$ and $a$.
	
	\begin{theorem}[See Theorem \ref{Thm_tame}]\label{Thm_1.4}
		Assume Assumption \ref{Assum_sm} and \ref{Assum_CT}.
		Let $\sigma$ be a polysimplex in the Bruhat-Tits building of $G$ over $E$. Let $I$ be a finite set, $V \in \Rep(\widehat{G}^I)$, and $b \in B(G)$. Let $s \in (\widehat{T_{M_{\sigma}}}\sslash W_{M_{\sigma}})^{\widehat{F}}(\Lambda)$ with corresponding character $\chi \in \on{CH}_{\Lambda}(\sfT)$. Then there is a filtration on the complex
		\begin{equation}
			A:=e_s(\HV)^{V_{\sigma}(k_E)}
		\end{equation}
		indexed by $X_*(T)$ such that
		\begin{enumerate}
			\item For every $\mu \in X_*(T)$, the $\mu$-th graded piece decomposes into a direct sum
			\begin{equation}
				\on{gr}_{\mu}(A)=\bigoplus_{\nu}A_{\nu},
			\end{equation}
			where $\nu=(\nu_i)_{i \in I}$ runs over elements in $X_*(T)^I$ such that $\sum_i \nu_i=\mu.$
			\item The $I_E^I$-action on $A_{\nu}$ is tame, i.e., trivial on $P_E^I$. The $(I_E/P_E)^I$-action on $A_{\nu}$ is given by $\boxtimes_{i \in I}a^*\nu_i^*\calL_{\chi}$.
		\end{enumerate}
	\end{theorem}
	
	The above theorem has the following consequences. We refer the reader to Subsection \ref{Subsection_statement} for precise statements.
	
	\begin{theorem}[Theorem \ref{Theorem_action}]
		Assume Assumption \ref{Assum_sm} and \ref{Assum_CT}.
		\begin{enumerate}
			\item For every $I, V, \sigma, b$, the $(I_E)^I$-module $(\HV)^{V_{\sigma}(k_E)}$ is tame, i.e., the $(I_E)^I$-action is trivial on the wild inertia $(P_E)^I$.
			\item For every $s \in (\widehat{T}_{M_{\sigma}}\sslash W_{M_{\sigma}})^{\widehat{F}}(\Lambda)$, the $\calO(Z^1(W_E, \widehat{G})/\widehat{G})$-module 
			$$e_s(\HV)^{V_{\sigma}(k_E)}$$ 
			is supported on $\on{ev}_{\tau_E}^{-1}(s)$, where 
			\begin{equation}
				\on{ev}_{\tau_E}: (Z^1(W_E/P_E, \widehat{G})\sslash \widehat{G})(\Lambda) \to (\widehat{G}\sslash \widehat{G})^{\widehat{F}}(\Lambda) \cong (\widehat{T}\sslash W)^{\widehat{F}}(\Lambda)
			\end{equation}
			is the evaluation map on a topological generator $\tau_E \in I_E/P_E \cong \widehat{\bbZ}^{(p)}(1)$. See Remark \ref{Rmk_depend} for the independence of the choice $\tau_E$.
		\end{enumerate}
	\end{theorem}
	
	\begin{theorem}[{Theorem \ref{Thm_LLC}}]\label{Thm_L_param_intro}
		Assume Assumption \ref{Assum_sm} and \ref{Assum_CT}. Let $\pi \in \Irr_{\Lambda}^0G(E)$ be an irreducible depth-zero representation.
		\begin{enumerate}
			\item The corresponding $L$-parameter $\varphi_{\pi}^{\on{FS}}$ is tame. In particular, we can view $\varphi_{\pi}^{\on{FS}}$ as a map from $W_E/P_E$.
			\item The restriction of $\varphi_{\pi}^{\on{FS}}$ to $I_E/P_E$ is determined by its value on the topological generator $\tau_E \in I_E/P_E$, which is computed by the following commutative diagram
			\[\begin{tikzcd}
				{\Irr_{\Lambda}^0G(E)} & {(Z^1(W_E/P_E, \widehat{G})\sslash \widehat{G})(\Lambda)} \\
				{(\widehat{T}\sslash W)^{\widehat{F}}(\Lambda)} & {(\widehat{T}\sslash W)^{\widehat{F}}(\Lambda)}
				\arrow["{\on{LLC}^{\on{FS}}}", from=1-1, to=1-2]
				\arrow["{\on{LS}}"', from=1-1, to=2-1]
				\arrow["{\on{ev}_{\tau_E}}", from=1-2, to=2-2]
				\arrow["{=}", from=2-1, to=2-2].
			\end{tikzcd}\]
		\end{enumerate}
	\end{theorem}

	As an application, we show in Section \ref{Section_DR} that for depth-zero regular supercuspidal representations, Fargues--Scholze $L$-parameters agree with DeBacker--Reeder $L$-parameters over $\overline{\bbQ}_{\ell}$.
	
	\begin{theorem}[Corollary \ref{Cor_FSDR}]
		Assuming Assumption \ref{Assum_sm} and \ref{Assum_CT}. Then the Fargues--Scholze and DeBacker--Reeder $L$-parameters agree:
		\begin{equation}
			\varphi_{\pi}^{\on{FS}} = \varphi_{\pi}^{\on{DR}}.
		\end{equation}
	\end{theorem}

	\subsection{Reduction from shtukas to the
		local Hecke stack}
	From now on, let us outline the proof of Theorem \ref{Thm_1.4}.
	There are two basic ideas: 
	\begin{enumerate}
		\item Reduce from the shtuka space $\Sht$ to the local Hecke stack $\Hk^{\loc}$.
		\item Reduce from a general unramified reductive group $G$ to the torus case.
	\end{enumerate}
	
	In this subsection, we sketch the reduction step from $\Sht$ to $\Hk^{\loc}$ following \cite[Chapter 4]{eteve2023monodromiques}. We refer the reader to Subsection \ref{Subsection_strategy} for details.
	
	Recall that we want to compute the inertia action on the cohomology
	$$e_s(\HV)^{V_{\sigma}(k_E)}.$$
	
	By Deligne-Lusztig theory, it suffices to compute
	$${^*}R_{w}(\cHV)^{V_{\sigma}(k_E)},$$
	for every $w \in W_{M_{\sigma}}$. Here ${^*}R_{w}$ is the Deligne-Lusztig restriction.
	
	It turns out that ${^*}R_{w}(\cHV)^{V_{\sigma}(k_E)}$ can be interpreted as certain cohomology of 
	$$\Sht_{\calG_{\sigma}, (\DY)^I, \overline{s}_0} \times^{\calG_{\sigma}(k_E)} M_{\sigma}/B_{M_{\sigma}},$$
	see Lemma \ref{Lemma_4512}.
	
	Now, the key observation (Lemma \ref{Lemma_Sht_Cart}) is that we have a closed embedding
	$$\emb: \Sht_{\calG_{\sigma}, (\DY)^I, \overline{s}_0} \times^{\calG_{\sigma}(k_E)} M_{\sigma}/B_{M_{\sigma}} \to (\Sht_{\calG_C, (\DY)^{I \sqcup 0}})|_{(\Spd \breve{E})^I \times \overline{s}_0},$$
	where $\calG_C$ is obtained from $\calG_{\sigma}$ by dilating the Borel on the special fiber. Let $\calG_C^0$ be the group scheme obtained from $\calG_{\sigma}$ by dilating the unipotent radical of the Borel on the special fiber.
	
	Moreover, we have a diagram
	\[\begin{tikzcd}
		{(\Spd \breve{E})^I} & {\Sht_{\calG_{\sigma}, (\DY)^I, \overline{s}_0} \times^{\calG_{\sigma}(k_E)}M_{\sigma}/B_{M_{\sigma}}} & {\Hk^{\loc}_{G, (\Spd\breve{E})^I} \times \frac{L^+\calG_{C}^0\backslash L^+\calG_{\sigma}/L^+\calG_{C}^0}{Ad_FT_{M_{\sigma}}}} \\
		{(\Spd \bOE)^I} & {\Sht_{\calG_C, (\DY)^{I \sqcup 0}}|_{(\Spd \breve{O}_E)^I \times \overline{s}_0}} & {\frac{\Hk^{\loc}_{\calG_C^0, (\DY)^{I\sqcup 0}}}{Ad_FT_{M_{\sigma}}}|_{{(\Spd \breve{O}_E)^I \times \overline{s}_0}}}
		\arrow[from=1-1, to=2-1]
		\arrow[from=1-2, to=1-1]
		\arrow[from=1-2, to=1-3]
		\arrow["\emb"', from=1-2, to=2-2]
		\arrow[from=1-3, to=2-3]
		\arrow["p"', from=2-2, to=2-1]
		\arrow["\epsilon", from=2-2, to=2-3]
	\end{tikzcd}\]
	
	In other words, we have extended our stacks from $(\Spd \bE)^I$ to $(\Spd \bOE)^I$. We can then use iterated nearby cycles to extract the monodromy as follows.
	
	We are interested in the $(I_E)^I$-action on $p_!\epsilon^*(\Sat(V) \boxtimes \nabla_w(\on{Reg}_{T^{wF}}))$, where the sheaf $\nabla_w(\on{Reg}_{T^{wF}})$ is defined in Equation \eqref{Eqdef_nabla}. This monodromy can be extracted from nearby cycles with respect to  the identity map $(\Spd \bOE)^I \to (\Spd \bOE)^I$ applied to the cohomology sheaves. We will show that $p_!$ commutes with nearby cycles (see Subsection \ref{Subsection_push_nearby} for a sketch), so it is enough to control the monodromy of shtuka spaces. Furthermore, by Assumption \ref{Assum}, $\epsilon^*$ commutes with nearby cycles, so it suffices to control the monodromy on the local Hecke stack. On the local Hecke stack, we can compute the monodromy of the corresponding sheaves by reducing to the torus case using twisted hyperbolic localization (see Subsection \ref{Subsection_computation_Hecke} for a sketch).

	\subsection{Commutation of nearby cycles with pushforward}\label{Subsection_push_nearby}
	
	As explained in the previous subsection, one key step is to show that nearby cycles commute with $p_!$.
	In this subsection, we sketch the proof.
	
	Recall the following general situation. Let $S$ be curve with a closed point $s$. Let $I=\{1, 2, ..., n\}$ be a finite set. Given a morphism $f: Z \to W$  of small $v$-stacks over $S^I$ and $A \in \scrD_{\et}(Z|_{(S-s)^I}; \Lambda)$, we ask whether the natural map
	$$f_!\Psi_1...\Psi_nA \to \Psi_1...\Psi_n f_!A$$
	is an isomorphism.
	
	In our case, $S=\DY=\Spd \bOE, s = \sbar=\Spd \overline{k}_E, S-s=\etab = \Spd \breve{E}$,
	and 
	$$Z=(\Sht_{\calG_C, (\DY)^{I \sqcup 0}})|_{(\Spd \bOE)^I \times \sbar}:=\Sht_{\calG_C, (\DY)^{I \sqcup 0}} \times_{(\Spd \bOE)^{I \sqcup 0}}{((\Spd \bOE)^I \times \sbar)}$$
	with the structure map $p$ to $S^I$ being the projection to the first $I$-legs. In other words, we have
	$$p: Z=(\Sht_{\calG_C, (\DY)^{I \sqcup 0}})|_{(\Spd \bOE)^I \times \sbar} \to (\Spd \bOE)^I \times \sbar \to (\Spd \bOE)^I.$$	
	Let $W=S^I$ with the structure map being the identity map. Let $f:=p: Z \to W$. Let 
	$$A:=\emb_!\epsilon^*(\Sat V \boxtimes \nabla_w(\on{Reg}_{T^{wF}})) \in \scrD_{\et}(Z|_{\etab^I}; \Lambda).$$
	We can ask whether the natural map
	\begin{equation}\label{Eq_commute?}
		p_!\Psi_1...\Psi_n A  \to \Psi_1...\Psi_n p_! A
	\end{equation}
	is an isomorphism.
	
	Following \cite{salmon2023unipotent} and \cite{eteve2023monodromiques}, we are going to construct an inverse of the map \eqref{Eq_commute?} using the idea of fusion. Write $\Psi_I=\Psi_1...\Psi_n$ and introduce three disjoint finite sets $I_1, I_2, I_3$ which are bijective to $I$. Formally, we will construct maps
	$$\Psi_{I_1}p_! \to \Psi_{I_1}p_!\Psi_{I_2} \Psi_{I_3} \to \Psi_{I_1}\Psi_{I_2}p_! \Psi_{I_3} \to p_!\Psi_{I_3}.$$
	(see Subsection \ref{Subsection_p_proof} for details.)
	Here, the first map can be constructed using fusion on the local Hecke stack, which translates to the monoidality of twisted central sheaves. The second map is induced from the natural map $p_!\Psi_{I_2} \to \Psi_{I_2}p_!$. The third map can be constructed once we know certain cohomology of shtuka space is locally constant, which ultimately reduce to \cite[Corollary IX.2.3]{fargues2021geometrization}.

	\subsection{Computation on the local Hecke stack}\label{Subsection_computation_Hecke}
	
	As explained, using $p_!\Psi=\Psi p_!$ and our assumption $\epsilon^*\Psi=\Psi \epsilon^*$, we reduce the computation to the monodromy (i.e., $(I_E)^I$-action) on the so-called twisted central sheaf
	$$\Psi(\Sat V \boxtimes \nabla_w(\chi)),$$
	where $\Psi=\Psi_h$ is the nearby cycle associated to the map
	$$h: \Hk^{\loc}_{\calG_C^0, \breve{S}_0^I \times \sbar} :=\Hk^{\loc}_{\calG_C^0, (\DY)^{I \sqcup 0}} \times_{(\DY)^{I \sqcup 0}} ((\Spd \breve{O}_E)^I \times \Spd \overline{k}_E) \to (\Spd \breve{O}_E)^I=\breve{S}_0^I .$$
	
	Recall that we have fixed a reductive group $\calG$ over $O_E$ and we denote by $G=\calG_E$ its generic fiber. We fix a maximally $O_E$-split maximal torus $\calT \subseteq \calG$ with generic fiber $T$. This choice determines an apartment $\mathcal{A}(G, T) \subseteq \calB(G, E)$ and the integral model $\calG$ determines a hyperspecial point $x_0 \in \mathcal{A}(G, T)$. Let $C$ be an alcove in $\calA(G, T)$ containing $x_0$ in its closure and let $\sigma$ be a polysimplex in the closure of $C$. The choices of $C$ and $\sigma$ determine three smooth group schemes over $O_E$.
	\begin{enumerate}
		\item $\calG_{\sigma}$ the parahoric group scheme corresponding to $\sigma$.
		\item $\calG_{C}$ the Iwahori group scheme corresponding to $C$. Its group of $O_E$-points is an Iwahori subgroup of $G(E)$.
		\item $\calG_{C}^0$ the pro-$p$ Iwahori group scheme corresponding to $C$. Its group of $O_E$-points is the pro-unipotent radical of the Iwahori $\calG_{C}(O_E)$.
	\end{enumerate}
	
	Moreover, let $\sfT:=\calT_{\overline{k}_E}$ denote the geometric speical fiber of $\calT$. Recall that we have
	$$\on{CH}_{\Lambda}(\sfT)=\Hom(\pi_1^t(\sfT)_{\ell'}, \Lambda^*)^{\on{tors}}.$$
	For each $\chi \in \on{CH}_{\Lambda}(\sfT)$, there is a local system $\calL_{\chi}$ on $\sfT$, called the Kummer sheaf.

	The idea is to use the constant term functor, a special case of hyperbolic localization, to reduce to the torus case. The paper \cite{anschutz2023gaitsgory} considers the constant term functor for $\Hk^{\loc}_{\calG_C}$. Using the theory of $(\sfT, \calL_{\chi})$-equivariant sheaves reviewed in the next paragraph, we can similarly define a constant-term functor for $\Hk^{\loc}_{\calG_C^0}$ with respect to a fixed 
	$\chi \in \on{CH}_{\Lambda}(\mathsf{T})$. 
	
	In this paragraph, we review the notion of $(\sfT, \calL_{\chi})$-equivariant sheaves following \cite[Section 2]{lusztig2020endoscopy}. Let $H$ be a connected algebraic group acting on a scheme $X$. We have the action map $\alpha: H \times X \to X$. Let $\calL$ be a character sheaf on $H$. We can define the category $$\scrD_{\et}(X/(H, \calL))$$
	of $(H, \calL)$-equivariant sheaves on $X$. An object in $\scrD_{\et}(X/(H, \calL))$ is an object $A \in \scrD_{\et}(X)$ together with an isomorphism
	$$\alpha^*A \cong \calL \boxtimes A$$
	and higher coherence data.
	
	We can now define a functor
	$$\CT_B^{\chi}: \scrD_{\et}((L^+\calG_C, \calL_{\chi})\backslash \Gr_{\calG_C^0, \overline{s}_0}/(\mathsf{T}, \calL_{\chi})) \to \scrD_{\et}((L^+\calT, \calL_{\chi})\backslash (\Gr_{\calG_C^0, \overline{s}_0})^0/(\mathsf{T}, \calL_{\chi})),$$
	see Subsection \ref{Subsection_Constant_terms}.
	The functor $\CT_B^{\chi}$ is the constant term functor over the special fiber $\Spd \overline{k}_E$. We can also define an integral version of $\on{CT}_B^{\chi}$ with several legs. We assume that $\CT_B^{\chi}$ has the formal properties specified in Assumption \ref{Assum_CT}.  In particular, we assume that the functor $\CT_B^{\chi}$ commutes with nearby cycles.

	We will apply the constant term  functor for the opposite Borel, $\CT^{\chi}_{B^-}$, to the twisted central sheaf
	$$\scrZ_{\chi}(V):=\Psi_{h}(\Sat V \boxtimes \calL_{\chi}[\dim \sfT]),$$
	where $B^-$ is the opposite Borel.
	It turns out that we can recover $\scrZ_{\chi}(V)$ from $\CT^{\chi}_{B^-}(\scrZ_{\chi}(V))$ in the graded sense (Theorem \ref{Theorem_central}). For this, we need to develop the theory of twisted Wakimoto sheaves and Wakimoto filtrations. For every $\nu \in X_*(T)$, the twisted Wakimoto sheaf $J_{\nu, \chi}$ is a  certain explicitly defined sheaf that is $(\sfT \times \sfT, \calL_{\chi} \boxtimes \calL_{\chi})$-equivariant (Subsection \ref{Subsection_Wakimoto}). The sheaves $\{J_{\nu, \chi}\}_{\nu \in X_*(T)}$ generate a full subcategory $\on{Wak}_{\chi} \subseteq \scrD_{\et}((L^+\calG_C, \calL_{\chi})\backslash \Gr_{\calG_C^0, \overline{s}_0}/(\mathsf{T}, \calL_{\chi}))$. Using the Bruhat order on $X_*(T)$, we can define the Wakimoto filtration for an object $\calF \in \on{Wak}_{\chi}$. Wakimoto sheaves are defined via pushforward from strata on $I^0\backslash LG/I^0$, and the $(\sfT \times \sfT, \calL_{\chi} \boxtimes \calL_{\chi})$-equivariant sheaves on each stratum are easy to describe. Therefore, an explicit computation shows that the graded piece of $\calF \in \on{Wak}_{\chi}$ is computed by the constant term functor for the opposite Borel (Lemma \ref{Lemma_3.24'}). Moreover, we can show that the twisted central sheaves $\scrZ_{\chi}(V)$ lie in $\on{Wak}_{\chi}$ using the centrality of $\scrZ_{\chi}(V)$, i.e., $\scrZ_{\chi}(V)$ commutes with $J_{\nu, \chi}$ under convolution. Combining everything, we deduce that $\scrZ_{\chi}(V)$ admits a filtration whose
	associated graded is given by $\CT_{B^-}^{\chi}(\scrZ_{\chi}(V))$.  Using our assumption that nearby cycles commute with $\CT_{B^-}^{\chi}$, the complex $\CT_{B^-}^{\chi}(\scrZ_{\chi}(V))$ is essentially the twisted central sheaf for $G=T$ (Theorem \ref{Theorem_central}). Thus, we are reduced to the torus case, where an explicit
	computation completes the proof (see Subsections \ref{Subsection_compute_equal}, \ref{Subsection_compute_multi}, and \ref{Subsection_compute_mixed}).

	\subsection{Organization of the paper} 
	In Section \ref{Section_basic} we review Deligne-Lusztig theory (Subsection \ref{Subsection_finite_red}), depth-zero representations of $p$-adic groups (Subsection \ref{Subsection_padic group}), the shtuka spaces and related spaces (Subsection \ref{Subsection_Sht}). We explain our first assumption in details in Subsection \ref{Subsection_ass_sm}. We state our main theorems in Subsection \ref{Subsection_statement}. We outline the strategy and review the work \cite{eteve2023monodromiques} in Subsection \ref{Subsection_strategy}. 
	
	In Section \ref{Section_pcomm} we prove $p_!\Psi = \Psi p_!$.
	
	In Section \ref{Section_monodromy} we compute the inertia actions on the twisted central sheaves. In Subsection \ref{Subsection_Constant_terms} we introduce hyperbolic localization and constant term for twisted sheaves on the local Hecke stack with pro-$p$ Iwahori level. In Subsection \ref{Subsection_CT_Ass} we explain our second assumption in details. In Subsection \ref{Subsection_Wakimoto} we introduce twisted Wakimoto sheaves. In Subsection \ref{Subsection_CT_Wak} we show that the Wakimoto graded piece can be computed by the constant term functor for the opposite Borel. In Subsection \ref{Subsection_tame} we state our main theorems and sketch a proof assuming the computation in the torus case. In Subsection \ref{Subsection_compute_equal}, \ref{Subsection_compute_multi}, \ref{Subsection_compute_mixed} we compute the torus case explicitly in the equal characteristic case, multi-legs case, and mixed characteristic case, respectively.
	
	In Section \ref{Section_DR}, we prove that for irreducible depth-zero regular supercuspidal representations, the Fargues--Scholze $L$-parameters and DeBacker--Reeder $L$-parameters agree.
	
	\subsection{Related work}
	Our paper is largely inspired by the paper \cite{eteve2023monodromiques}, where Arnaud Eteve proved similar results for Genestier-Lafforgue's construction of $L$-parameters for $E$ a local field of characteristic $p$. The paper \cite{eteve2023monodromiques} is closely related to the paper \cite{xue2020smooth} and \cite{salmon2023unipotent}.
	
	The use of hyperbolic localization to compute nearby cycles is inspired by the paper \cite{anschutz2023gaitsgory}. Similar methods were used earlier in the work of \cite{arkhipov2009perverse}. 
	
	The paper \cite{dat2025depthzero} also proved that the Fargues-Scholze $L$-parameter associated to a depth-zero irreducible representation is tame, using a different method, but does not identify the $L$-parameter.
	
	During the preparation of our work, we learned that Ian Gleason, Xiangqian Yang, and Xinwen Zhu were working on a project which has some overlaps with our work. They also use twisted Wakimoto sheaves  and hyperbolic localizations. Notably, they will prove a version of Assumption \ref{Assum_sm} and \ref{Assum_CT} in their work.
	
	During the preparation of our work, we learned that Sean Cotner and Tony Feng (\cite{cotner2026modular}, \cite{cotner2026local1}, and \cite{cotner2026local2}) were working on a project that obtains Theorem \ref{Thm_L_param_intro}, using a different method. Their method also applies to higher depth representations.
	
	Classically, the local Langlands correspondence for depth-zero regular supercuspidal representations is constructed by \cite{kazhdan2006endoscopic} and \cite{debacker2009depth} for coefficients $\Lambda=\overline{\bbQ}_{\ell} \cong \bbC$. In Section \ref{Section_DR}, we prove that for irreducible depth-zero regular supercuspidal representations, the Fargues--Scholze $L$-parameter and DeBacker--Reeder $L$-parameter agree. 
	
	In the paper \cite{yoshida2010nonabelian}, Yoshida proved that the depth-zero local Langlands for $GL_n$ can be realized using cohomology of Lubin-Tate spaces. Our paper can be viewed as a generalization of this to any unramified connected reductive group.
	
	In \cite{zhu2025tame}, Zhu proves a tame categorical local Langlands conjecture in his setting. Zhu gave a parametrization of depth-zero representations, having the same properties as in Theorem \ref{Thm_L_param_intro}, and proved that his parametrization agrees with DeBacker--Reeder's parametrization. It might be possible to show that Fargues--Scholze's parametrization and Zhu's parametrization agree (\cite[Remark 1.13]{zhu2025tame}). We expect that our article provides a step towards such a comparison.

	\subsection{Acknowledgements}
	
	It is a pleasure to thank my advisors Jessica Fintzen and Peter Scholze for their interest and suggestions regarding this work. I thank Arnaud Eteve for answering numerous questions about the paper \cite{eteve2023monodromiques}. I thank Ian Gleason for enlightening discussions and for explaining to me his unpublished results. Moreover, I thank Johannes Ansch\"utz,  Alexander Bertoloni Meli, Kei Yuen Chan, Sean Cotner, Jean-François Dat, Adam Vincent Daigoro Dauser, Tony Feng, Linus Hamann, David Hansen, Thibaud van den Hove, Alexander Ivanov, Tasho Kaletha,  Tim Kuppel, Jo\~ao Louren\c{c}o,  Lucas Mann, Michael Rapoport, Timo Richarz,  Yiming Tang, Eva Viehmann,  Haining Wang, Zhiyou Wu, Xiangqian Yang, Jize Yu, Mingjia Zhang, Zhiyu Zhang, Xiaoxiang Zhou, Xinwen Zhu, Konrad Zou for discussions and comments on this paper. Additionally, I would like to thank all members of the arithmetic geometry and representation theory group in Bonn for their constant support.

	I am funded by the Hausdorff Center of Mathematics and jointly hosted by the International Max-Planck Research School on Moduli Spaces. I thank both institutes for providing a nice working environment.

	\subsection{AI methodology} After a draft of the paper was mostly written, AI tools were used to proofread. GPT 5.6 found a sign mistake in Lemma \ref{Lemma_section}. All other statements and proofs in the paper originate from its human author.

	\subsection{Notation}
	\label{section-notation}
	
	First, let us introduce schematic notations.
	\begin{itemize}
		\item $p$ a prime.
		\item $E$ a non-archimedean local field, with ring of integers $O_E$ and residue field $k_E$. We fix a uniformizer $\pi_E \in O_E$.
		\item $\overline{E}$ a separable closure of $E$ and $\Gamma:=\Gamma_E:=\mathrm{Gal}(\overline{E}/E)$ the absolute Galois group. We have $P_E \subseteq I_E \subseteq \Gamma_E$ the wild inertia and the inertia subgroup.
		\item $E^{\on{un}}$ the maximal unramified extension of $E$.
		\item $\breve{E}$ the completion of $E^{\on{un}}$,
		$\bOE \subseteq \bE$  its ring of integers, and $\overline{k}_E$ the residue field of $\bOE$.
		\item $\bar{\breve{E}}$ an algebraic closure of $\breve{E}$.

		\item $\calG$ a connected reductive group over $O_E$ (in particular, the special fiber of $\calG$ is connected), with generic fiber $G:=\calG_{E}$ ($G$ is  unramified). Note that the base change $\bcG:=\calG_{\bOE}=\calG \times _{O_E} \bOE$ is split. Let $\bG:=G_{\bE}$. The integral model $\calG$ gives rise to a hyperspecial point $x_0$ in the Bruhat-Tits building  $\calB(G, E)$ of $G$ over $E$.

		\item Let $\calS \subseteq \calG$ be a maximal $O_E$-split torus. Let $S:=\calS_E$, which is a maximal $E$-split torus of $G$.
		
		\item Let $\calT \subseteq \calG$ be the centralizer of $\calS$. Then $\calT$ is a maximal torus of $\calG$. Let $T:=\calT_E$, which is a maximal torus of $G$. Note that the base change $\breve{\calT}:=\calT \times_{O_E}\breve{O}_E$ is a split maximal torus of $\calG_{\bOE}$. Let $\breve{T}:=T_{\bE}$.

		\item $\mathsf{T}:=\bcT_{\overline{k}_E}$  the special fiber of $\bcT$.
		
		\item $\calB \subseteq \calG$  a Borel containing $\calT$. Let $B:=\calB_E$.
		\item $U$ the unipotent radical of $B$.
		\item $B^-\subseteq G$,  the opposite Borel of $B$.
		\item $N:=N_{G}(T)$ denotes the normalizer of $T$ in $G$.
		
		\item $\bbX_\bullet:=\bbX_\bullet(T)$, respectively, $\ \bbX^\bullet:=\bbX^\bullet(T)$, the groups of (geometric) cocharacters, respectively, characters of $T$.
		
		\item $\bbX_\bullet^+,\ \bbX^{\bullet,+}$ denote the dominant cocharacters resp.\ dominant characters of $T$ with respect to $B$.
		
	\end{itemize}

	Next, let us introduce combinatorial notations.
	\begin{itemize}
		\item $\tilde{W}:=N(\bE)/\calT(\bOE)$ the extended affine Weyl group of $\calT$. 
		\item $\calB(G, E)$ the Bruhat--Tits building of $G$ over $E$, $\calB(\breve{G}, \breve{E})$ the Bruhat--Tits building of $\breve{G}$ over $\breve{E}$.
		\item $\mathcal{A}(G,S) \subseteq \calB(G, E)$ the apartment associated with $S$, identified with $\bbX_\bullet(S)_{\bbR}$ for pinned $G$. Note we also have $\calA(\bG, \bT) \subseteq \calB(\bG, \bE)$, and we have $\calA(G, S)=\calA(\bG, \breve{T})^{\Gal(\bE/E)}$.
		\item The integral model $\calG$ of $G$ yields a hyperspecial point $x_0 \in \mathcal{A}(G,S)$. Let $C \subseteq \mathcal{A}(G,S)$ be the alcove which contains $x_0$ and lies in the dominant Weyl chamber with respect to $B$.
		\item $\bbS\subseteq \tilde{W}$ the set of reflections at the walls of $C$, also called the set of simple reflections.
		\item $W_\aff\subseteq \tilde{W}$ the affine Weyl group, which is the Coxeter group generated by the simple reflections.
		\item $\Omega_{C}$ the stabilizer of $C$, which yields an isomorphism
		\begin{equation}
			\tilde{W} \cong W_\aff\rtimes \Omega_{C}.
		\end{equation}
		\item $\ell\colon \tilde{W} \to \bbN_{\geq 0}$ the length function on $\tilde{W}$, i.e., the unique function $\ell(-)\colon \tilde{W} \to \mathbb{N}_{\geq 0}$, which extends the length function $\ell(-)\colon W_\aff\to \mathbb{N}_{\geq 0}$ on the Coxeter group $W_\aff$, such that $\ell(\tau)=0$ for $\tau\in \Omega_{C}$.
		\item $\leq$ is the Bruhat order on $\tilde{W}$, i.e., $w\leq w^\prime$ for $w=(w_\aff,\tau), w^\prime=(w^\prime_\aff,\tau^\prime)\in \tilde{W} \cong W_\aff\rtimes \pi_1(G)_{\breve{\Gamma}}$ if and only if $\tau=\tau^\prime$ and $w_\aff\leq w^\prime_\aff$ for the Bruhat order $\leq$ on $W_\aff$ coming from its Coxeter structure.
		\item $W_{\mathrm{fin}}=N(\breve{E})/T(\breve{E})$  the finite Weyl group, which sits in an exact sequence
		\begin{equation}
			1\to {\bbX}_\bullet\to \tilde{W} \to W_{\mathrm{fin}}\to 1.
		\end{equation}
		\item $t_{\nu}\in \tilde{W}$ is the translation element associated with $\nu\in {\bbX}_\bullet$. 
	\end{itemize}

	Now let us define notations related to affine flag varieties and perfect geometry. 
	\begin{itemize}
		\item Let $(k, O)=(k_E, O_E)$ or $(\overline{k}_E, \bOE)$. We will mostly use the case $(k, O)=(\overline{k}_E, \bOE)$, but sometimes we need to notice that things are already defined over $(k_E, O_E)$.
		\item $\mathrm{Alg}^{\mathrm{perf}}_{k}$ the category of perfect $k$-algebras.
		\item For $R\in \mathrm{Alg}^{\mathrm{perf}}_{k}$, when the characteristic $\on{char} O =0$ we set $W_{O}(R):=O\otimes_{W(k)}W(R)$, the ring of $O$-Witt vectors. When $\on{char} O=p$, we set $W_O(R) = R \hat{\otimes} O \cong R[[\pi_E]]$, the formal power series ring over $R$.
		\item If $X/O$ is an affine scheme of finite type, then $L^+X\colon (\mathrm{Alg}^{\mathrm{perf}}_{k})\to (\mathrm{Sets}),\ R\mapsto X(W_{O}(R))$ is the positive loop functor for $X$.
		\item If $Z/O[1/\pi_E]$ is an affine scheme of finite type, then $LZ\colon (\mathrm{Alg}^{\mathrm{perf}}_{k})\to (\mathrm{Sets}),\ R\mapsto Z(W_{O}(R)[1/\pi_E])$ is the loop functor for $Z$.		
		\item $\calG_{C, O}$ the Iwahori group scheme over $O$ for $G$ associated with the alcove $C$. Note that $\calG_{C, O_E} \times_{O_E} \bOE  \cong \calG_{C, \bOE}$ (\cite[Section 4.1, Axiom UR4]{kaletha2023bruhat}), so we will denote both by $\calG_C$ when no confusion can arise. Keeping in mind the fixed alcove $C$, we will sometimes also denote $\calI:=\calG_C$. There is a natural map $r:\calI \to \calG$, such that the image of $r_k$ is the Borel subgroup $\calB_k$ of the special fiber $\calG_k$. The Iwahori group scheme $\calI$ can be obtained by dilating this Borel.
		\item $\calI^0:=\calG_C^0:=\calG_{C, O}^0$ the pro-unipotent radical of $\calI=\calG_C=\calG_{C, O}$, called the pro-$p$ Iwahori group scheme for $G$ associated with the alcove $C$. 
		\item $I:=L^+\calI$, $I^0:=L^+\calI^0$.
	\end{itemize}

	Next, let us introduce some cohomological notations.
	\begin{itemize}
		\item Let $\ell$ be a prime number which does not divide $p|\pi_0Z(G)|$.
		\item $\Lambda=\overline{\bbF}_{\ell}$.
		\item $\scrD_\et(-):=\scrD_\et(-,\Lambda)$ denotes the (left-completed) $\infty$-derived category of ``\'etale sheaves of $\Lambda$-modules'' on a small v-stack on either perfect schemes or perfectoid spaces of characteristic $p$, compare with \cite[Definition 1.7]{scholze2017ecd}. 
		\item Let $\Hk^{\loc}_{\calI, k}$ be the local Hecke stack over $k$ for the group scheme $\calI$. $\scrD_\et(\Hk^{\loc}_{\calI, k}):=\scrD_\et(\Hk^{\loc}_{\calI, k},\Lambda)^{\mathrm{bd}}$ denotes the $\infty$-category of ``\'etale sheaves of $\Lambda$-modules'' on $\Hk^{\loc}_{\calI, k}$, whose support is a finite subset of the underlying topological space of $\Hk^{\loc}_{\calI, k}$.
		\item Let $\scrD_{\on{cons}}(\Hk^{\loc}_{\calI, k}) \subseteq \scrD_\et(\Hk^{\loc}_{\calI, k})$ be the constructible complexes.
		\item Let $Y$ be a small v-stack with an action of a group object $H$. Let $\pi_{Y}: Y \to Y/H$ be the quotient map. For an object $A \in \scrD_\et(Y)$ which is equivariant with respect to the $H$-action, we denote $\on{Des}_{\pi_Y}(A) \in \scrD_\et(Y/H)$ the unique object (\cite[Section 6.8]{achar2021perverse}) such that $$(\pi_Y)^*\on{Des}_{\pi_Y}(A) \cong A.$$
	\end{itemize}
	
	We collect notation for the ``coherent'' side.

	\begin{itemize}
		\item $\widehat{G}$ denotes the dual group of $G$ over $\Lambda$,
		\item $\widehat{T}\subseteq \widehat{G}$ denotes the dual torus to $T$, and we identify $\bbX^\bullet(\widehat{T})\cong \bbX_\bullet(T)$,
		\item $\widehat{T}\subseteq \widehat{B}\subseteq \widehat{G}$ denotes the Borel subgroup for which the  dominant characters identifying with $\bbX_\bullet(T)^+$.
	\end{itemize}
	
	We fix a trivialization of the group of roots of unity $\overline{k}_E^* \cong (\bbQ/\bbZ)_{p'}$ where $p'$ denotes the prime-to-$p$ part. This induces a trivialization of the Tate twist $(1) \cong \id$.

	\section{Basic notions}\label{Section_basic}
	
	\subsection{Finite reductive groups and Deligne-Lusztig theory}\label{Subsection_finite_red}
	In this subsection, we review certain aspects of Deligne-Lusztig theory. The main references are  \cite[Section 4.2]{eteve2023monodromiques} and \cite{digne2020representations}. Let $M$ be a reductive group over $\overline{\bbF}_q$ equipped with a Frobenius endomorphism $F: M \to M$ coming from some $\bbF_q$-structure. Fix an $F$-stable Borel pair $B_M=TU_M$, and let $W_M=N_M(T)/T$ be the corresponding Weyl group.
	
	For every $w \in W_M$, we can define the Deligne-Lusztig induction and restriction functors
	\begin{equation}
		R_w: \D^b(\Rep_{\Lambda}(T^{wF})) \to \D^b(\Rep_{\Lambda}(M^F)), \quad A \mapsto A \otimes_{T^{wF}}R_w^0,
	\end{equation}
	\begin{equation}
		{^*R}_w: \D^b(\Rep_{\Lambda}(M^F)) \to \D^b(\Rep_{\Lambda}(T^{wF})), \quad B \mapsto \on{RHom}_{M^F}(R_w^0, B),
	\end{equation}
	where $R_w^0$ is the compactly supported cohomology of the Deligne-Lusztig variety associated to $w$.
	
	\begin{theorem}[\cite{deligne1976representations}, \cite{bonnafe03derived}]
		The collection of functors $\{{^*}R_w\}_{w \in W_M}$ is jointly conservative.
	\end{theorem}

	Let $\widehat{T}$ be the torus dual to $T$ defined over $\Lambda$. We recall the construction of Lusztig series. We denote by $\widehat{F}$ the isogeny dual to the Frobenius $F$. Note that we have an isomorphism (depending on the fixed trivialization of roots of unity of $\overline{\bbF}_q$)
	\begin{equation}
		\Lambda[T^{wF}]=\calO(\widehat{T}^{w\widehat{F}})
	\end{equation}
	as both sides are isomorphic to $\Lambda[X_*(T)/(wF-\on{id})X_*(T)]$. We therefore have a bijection
	\begin{equation}
		\Hom(T^{wF}, \Lambda^*)=\widehat{T}^{w\widehat{F}}(\Lambda).
	\end{equation}
	Let $\theta$ be a character of $T^{wF}$ and let $\pi$ be an irreducible representation of $M^F$. We say that $\pi$ lie in the Lusztig series of $(w, \theta)$ if $\pi$ is a subquotient of a cohomology group of $R_w(\theta)$. 
	
	Consider the morphism $\xi: \widehat{T} \to \widehat{T}\sslash W_M$. We still denote by $\widehat{F}$ the morphism induced by $\widehat{F}$ on $\widehat{T}\sslash W_M$.
	
	\begin{theorem}[\cite{deligne1976representations}, \cite{bonnafe03derived}]
		\begin{enumerate}
			\item Let $\pi$ be an irreducible representation of $M^F$. Suppose that $\pi$ belongs to the Lusztig series $(w, \theta)$ and $(w', \theta')$, then $\xi(\theta)=\xi(\theta')$. Hence we have a well-defined map
			\begin{equation}
				\on{LS}:	\Irr_{\Lambda} M^F \to (\widehat{T}\sslash W_M)^{\widehat{F}}(\Lambda).
			\end{equation}
			\item There is a complete collection of central orthogonal idempotents $e_s \in \Lambda[M^F]$ for $s \in (\widehat{T}\sslash W_M)^{\widehat{F}}(\Lambda)$ such that for all $\pi \in \Irr_{\Lambda}M^F$, we have $e_s\pi=\pi$ if and only if $\on{LS}(\pi)=s$.
		\end{enumerate}
	\end{theorem}

	\subsection{Depth-zero representations of $p$-adic groups}\label{Subsection_padic group}
	In this subsection, we recall the definition of depth-zero representations of a $p$-adic group. We also record Lanard's result (Theorem \ref{Thm_lanard}) on the decomposition of the category of depth-zero representations.
	The main references are \cite{lanard18} and \cite[Section 4.3]{eteve2023monodromiques}.
	
	Let $\calH$ be a connected reductive group over $O_E$ with generic fiber $H$. The group $H$ splits over an unramified extension. We introduce the following notations.
	\begin{enumerate}
		\item Let $\calT \subseteq \calH$ be a maximally $O_E$-split maximal torus. We denote its generic fiber by $T \subseteq H$. Let $X^*(T)$ and $X_*(T)$ be the groups of characters and cocharacters defined over $E^{\on{un}}$. The group $\Gal(\overline{k}_E/k_E)$ naturally acts on them.
		\item Write $\calB(H)$ for the reduced Bruhat-Tits building of $H(E)$ and  $\calB(H_{E^{\on{un}}})$ for that of $H(E^{\on{un}})$. Let $A$ be the apartment corresponding to $T$. Let $\tilde{A}$ be the apartment corresponding to $T_{E^{\on{un}}}$.
		\item For $x \in \calB(H)$, let $\calH_x$ be the corresponding parahoric group scheme over $O_E$. Let $M_x$ be the reductive quotient of $\calH_{x, k_E}$, the special fiber of $\calH_x$.
	\end{enumerate}
	
	The integral model $\calH$ of $H$ yields a hyperspecial point $x_0 \in A$. We choose $C \subseteq A$ a chamber containing $x_0$ in its closure. We denote by $\tilde{C} \subseteq \tilde{A}$ the chamber containing $C$.
	
	The group $\calH_x(O_E)$ comes equipped with its   Moy--Prasad filtration denoted by $\calH_x(O_E)_r$ for $r \in \bbR_{\geq 0}$. We have a canonical identification $\calH_x(O_E)/\calH_x(O_E)_{> 0}=M_x(k_E)$.
	
	Let $x \in \calB(H)$. The choice of $\calT$ and $C$ determines a Borel pair $B_M=T_MU_M$ of $M_x$ over $k_E$. Let $W_M$ be the finite Weyl group of $T_{M, \overline{k}_E}$.
	
	Let $\tilde{W}$ be the extended affine Weyl group of $H$ and $W_{\on{aff}}$ the affine Weyl group, i.e., the group generated by reflections along all hyperplanes in $\tilde{A}$. The Weyl group $W_M$ of $M_x$ is canonically identified with the subgroup of $\tilde{W}$ generated by reflections along hyperplanes in $\tilde{A}$ containing the image of the point $x$. Consider the composition
	\begin{equation}
		W_M \subseteq \tilde{W} \to W.
	\end{equation}
	Since the kernel of $\tilde{W} \to W$ is torsion free, the morphism $W_M \to W$ is injective.
	
	There are isomorphisms 
	\begin{equation}
		X_*(T_M) \cong X_*(\calT_{O_{E^{\on{un}}}}) \cong X_*(T).
	\end{equation}
	The composite is equivariant for the actions of $W_M$ and $\Gal(\overline{k}_E/k_E)$. Subsequently, taking group algebras over $\Lambda$, taking quotient under $W_M$ and $W$, and taking Frobenius fixed points we get a well-defined transfer morphism
	\begin{equation}
		\xi_x: (\widehat{T}_M \sslash W_M)^{\hat{F}} \to (\widehat{T} \sslash W)^{\hat{F}},
	\end{equation}
	where $\widehat{(-)}$ denotes the dual torus over $\Lambda$.
	
	\begin{definition}
		Let $\pi \in \Irr_{\Lambda}H(E)$. We say that $\pi$ has depth $r$ if $r$ is the minimal real number such that  
		$$\pi=\cup_{x \in \calB(H)}\pi^{\calH_x(O_E)_{>r}}.$$
	\end{definition}
	
	\begin{theorem}[\cite{moy1994unrefined}]
		Every representation has a well-defined depth.
	\end{theorem}
	
	\begin{theorem}
		$\on{Rep}_{\Lambda}H(E) \cong \bigoplus_r \on{Rep}_{\Lambda}^rH(E).$
	\end{theorem}
	
	We refer to \cite[Appendix A]{dat2009finitude} for a proof of this theorem over any coefficient ring.
	
	\begin{definition}
		\begin{enumerate}
			\item An unrefined minimal depth-zero type is a pair $(x, \tau)$ where $x$ is a point in $\calB(H)$ and $\tau$ is an irreducible supercuspidal representation of $M_x(k_E)$.
			\item Let $\pi$ be an irreducible depth-zero representation of $H(E)$. A depth-zero type for $\pi$ is an unrefined minimal depth-zero type $(x, \tau)$ such that there is a nonzero map $\cInd_{{\calH}_x(O_E)}^{H(E)}\tau \to \pi$.
		\end{enumerate}
	\end{definition}
	
	\begin{theorem}[\cite{moy1994unrefined}]
		Every irreducible depth-zero representation admits a depth-zero type.
	\end{theorem}
	
	\begin{remark}
		Given a depth-zero irreducible representation $\pi$, a depth-zero type $(x, \tau)$ for it can be chosen such that $x$ lies in the closure of the chamber $C$.
	\end{remark}
	
	\begin{theorem}[\cite{lanard18, lanard21}]\label{Thm_lanard}
		\begin{enumerate}
			\item Let $\pi$ be an irreducible depth-zero representation of $H(E)$ and let $(x, \tau), (x', \tau')$ be two depth-zero types for $\pi$ such that $x, x' \in A$. We have
			$$\xi_x(\on{LS}(\tau))=\xi_{x'}(\on{LS}(\tau')).$$
			In particular, we have a well-defined map
			\begin{equation}\label{Eq_LS}
				\on{LS}_H: \Irr^0_{\Lambda}H(E) \to (\widehat{T} \sslash W)^{\hat{F}}(\Lambda).
			\end{equation}
			\item Let $\frak{Z}_H^0$ be the Bernstein center of $\Rep_{\Lambda}^0H(E)$. There is a complete collection $(e_s)_{s \in (\widehat{T} \sslash W)^{\hat{F}}(\Lambda)}$ of orthogonal idempotents in $\frak{Z}_H^0$ such that for all $\pi \in \Irr_{\Lambda}^0H(E)$, we have $e_s\pi=\pi$ if and only if $\on{LS}_{H}(\pi)=s$.
		\end{enumerate}
	\end{theorem}

	\subsection{Mixed-characteristic shtukas and related stacks}\label{Subsection_Sht}
	Following \cite[Section 4.4.1]{eteve2023monodromiques}, this subsection defines mixed-characteristic shtukas and related stacks. Most of the objects were already defined in \cite{fargues2021geometrization}.
	
	Let $\Perf$ be the category of perfectoid spaces over $\Fpbar$. For every $S \in \Perf$, we have the Fargues--Fontaine curve $X_S$, and $Y_S, \calY_S$.
	
	Let $G$ be a connected reductive group over $E$. Let
	$$\Bun_G: \Perf \to \on{Groupoids}, \quad S \mapsto \{\text{$G$-torsors over $X_S$}\}.$$
	The functor $\Bun_G$ is a small $v$-stack.
	
	Let $O=O_E$ or $\bOE$. Let $\calQ$ be a smooth affine group scheme over $O$. Let $$\Bun_{\calQ}: \Perf \to \on{Groupoids}, \quad S \mapsto \{\text{$\calQ$-torsors over $\calY_S$}\}.$$
	The functor $\Bun_{\calQ}$ is a small $v$-stack.

	Recall that working over $\Fpbar$, we have the space of degree-one Cartier divisors on the Fargues-Fontaine curves (\cite[Section VI.1]{fargues2021geometrization}), explicitly, 
	\begin{equation}
		\Div^1_{\calY}=\Spd \bOE, \; \Div^1_Y=\Spd \bE, \; \Div^1_X=\Spd \bE/\varphi^{\bbZ}.
	\end{equation}
	
	\begin{definition}[Local Hecke stack for $X_S$]
		Let $I$ be a finite set. We define $\Hk^{\loc}_{G, (\Div^1_X)^I}$ to be the $v$-stack whose $S$-points form the groupoid of tuples $(x, \calE, \calE', \alpha)$ where
		\begin{enumerate}
			\item $x=(x_i)_{i \in I} \in (\DX)^I(S),$
			\item $\calE, \calE'$ are $G$-torsors over $\hat{\Gamma}_x$, the completion of $X_S$ along the graph $\Gamma_x$ of $x$,
			\item $\alpha: \calE|_{\hat{\Gamma}_x-\Gamma_x} \to \calE'|_{\hat{\Gamma}_x-\Gamma_x}$ is an isomorphism over $\hat{\Gamma}_x-\Gamma_x$ which is meromorphic along $\Gamma_x$.
		\end{enumerate}
	\end{definition}
	
	After trivializing $\calE'$, we get the Beilinson Drinfeld Grassmannian
	$$\Gr_{G, (\DX)^I}$$
	defined as in \cite[Section VI.1]{fargues2021geometrization}.
	
	As in \cite[Section VI.2]{fargues2021geometrization}, for a collection of cocharacters $\mu$, we can define a bounded-by-$\mu$ version of the local Hecke stack
	$$\Hk^{\loc}_{G, (\DX)^I, \leq \mu}$$
	that fits into a Cartesian diagram
	\[\begin{tikzcd}
		{\Gr_{G, (\DX)^I, \leq \mu}} & {\Gr_{G, (\DX)^I}} \\
		{\Hk^{\loc}_{G, (\DX)^I, \leq \mu}} & {\Hk^{\loc}_{G, (\DX)^I}}
		\arrow[from=1-1, to=1-2]
		\arrow[from=1-1, to=2-1]
		\arrow["\lrcorner"{anchor=center, pos=0.125}, draw=none, from=1-1, to=2-2]
		\arrow[from=1-2, to=2-2]
		\arrow[from=2-1, to=2-2].
	\end{tikzcd}\]

	\begin{definition}[Local Hecke stack for $\calY_S$]\label{Def_HklocY}
		Let $I$ be a finite set. We define $\Hk^{\loc}_{\calQ, (\DY)^I}$ to be the $v$-stack whose $S$-points form the groupoid of tuples $(y, \calE, \calE', \alpha)$ where
		\begin{enumerate}
			\item $y=(y_i)_{i \in I} \in (\DY)^I(S)$,
			\item $\calE, \calE'$ are $\calQ$-torsors over $\hat{\Gamma}_y$, the completion of $\calY_S$ along the graph $\Gamma_y$ of $y$,
			\item $\alpha: \calE|_{\hat{\Gamma}_y-\Gamma_y} \to \calE'|_{\hat{\Gamma}_y-\Gamma_y}$ is an isomorphism over $\hat{\Gamma}_y-\Gamma_y$.
		\end{enumerate}
	\end{definition}
	
	After trivializing $\calE'$, we get the Beilinson Drinfeld Grassmannian
	$$\Gr_{\calQ, (\DY)^I}$$
	defined as in \cite[Section VI.1]{fargues2021geometrization}.
	
	Similar to \cite[Section VI.2]{fargues2021geometrization}, for a collection of cocharacters $\mu$, we can define a bounded-by-$\mu$ version of local Hecke stack
	$$\Hk^{\loc}_{\calQ, (\DY)^I, \leq \mu}.$$
	
	\begin{definition}[Global Hecke stack for $X_S$]
		Let $I$ be a finite set. We define $\Hk^{\glob}_{G, (\DX)^I}$ to be the $v$-stack whose $S$-points form the groupoid of tuples $(x, \calE, \calE', \alpha)$ where
		\begin{enumerate}
			\item $x=(x_i)_{i \in I} \in (\DX)^I(S),$
			\item $\calE, \calE'$ are $G$-torsors over $X_S$, 
			\item $\alpha: \calE|_{X_S-\Gamma_x} \to \calE'|_{X_S-\Gamma_x}$ is an isomorphism over $X_S-\Gamma_x$, meromorphic along $\Gamma_x$.
		\end{enumerate}
	\end{definition}

	After trivializing $\calE'$, we get $\Gr_{G, (\DX)^I}$ again, by the Beauville--Laszlo lemma.
	
	\begin{definition}[Global Hecke stack for $\calY_S$]
		Let $I$ be a finite set. We define $\Hk^{\glob}_{\calQ, (\DY)^I}$ to be the $v$-stack whose $S$-points form the groupoid of tuples $(y, \calE, \calE', \alpha)$ where
		\begin{enumerate}
			\item $y=(y_i)_{i \in I} \in (\DY)^I(S),$
			\item $\calE, \calE'$ are $\calQ$-torsors over $\calY_S$, 
			\item $\alpha: \calE|_{\calY_S-\Gamma_y} \to \calE'|_{\calY_S-\Gamma_y}$ is an isomorphism over $\calY_S-\Gamma_y$, meromorphic along $\Gamma_y$.
		\end{enumerate}
	\end{definition}
	
	After trivializing $\calE'$, we get $\Gr_{\calQ, (\DY)^I}$ again, by the Beauville--Laszlo lemma.
	
	\begin{definition}[Shtuka space]\label{Def_Sht}
		Let $\calQ$ be a smooth affine group scheme over $\bOE$ with generic fiber $G_{\bE}$, and let $I$ be a finite set. Define $\Sht_{\calQ, (\DY)^I}$ by the following pullback
		\[\begin{tikzcd}
			{\Sht_{\calQ, (\DY)^I}} && {\Hk^{\glob}_{\calQ, (\DY)^I}} \\
			{\Bun_{\calQ}} && {\Bun_{\calQ} \times \Bun_{\calQ}}
			\arrow[from=1-1, to=1-3]
			\arrow[from=1-1, to=2-1]
			\arrow[from=1-3, to=2-3]
			\arrow["{(\id, \Frob_S)}", from=2-1, to=2-3].
		\end{tikzcd}\]
		The functor $\Sht_{\calQ, (\DY)^I}$ is a small $v$-stack.
	\end{definition}
	
	By definition, for $S \in \Perf$, $\Sht_{\calQ, (\DY)^I}(S)$ classifies tuples $(y, \calP, \alpha)$, where
	\begin{enumerate}
		\item $y=(y_i)_{i \in I} \in (\DY)^I(S),$
		\item $\calP$ is a $\calQ$-torsor over $\calY_S$,
		\item $\alpha: (\Frob_S^*\calP)|_{\calY_S-\Gamma_y} \to \calP|_{\calY_S-\Gamma_y}$ is an isomorphism over ${\calY_S-\Gamma_y}$ and is meromorphic along $\Gamma_y$.
	\end{enumerate}

	Let $\calG$ be a connected reductive group scheme over $\bOE$ with generic fiber $G_{\bE}$. We compare $\Sht_{\calG, (\DY)^I}$ with $\Sht_{\calG, b, \{\mu_i\}}$ defined in \cite[Definition 23.1.1]{scholze2020berkeley} (see Lemma \ref{Lemma_Sht_compare}).

	Note that there is a map 
	\begin{equation}
		\res_{\infty}: \Sht_{\calG, (\DY)^I} \to \Bun_G,
	\end{equation}
	given by restricting the $\calG$-torsor near infinity (note that away from the locus $\{p=0\} \subseteq \calY_S$, $\calG=G$) and then descending it along $Y_S \to X_S$. For $b \in B(G)$, we define $\Sht_{\calG, (\DY)^I}^{\infty=b}$ by the following Cartesian diagram: 
	\[\begin{tikzcd}
		{\Sht_{\calG, (\DY)^I}^{\infty=b}} & {*=\Spd \Fpbar} \\
		{\Sht_{\calG, (\DY)^I}} & {\Bun_G}
		\arrow[from=1-1, to=1-2]
		\arrow[from=1-1, to=2-1]
		\arrow["\lrcorner"{anchor=center, pos=0.125}, draw=none, from=1-1, to=2-2]
		\arrow["{\calE^b}", from=1-2, to=2-2]
		\arrow["{\res_{\infty}}", from=2-1, to=2-2]
	\end{tikzcd}\]
	
	Write $I=\{1, 2, ..., m\}$ for some $m \in \bbZ_{\geq 0}$. Now let $\mu:=(\mu_i \;|\; i \in I)$ be a collection of conjugacy classes of cocharacters $\bbG_m \to G_{\Ebar}$, with fields of definition $E_i$ over $E$, $i \in I$. We define 
	$$\Sht_{\calG, \TE/(\DY)^I}^{\infty=b}:=\Sht_{\calG, (\DY)^I}^{\infty=b} \times_{(\DY)^I}\TE.$$
	
	Note that $\Sht_{\calG, \TE/(\DY)^I}^{\infty=b}$ inherits a map to $\Hk^{\loc}_{\calG, \TE/(\DY)^I}$ recording the local modification. We continue to denote this map by
	$$\epsilon: \Sht_{\calG, \TE/(\DY)^I}^{\infty=b} \to \Hk^{\loc}_{\calG, \TE/(\DY)^I}.$$
	We define ${ \Sht_{\calG, \TE/(\DY)^I, \leq\mu}^{\infty=b}}$ by the following pullback diagram
	\[\begin{tikzcd}
		{ \Sht_{\calG, \TE/(\DY)^I, \leq\mu}^{\infty=b}} & { \Sht_{\calG, \TE/(\DY)^I}^{\infty=b}} \\
		{\Hk^{\loc}_{\calG, \TE/(\DY)^I, \leq \mu}} & {\Hk^{\loc}_{\calG, \TE/(\DY)^I}}
		\arrow[from=1-1, to=1-2]
		\arrow[from=1-1, to=2-1]
		\arrow["\epsilon", from=1-2, to=2-2]
		\arrow[from=2-1, to=2-2].
	\end{tikzcd}\]

	Unraveling the definitions, we have the following lemma.
	
	\begin{lemma}\label{Lemma_Sht_compare}
		Let $\Sht_{\calG, b, \mu}$ be the moduli space defined in \cite[Definition 23.1.1]{scholze2020berkeley}. We have
		\begin{equation}
			\Sht_{\calG, b, \mu} \cong {\Sht_{\calG, \TE/(\DY)^I, \leq\mu}^{\infty=b}}.
		\end{equation}
	\end{lemma}

	\begin{definition}\label{Def_Sht_depth0}
		Define 
		$$\Sht_{\calQ, (\DY)^I, \sbar},$$ 
		the space of $\calQ$-shtukas with level-$\sbar$ structure  as follows. For every $S \in \on{Perf}$, the groupoid $\Sht_{\calQ, (\DY)^I, \sbar}(S)$ classifies tuples $(y, \calP, \alpha, \psi_{\sbar})$, where
		\begin{enumerate}
			\item $(y, \calP, \alpha) \in \Sht_{\calQ, (\DY)^I}(S)$ such that $y \in (\DY-\sbar)^I(S)$. In other words, $y \in (\DY-\sbar)^I(S)$, $\calP$ is a $\calQ$-torsor over $\calY_S$, and $\alpha: (\Frob_S^*\calP)|_{\calY_S-\Gamma_y} \to \calP|_{\calY_S-\Gamma_y}$ is an isomorphism over ${\calY_S-\Gamma_y}$ which is meromorphic along $\Gamma_y$.
			\item $\psi_{\sbar}: \calP|_{S} \xrightarrow{\cong} \calE^0|_{S}$ is a trivialization of $\calP$ over $S \cong \sbar \times S$ that is compatible with $\alpha$, where $|_S$ means pullback along the trivial untilt $S \to \calY_S$.
		\end{enumerate}
	\end{definition}

	Recall that for every compact open subgroup $K \subseteq G(E)$, \cite[Section 23]{scholze2020berkeley} defines
	$$\Sht_{G, b, \mu, K},$$
	a shtuka space with level-$K$ structure. Thanks to \cite[Proposition 23.3.1]{scholze2020berkeley}, we can encode the level structure by changing the group $\calG$. The main examples that we will use are as follows.
	
	\begin{example}
		Let $\calG'$ be a Bruhat--Tits group scheme over $O_E$ with generic fiber $G$. When $K=\calG'(O_E)$, we have 
		$$\Sht_{G, b, \mu, \calG'(O_E)} \cong \Sht_{\calG', b, \mu} \cong {\Sht_{\calG', \TE/(\DY)^I, \leq\mu}^{\infty=b}}.$$
		
		In particular, when $\calG'=\calI$ is an Iwahori group scheme, we have 
		$$\Sht_{G, b, \mu, \calI (O_E)} \cong {\Sht_{\calI, \TE/(\DY)^I, \leq\mu}^{\infty=b}}.$$
		Thanks to \cite[Corollary 2]{arnaud2023neron}, the space $\Sht_{G, b, \mu, \calI (O_E)}$	 classifies $G$-shtukas with $B$-reduction, where $B \subseteq G_{\bE}$ is the Borel subgroup corresponding to $\calI$.
		
		Similarly, when $\calG'=\calI^0$, the pro-unipotent radical of an Iwahori $\calI$, we have
		$$\Sht_{G, b, \mu, \calI^0 (O_E)} \cong {\Sht_{\calI^0, \TE/(\DY)^I, \leq\mu}^{\infty=b}}.$$
		The space $\Sht_{G, b, \mu, \calI^0 (O_E)}$ classifies $G$-Shtukas with $U$-reduction, where $U$ is the unipotent radical of $B$.
	\end{example}

	\begin{example}\label{Example_depthzero}
		Let $\sigma$ be a polysimplex in the Bruhat-Tits building of $G$ over $E$. Let $\calG_{\sigma}$ be the parahoric group scheme over $O_E$ associated to $\sigma$. When $K=\ker(\calG_{\sigma}(O_E) \to \calG_{\sigma}(k_E))$, by \cite[Corollary 2]{arnaud2023neron}, we have 
		\begin{equation}\label{Eq_Sht_1}
			\Sht_{G, b, \mu, K} \cong {\Sht_{\calG_{\sigma, 1}, \TE/(\DY)^I, \leq\mu}^{\infty=b}} 
		\end{equation}
		where $\calG_{\sigma, 1}$ is obtained from $\calG_{\sigma}$ by dilating the trivial subgroup on the special fiber. The space ${\Sht_{\calG_{\sigma, 1}, \TE/(\DY)^I, \leq\mu}^{\infty=b}}$ admits the following moduli interpretation. 
		For $S \in \Perf$, the groupoid ${\Sht_{\calG_{\sigma, 1}, \TE/(\DY)^I, \leq\mu}^{\infty=b}}(S)$ classifies tuples $(y, \calP, \alpha, \iota_r, \psi_{\sbar})$, where
		\begin{enumerate}
			\item $y=(y_i)_{i \in I} \in \TE(S),$
			\item $\calP$ is a $\calG_{\sigma}$-torsor over $\calY_S$,
			\item $\alpha: (\Frob_S^*\calP)|_{\calY_S-\Gamma_y} \to \calP|_{\calY_S-\Gamma_y}$ is an isomorphism over ${\calY_S-\Gamma_y}$ which is meromorphic along $\Gamma_y$ and bounded by $\mu$,
			\item $\iota_r: \calP|_{\calY_S[r, \infty)} \to \calE^b|_{\calY_S[r, \infty)}$ is an isomorphism for some $r \gg 0$,
			\item $\psi_{\sbar}: \calP|_S \to \calE^0|_S$ is a trivialization of $\calP$ over $\sbar \times S \cong S$ that is compatible with $\alpha$.
		\end{enumerate}
		In other words, the space ${\Sht_{\calG_{\sigma, 1}, \TE/(\DY)^I, \leq\mu}^{\infty=b}}$ is obtained from the space $\Sht_{\calG_{\sigma}, (\DY)^I, \sbar}$ (Definition \ref{Def_Sht_depth0}) by base change to $\TE$, fixing a trivialization $\calP \cong \calE^b$ near infinity, and imposing a bounded-by-$\mu$ condition on local modifications:
		\begin{equation}\label{Eq_Sht_2}
			{\Sht_{\calG_{\sigma, 1}, \TE/(\DY)^I, \leq\mu}^{\infty=b}} \cong (\Sht_{\calG_{\sigma}, (\DY)^I, \sbar} \times_{(\DY)^I} \TE)^{\infty=b}_{\leq \mu }.
		\end{equation}

		For example, when $\sigma$ is a hyperspecial vertex, the cohomology of ${\Sht_{\calG_{\sigma, 1}, \TE/(\DY)^I, \leq\mu}^{\infty=b}}$ encodes depth-zero representations. More specifically, when $\calG_{\sigma}$ is $GL_2$ over $O_E$, we have 
		$$\calG_{\sigma, 1}(O_E)=\begin{pmatrix}
			1+(\pi_E) & (\pi_E) \\ 
			(\pi_E) & 1+(\pi_E)
		\end{pmatrix}.$$
		
		When $\sigma$ is not a vertex, the cohomology of ${\Sht_{\calG_{\sigma, 1}, \TE/(\DY)^I, \leq\mu}^{\infty=b}}$ contains higher-depth representations in general. More specifically, when $G=GL_2$ and $\calG_{\sigma}$ is an Iwahori subgroup with
		$$\calI(O_E):=\calG_{\sigma}(O_E)=\begin{pmatrix}
			O_E^* & O_E \\
			(\pi_E) & O_E^*
		\end{pmatrix},$$
		we have 
		$$\calG_{\sigma, 1}(O_E)=\begin{pmatrix}
			1+(\pi_E) & (\pi_E)\\
			(\pi_E^2) & 1+(\pi_E)
		\end{pmatrix}.$$
		Let $V_{\sigma}$ be the unipotent radical of the special fiber of $\calG_{\sigma}$. Let 
		$$\calI^0(O_E)=\ker(\calG_{\sigma}(O_E) \to \calG_{\sigma}(k_E)/V_{\sigma}(k_E)).$$
		Note that we have the following diagram of groups
		\[\begin{tikzcd}
			1 & {\calG_{\sigma, 1}(O_E)} & {\calI^0(O_E)} & {V_{\sigma}(k_E)} & 1 \\
			1 & {\calG_{\sigma, 1}(O_E)} & {\calG_{\sigma}(O_E)} & {\calG_{\sigma}(k_E)} & 1
			\arrow[from=1-1, to=1-2]
			\arrow[from=1-2, to=1-3]
			\arrow[from=1-2, to=2-2]
			\arrow[from=1-3, to=1-4]
			\arrow[from=1-3, to=2-3]
			\arrow["\lrcorner"{anchor=center, pos=0.125}, draw=none, from=1-3, to=2-4]
			\arrow[from=1-4, to=1-5]
			\arrow[from=1-4, to=2-4]
			\arrow[from=2-1, to=2-2]
			\arrow[from=2-2, to=2-3]
			\arrow[from=2-3, to=2-4]
			\arrow[from=2-4, to=2-5].
		\end{tikzcd}\]
		
		Taking $V_{\sigma}(k_E)$-invariants of the cohomology of ${\Sht_{\calG_{\sigma, 1}, \TE/(\DY)^I, \leq\mu}^{\infty=b}}$ cuts out the depth-zero part. This corresponds to the fact that
		$$\calI^0(O_E)=\ker(\calG_{\sigma}(O_E) \to \calG_{\sigma}(k_E)/V_{\sigma}(k_E))=\begin{pmatrix}
			1+(\pi_E) & O_E\\
			(\pi_E) & 1+(\pi_E)
		\end{pmatrix},$$
		and
		$$\Sht_{G, b, \mu, \calI^0(O_E)}=\Sht_{G, b, \mu, \infty}/\calI^0(O_E)=(\Sht_{G, b, \mu, \infty}/\calG_{\sigma, 1}(O_E))/V_{\sigma}(k_E)=\Sht_{G, b, \mu, \calG_{\sigma, 1}}/V_{\sigma}(k_E).$$
		
	\end{example}

	\subsection{Assumptions on smoothness}\label{Subsection_ass_sm}
	
	We assume certain smoothness of the cohomology of the shtuka space over the local Hecke stack, in the sense that pullback commutes with nearby cycles, as follows.
	
	\begin{assumption}\label{Assum_sm}
		Let $\sigma$ be a polysimplex in the Bruhat-Tits building $\calB(G, E)$. Let $\calG_{\sigma}$ be the corresponding parahoric group scheme. Let $M_{\sigma}$ be the reductive quotient of the special fiber of $\calG_{\sigma}$. We fix a Borel pair $B_M=T_MU_M$. Let $C$  be the alcove corresponding to the Borel $B_M$ and containing $\sigma$ in its closure. Let $\calG_C$ be the corresponding Iwahori group scheme and $\calG_C^0$ be its pro-$p$ Iwahori group scheme. Let $b \in B(G)$ and $I$ a finite set. Let $\mu$ be a collection of dominant cocharacters of $T$. Consider the map 
		$$\epsilon:  \Sht^{\infty=b}_{\calG_{C}, (\Spd \bOE)^I, \leq \mu} \to \Hk^{\loc}_{\calG_{C}^0, (\Spd \bOE)^I, \leq \mu}/Ad_FT_M$$
		induced from the map
		$$\Sht^{\infty=b}_{\calG_{C}^0, (\Spd \bOE)^I} \to \Hk^{\loc}_{\calG_{C}^0, (\Spd \bOE)^I}.$$
		We assume that $\epsilon^*$ commutes with iterated nearby cycles over $(\Spd \bOE)^I$.

	\end{assumption}

	\subsection{Statements of the main theorems}
	\label{Subsection_statement}
	
	Recall that $E$ is a non-archimedean local field, $\calG$ is a connected reductive group scheme over $O_E$ with generic fiber $G$. Let $W_E$ denote the Weil group of $E$. Let $\Lambda=\overline{\bbF}_{\ell}$. We recall the following theorem.
	
	\begin{construction}[{\cite[Chapter VIII]{fargues2021geometrization}}]
		There is a map $\on{LLC}^{\on{FS}}: \pi \mapsto \varphi_{\pi}^{\on{FS}}$ from the set of smooth irreducible $\Lambda$-linear representations of $G(E)$ to the set of   semisimple  $L$-parameters.
	\end{construction}
	
	Let $(\varphi_{\pi}^{\on{FS}})|_{I_E}$ be the restriction of $\varphi_{\pi}^{\on{FS}}$ to the inertia subgroup $I_E$. The goal of this paper is to compute  $(\varphi_{\pi}^{\on{FS}})|_{I_E}$ for $\pi$ a depth-zero irreducible representation. 
	
	Let us now formulate our main results. Let $T \subseteq G$ be a maximal torus and let $W$ be the Weyl group of $(G, T)$. We recall the following result from \cite{lanard18} and \cite{lanard21} (see also Theorem \ref{Thm_lanard}).

	\begin{theorem}
		The category $\Rep_{\Lambda}^0G(E)$ of depth-zero representations of $G(E)$ decomposes as a direct sum
		\begin{equation}
			\Rep_{\Lambda}^0G(E) = \bigoplus_{s \in (\widehat{T}\sslash W)^{\widehat{F}}(\Lambda) }\Rep_{\Lambda}^sG(E),
		\end{equation}
		where $\widehat{T}$ denotes the dual torus over $\Lambda$, $\widehat{T}\sslash W$ the GIT quotient by the action of $W$ and $(-)^{\widehat{F}}$ the scheme of invariants under the morphism dual to the Frobenius of $T$.
	\end{theorem}
	
	The above decomposition of $\Rep_{\Lambda}^0G(E)$ yields a map
	\begin{equation}
		\on{LS}: \Irr_{\Lambda}^0G(E) \to (\widehat{T}\sslash W)^{\widehat{F}}(\Lambda),
	\end{equation}
	where $\Irr_{\Lambda}^0G(E)$ is the set of irreducible depth-zero representations of $G(E)$, characterized by $\on{LS}(\pi)=s$ if and only if $\pi$ lies in the direct summand indexed by $s$.
	
	Fix a topological generator $\tau_E$  of the tame inertia $I_E/P_E$ that corresponds to the isomorphism $\overline{k}_E^* \cong (\bbQ/\bbZ)_{p'}$. We denote by $\on{ev}_{\tau_E}(\varphi)$ the image in $\widehat{G}\sslash \widehat{G}=\widehat{T}\sslash W$ of $\varphi(\tau_E)$. We can now state our main theorem.
	
	\begin{theorem}\label{Thm_LLC}
		Assume Assumption \ref{Assum_sm} and \ref{Assum_CT}. Let $\pi \in \Irr_{\Lambda}^0G(E)$ be an irreducible depth-zero representation.
		\begin{enumerate}
			\item The corresponding $L$-parameter $\varphi_{\pi}^{\on{FS}}$ is tame. In particular, we can view $\varphi_{\pi}^{\on{FS}}$ as a map from $W_E/P_E$.
			\item The restriction of $\varphi_{\pi}^{\on{FS}}$ to $I_E/P_E$ is determined by its value on the topological generator $\tau_E \in I_E/P_E$, in the sense that the following diagram commutes
			\[\begin{tikzcd}
				{\Irr_{\Lambda}^0G(E)} & {(Z^1(W_E/P_E, \widehat{G})\sslash \widehat{G})(\Lambda)} \\
				{(\widehat{T}\sslash W)^{\widehat{F}}(\Lambda)} & {(\widehat{T}\sslash W)^{\widehat{F}}(\Lambda)}
				\arrow["{\on{LLC}^{\on{FS}}}", from=1-1, to=1-2]
				\arrow["{\on{LS}}"', from=1-1, to=2-1]
				\arrow["{\on{ev}_{\tau_E}}", from=1-2, to=2-2]
				\arrow["{=}", from=2-1, to=2-2].
			\end{tikzcd}\]
		\end{enumerate}
	\end{theorem}

	\begin{remark}\label{Rmk_depend}
		Recall that we have fixed a trivialization of the group of roots of unity $\overline{k}_E^* \cong (\bbQ/\bbZ)_{p'}$ where $p'$ denotes the prime-to-$p$ part. This choice enters in two places, and the two dependencies cancel:
		\begin{enumerate}
			\item It occurs in the construction of the map $$\on{LS}: \Irr^0_{\Lambda}G(E) \to (\widehat{T}\sslash W)^{\widehat{F}}(\Lambda).$$
			\item It occurs in the choice of a topological generator $\tau_E \in I_E/P_E$.
		\end{enumerate}
	\end{remark}

	We now explain the proof strategy for Theorem \ref{Thm_LLC}. By construction of Fargues-Scholze $L$-parameters $\varphi_{\pi}^{\on{FS}}$, we need to compute certain Hecke operators, i.e., cohomology of certain shtuka space. Recall that $\calG$ is a connected reductive group scheme over $O_E$ with generic fiber $G$. Let $\sigma$ be a polysimplex in the Bruhat-Tits building $\calB(G, E)$. Let $\calG_{\sigma}$ be the parahoric group scheme over $O_E$ corresponding to $\sigma$. Let $\breve{S}_0:=\DY=\Spd \bOE, \etab:=\Spd \bE, \sbar:=\Spd \overline{k}_E$. Choose a geometric point $\overline{\eta}$ of $\etab$. Fix $b \in B(G)$.

	As in Definition \ref{Def_Sht}, we have the associated shtuka space 
	\begin{equation}
		p_{\sigma, b, I}: \Sht^{\infty=b}_{\calG_{\sigma}, (\DY)^I} \to (\DY)^I=(\Spd \bOE)^I.
	\end{equation}
	Restricting to $(\Spd \bE)^I$ and adding the datum of a trivialization at $\sbar \times S$, we get the shtuka space with level structure $\sbar$,
	\begin{equation}
		p_{\sigma, b, I, \sbar}: \Sht^{\infty=b}_{\calG_{\sigma}, (\DY)^I, \sbar} \to (\Spd \bOE-\sbar)^I=(\Spd \bE)^I.
	\end{equation} 
	We also have the map
	$$\epsilon_{\sigma, b, I}:  \Sht^{\infty=b}_{\calG_{\sigma}, (\DY)^I} \to \Hk^{\loc}_{\calG_{\sigma}, (\DY)^I}$$
	recording the local modification. Similarly, we have 
	$$\epsilon_{\sigma, b, I, \sbar}: \Sht^{\infty=b}_{\calG_{\sigma}, (\DY)^I, \sbar}  \to \Hk^{\loc}_{\calG_{\sigma}, (\Spd \bE)^I}.$$
	
	\begin{definition}
		\begin{enumerate}
			\item Choose $\sigma, b, I$ as before. For $V \in \Rep(\widehat{G}^I)$, the \textit{cohomology of the shtuka space with depth-zero level structure associated with $(\sigma, b, I, V)$} is 
			$$\calH_{\sigma, b, I, \overline{s}_0, V}:=p_{\sigma, b, I, \sbar, !} \epsilon_{\sigma, b, I, \sbar}^*\Sat V \in \scrD_{\et}((\Spd \bE)^I, \Lambda),$$
			where $\Sat V \in \scrD_{\et}(\Hk^{\loc}_{\calG_{\sigma}, \breve{\eta}^I})$ is the Satake sheaf.
			\item We write
			$$H_{\sigma, b, I, \overline{s}_0, V}:=(\calH_{\sigma, b, I, \overline{s}_0, V})|_{\Delta(\overline{\eta})} \in \scrD_{\et}(\overline{\eta})$$
			for the fiber of $\calH_{\sigma, b, I, \overline{s}_0, V}$ at the geometric generic point $\overline{\eta}$.
		\end{enumerate}
	\end{definition}

	Since $H_{\sigma, b, I, \overline{s}_0, V}$ is the pullback to the geometric fiber, it carries a $\Gal(\overline{\bE}/\bE)^I=(I_E)^I$-action by \cite[Proposition IX.3.2]{fargues2021geometrization}.
	
	We next give an expression for $\calH_{\sigma, b, I, \overline{s}_0, V}$ in terms of Hecke operators. This will allow us to express $\calH_{\sigma, b, I, \overline{s}_0, V}$ in terms of Hecke operators. In particular, it shows that $\calH_{\sigma, b, I, \overline{s}_0, V}$ is equipped with an action of $\calO(Z^1(W_E, \widehat{G})/\widehat{G})$, the ring of functions on the stack of $L$-parameters.
	For every compact open subgroup $K \subseteq G(E)$ and $W \in \Rep_{\Lambda}(\widehat{G}^I)$, we recall the maps
	\begin{equation}
		f_K: \Sht_{G, b, \mu, K} \to \prod_{i \in I} \Spd \breve{E}_i,  \qquad \pi_K: \Sht_{G, b, \mu, K} \to \Gr^{\on{tw}}_{G, \prod_{i \in I}\Spd \breve{E}_i, \leq \mu}
	\end{equation}
	and the sheaf $\calS_W$ from \cite[Section IX.3]{fargues2021geometrization}.
	
	\begin{lemma}\label{Lemma_Sht_Exc}
		Fix $b \in B(G)$ as before. Let $K=\calG_{\sigma, 1}(O_E)$. Let $V \in \Rep_{\Lambda}(\widehat{G}^I)$ be an irreducible representation. Let $\mu$ be a tuple of cocharacters that is sufficiently dominant. We have
		\begin{enumerate}
			\item $\calH_{\sigma, b, I, \overline{s}_0, V}=f_{K, !}(\calS_V|_{\Sht_{G, b, \mu, K}})$.
			\item $f_{K, !}(\calS_V|_{\Sht_{G, b, \mu, K}})=i_b^*T_Vi_{1, !}(\cInd_K^{G(E)}\Lambda)$.
			\item $\calH_{\sigma, b, I, \overline{s}_0, V}$ carries an action of $\calO(Z^1(W_E, \widehat{G})/\widehat{G})$.
		\end{enumerate}
	\end{lemma}
	
	\begin{proof}
		For (1), let us first recall  equations \eqref{Eq_Sht_1} and \eqref{Eq_Sht_2} comparing $\Sht^{\infty=b}_{\calG_{\sigma}, (\DY)^I, \sbar}$ and $\Sht_{G, b, \mu, K}$. It suffices to show that the pullbacks of $\Sat V$ and $\calS_V$ to $\Sht_{G, b, \mu, K}$ agree. Away from the twisted diagonals, this follows from the definition of $\calS_V$. By uniqueness of ULA extensions, it therefore suffices to show that $(\Sat V)|_{\Sht_{G, b, \mu, K}}$ and $(\calS_V)|_{\Sht_{G, b, \mu, K}}$ are ULA over $\prod_{i \in I}\Spd \bE_i$. This follows from Lemma \ref{Lemma_ULA}.
		
		For (2), see \cite[Proof of Proposition IX.3.2]{fargues2021geometrization}.
		
		For (3), recall that we have a map
		\begin{equation}
			\calO(Z^1(W_E, \widehat{G})/\widehat{G}) \to \on{End}(\on{id}_{\scrD_{\et}(\Bun_G)}).
		\end{equation}
		The categorical center $\on{End}(\on{id}_{\scrD_{\et}(\Bun_G)})$ acts on $T_Vi_{1, !}(\cInd_K^{G(E)}\Lambda) \in \scrD_{\et}(\Bun_G)$. Therefore, (3) follows from (2).
		
	\end{proof}
	
	Note that the map
	\begin{equation}
		\Sht_{G, b, \mu, K} \to \Hk^{\loc}_{\calG_\sigma, \prod_{i \in I}\Spd \bE_i}=\Hk^{\loc}_{G,\; \prod_{i \in I}\Spd \bE_i}
	\end{equation}
	factors through the map $\pi_K$: 
	\[\begin{tikzcd}
		{\Sht_{G, b, \mu, K}} & {\Gr^{\on{tw}}_{G, \prod_{i \in I}\Spd \breve{E}_i, \leq \mu}} \\
		& {\Hk^{\loc}_{G, \prod_{i \in I}\Spd \bE_i}}
		\arrow["{\pi_K}", from=1-1, to=1-2]
		\arrow["{\epsilon_K}"', from=1-1, to=2-2]
		\arrow["{a_K}", from=1-2, to=2-2].
	\end{tikzcd}\]
	
	\begin{lemma}\label{Lemma_ULA}
		\begin{enumerate}
			\item The sheaf $\pi_K^*(\calS_{V})=(\calS_{V})|_{\Sht_{G, b, \mu, K}}$ is ULA over $\prod_{i \in I}\Spd \bE_i$.
			\item The sheaf $\epsilon_K^*(\Sat V)=(\Sat V)|_{\Sht_{G, b, \mu, K}}$ is ULA over $\prod_{i \in I}\Spd \bE_i$.
		\end{enumerate}
	\end{lemma}
	
	\begin{proof}
		(1) follows from the fact that $\calS_V$ is ULA over $\prod_{i \in I}\Spd \bE_i$ and the period map
		$$\pi_K: \Sht_{G, b, \mu, K} \to \Gr^{\on{tw}}_{G, \prod_{i \in I}\Spd \breve{E}_i}$$
		is \'etale (\cite[Corollary 23.5.3]{scholze2020berkeley}).
		
		For (2), It suffices to show that $a_K^*\Sat V$ is ULA over $\prod_{i \in I}\Spd \bE_i$. This can be shown using the explicit description of $\Gr^{\on{tw}}$ (\cite[Proposition 23.4.2]{scholze2020berkeley}) and \cite[Lemma VI.6.3]{fargues2021geometrization}.
		
	\end{proof}

	We would like to cut out the depth-zero part of $H_{\sigma, b, I, \overline{s}_0, V}$ (cf. Example \ref{Example_depthzero}). By changing the level structure, there is a $\calG_{\sigma}(k_E)$-action  on $\Sht^{\infty=b}_{\calG_{\sigma}, (\DY)^I, \sbar}$, hence on $\cHV$ and $\HV$. Let $V_{\sigma}$ be the unipotent radical of $\calG_{\sigma, k_E}$ and $M_{\sigma}$ the reductive quotient. We consider the invariants
	$$(\HV)^{V_{\sigma}(k_E)},$$ 
	which then carries an action of $M_{\sigma}(k_E)$.
	
	Let $T_{M_{\sigma}}$ be a maximal torus of $M_{\sigma}$, and $W_{M_{\sigma}}$ the Weyl group of $T_{M_{\sigma}, \overline{k}_E}$. Recall that from Deligne-Lusztig theory, for every $s \in (\widehat{T}_{M_{\sigma}}\sslash W_{M_{\sigma}})^{\widehat{F}}(\Lambda)$, there is an associated idempotent element $e_s$ in the group algebra $\Lambda[M_{\sigma}(k_E)]$. Using the $M_{\sigma}(k_E)$-action on $(\HV)^{V_{\sigma}(k_E)}$, we can define
	$$e_s(\HV)^{V_{\sigma}(k_E)},$$ 
	the summand of cohomology of shtuka with respect to $s$. Moreover, we can view $s$ as an element in $(\widehat{T}\sslash W)^{\widehat{F}}(\Lambda)$ thanks to \cite{lanard18} and \cite{lanard21} (see Theorem \ref{Thm_lanard}).

	\begin{theorem}\label{Theorem_action}
		Assume Assumption \ref{Assum_sm} and \ref{Assum_CT}.
		\begin{enumerate}
			\item For every $I, V, \sigma, b$ as above, the $(I_E)^{I}$-module $(\HV)^{V_{\sigma}(k_E)}$ is tame, i.e., the $(I_E)^I$-action is trivial on the wild inertia $(P_E)^I$.
			\item For every $s \in (\widehat{T}_{M_{\sigma}}\sslash W_{M_{\sigma}})^{\widehat{F}}(\Lambda)$, the $\calO(Z^1(W_E, \widehat{G})/\widehat{G})$-module 
			$$e_s(\HV)^{V_{\sigma}(k_E)}$$ 
			is supported on $\on{ev}_{\tau_E}^{-1}(s)$. 
		\end{enumerate}
	\end{theorem}

	\subsection{The strategy and recollections on Eteve's thesis}\label{Subsection_strategy}
	In order to prove Theorem \ref{Thm_LLC} and \ref{Theorem_action}, we will follow the strategy of \cite{eteve2023monodromiques}, where Eteve proved the similar result when $E$ is a non-archimedean local field of equal characteristic. A large part of that argument extends to the mixed-characteristic case. In this subsection, we review \cite{eteve2023monodromiques}  and point out the part that we need to modify.
	
	There are two basic ideas to prove Theorem \ref{Thm_LLC} and \ref{Theorem_action}: Reducing from $\Sht$ to $\Hk^{\loc}$ and reducing from a general unramified reductive group $G$ to the torus case. The reduction from $\Sht$ to $\Hk^{\loc}$ has already been carried out in \cite{eteve2023monodromiques} and can be adapted to our case under Assumption \ref{Assum_sm}, as we are going to explain now.	
	
	For every smooth affine group scheme $\calQ$ over $\bOE$ with generic fiber $G_{\bE}$ , $b \in B(G)$, and a finite set $I$, we have a map recording the local modifications
	$$\epsilon: \Sht^{\infty=b}_{\calQ, (\Spd \bOE)^I} \to \Hk^{\loc}_{\calQ, (\Spd \bOE)^I},$$ 
	which is defined over $(\Spd \bOE)^I$. 
	We expect the map $\epsilon$ to be "smooth" in the sense that $\epsilon^*$ commutes with nearby cycles. We refer the reader to Assumption \ref{Assum_sm} for the precise assumption we make.

	In our case, we can encode the group $\calG_{\sigma}$ together with the level-$\overline{s}_0$  structure  into a single group $\calQ=\calG_{\sigma, 1}$, in the sense that a $\calG_{\sigma}$-torsor together with a level-$\overline{s}_0$  structure is the same as a $\calG_{\sigma, 1}$-torsor (\cite[Corollary 2]{arnaud2023neron}).

	The next idea is to chop off the cohomology in question,
	$$(\cHV)^{V_{\sigma}(k_E)}.$$
	Recall that $(\cHV)^{V_{\sigma}(k_E)}$ carries a $M_{\sigma}(k_E)$-action. The choice of the alcove $C$ and the integral torus $\calT$ determines a Borel pair $B_M=T_MU_M$ of $M_{\sigma}$. By Deligne-Lusztig theory, for every $w$ in the Weyl group $W_M$ of $M_{\sigma}(k_E)$, we have a Deligne-Lusztig restriction functor,
	$${^*}R_{w}: \Rep_{\Lambda}(M_{\sigma}(k_E)) \to \Rep_{\Lambda}(T^{wF}).$$
	Moreover, the functors $\{{^*}R_{w}\}_{w \in W_M}$ are jointly conservative. So we can look at one $w$ at  a time and try to control the cohomology
	$${^*}R_{w}(\cHV)^{V_{\sigma}(k_E)}.$$
	It turns out that ${^*}R_{w}(\cHV)^{V_{\sigma}(k_E)}$ can be interpreted as certain cohomology of 
	$$\Sht^{\infty=b}_{\calG_{\sigma}, (\DY)^I, \overline{s}_0} \times^{\calG_{\sigma}(k_E)} M_{\sigma}/B_{M_{\sigma}},$$
	by a twisted version of K\"{u}nneth formula. More precisely, we recall the following lemma.
	
	Let $w \in W_M$. Recall that we have the following diagram over $\sbar$
	\[\begin{tikzcd}
		{\frac{U_M\backslash B_MwB_M/U_M}{Ad_FT_M}} & {\frac{U_M\backslash M_\sigma/U_M}{Ad_FT_M}} \\
		{[*/T^{wF}] \ltimes (U_M \cap w^{-1}U_{M}w)} \\
		{[*/T^{wF}]}
		\arrow["{j_w}", from=1-1, to=1-2]
		\arrow["{=}", from=1-1, to=2-1]
		\arrow["{k_w}", from=2-1, to=3-1].
	\end{tikzcd}\]
	Let 
	\begin{equation}\label{Eqdef_nabla}
		\nabla_w:=j_{w, *}k_w^*: \scrD_{\et}([*/T^{wF}], \Lambda) \to \scrD_{\et}({\frac{U_M\backslash M_\sigma/U_M}{Ad_FT_M}}, \Lambda).
	\end{equation}
	
	\begin{lemma}[{\cite[Lemma 4.5.12]{eteve2023monodromiques}}]\label{Lemma_4512}
		We have 
		\begin{equation}
			^*R_{w}(\cHV)^{V_{\sigma}(k_E)} \cong p_!\epsilon^*(\Sat(V) \boxtimes \nabla_w(\on{Reg}_{T^{wF}})),
		\end{equation}
		where $\on{Reg}_{T^{wF}}$ is the regular representation of $T^{wF}$ and
		\[\begin{tikzcd}
			{(\Spd \breve{E})^I \times \overline{s}_0} & {\Sht^{\infty=b}_{\calG_{\sigma}, (\DY)^I, \overline{s}_0} \times^{\calG_{\sigma}(k_E)}M_{\sigma}/B_{M_{\sigma}}} & {\Hk^{\loc}_{\calG_{\sigma}, (\Spd \breve{E})^I} \times \frac{U_{M_{\sigma}}\backslash M_{\sigma} /U_{M_{\sigma}}}{Ad_FT_{M_{\sigma}}}}
			\arrow["p"', from=1-2, to=1-1]
			\arrow["\epsilon", from=1-2, to=1-3].
		\end{tikzcd}\]
	\end{lemma}
	
	The key observation is that we have a map 
	$$\Sht^{\infty=b}_{\calG_{\sigma}, (\DY)^I, \sbar} \times M_{\sigma}/B_{M_{\sigma }} \to \Sht^{\infty=b}_{\calG_C, (\DY)^{I \sqcup 0}}|_{(\Spd \breve{E})^I \times \sbar},$$
	where $\calG_C$ is obtained from $\calG_{\sigma}$ by dilating the Borel on the special fiber.
	This induces a closed embedding
	$$\emb: \Sht^{\infty=b}_{\calG_{\sigma}, (\DY)^I, \overline{s}_0} \times^{\calG_{\sigma}(k_E)} M_{\sigma}/B_{M_{\sigma}} \to (\Sht^{\infty=b}_{\calG_C, (\DY)^{I \sqcup 0}})|_{(\Spd \breve{E})^I \times \overline{s}_0}.$$
	Indeed, tracking the moduli interpretation gives the stronger statement below.
	
	\begin{lemma}[{\cite[Lemma 4.5.4]{eteve2023monodromiques}}]\label{Lemma_Sht_Cart}
		Let $\calG_C^0$ denote the group scheme over $O_E$ obtained from $\calG_{\sigma}$ by dilating the unipotent radical of the Borel on the special fiber. Set $\Iw:=L^+\calG_C$ and $\Iw^0:=L^+\calG_C^0$. We have a Cartesian diagram
		\[\begin{tikzcd}
			{\Sht^{\infty=b}_{\calG_{\sigma}, (\DY)^I, \overline{s}_0} \times^{\calG_{\sigma}(k_E)}M_{\sigma}/B_{M_{\sigma}}} & {\Sht^{\infty=b}_{\calG_C, (\DY)^I \sqcup 0}|_{(\Spd \breve{E})^I \times \overline{s}_0}} \\
			{\Hk^{\loc}_{G, (\Spd \breve{E})^I} \times \frac{\Iw^0\backslash L^+\calG_{\sigma}/\Iw^0}{Ad_FT_{M_{\sigma}}}} & {\Hk^{\loc}_{G, (\Spd \breve{E})^I} \times \frac{\Iw^0\backslash L\calG_{\sigma}/\Iw^0}{Ad_FT_{M_{\sigma}}} \cong \frac{\Hk^{\loc}_{\calG_C^0, I\sqcup 0}}{Ad_FT_{M_{\sigma}}}|_{(\Spd \breve{E})^I \times \overline{s}_0}}
			\arrow["\emb", from=1-1, to=1-2]
			\arrow[from=1-1, to=2-1]
			\arrow[from=1-2, to=2-2]
			\arrow[from=2-1, to=2-2].
		\end{tikzcd}\]
		In particular, the map $\emb$ is induced by the map $L^+\calG_{\sigma} \to L\calG_{\sigma}$, hence a closed embedding.
	\end{lemma}

	Therefore, we have a diagram
	\[\begin{tikzcd}
		{(\Spd \breve{E})^I} & {\Sht^{\infty=b}_{\calG_{\sigma}, (\DY)^I, \overline{s}_0} \times^{\calG_{\sigma}(k_E)}M_{\sigma}/B_{M_{\sigma}}} & {\Hk^{\loc}_{G, (\Spd\breve{E})^I} \times \frac{\Iw^0\backslash L^+\calG_{\sigma}/\Iw^0}{Ad_FT_{M_{\sigma}}}} \\
		{(\Spd \bOE)^I} & {\Sht^{\infty=b}_{\calG_C, (\DY)^{I \sqcup 0}}|_{(\Spd \breve{O}_E)^I \times \overline{s}_0}} & {\frac{\Hk^{\loc}_{\calG_C^0, (\DY)^{I\sqcup 0}}}{Ad_FT_{M_{\sigma}}}|_{{(\Spd \breve{O}_E)^I \times \overline{s}_0}}}
		\arrow[from=1-1, to=2-1]
		\arrow[from=1-2, to=1-1]
		\arrow[from=1-2, to=1-3]
		\arrow["\emb"', from=1-2, to=2-2]
		\arrow[from=1-3, to=2-3]
		\arrow["p"', from=2-2, to=2-1]
		\arrow["\epsilon", from=2-2, to=2-3].
	\end{tikzcd}\]
	In other words, we have extended our stacks from $(\Spd \bE)^I$ to $(\Spd \bOE)^I$. We can then use iterated nearby cycles to extract the monodromy, as follows.

	We are interested in the $(I_E)^I$-action on $p_!\epsilon^*(\Sat(V) \boxtimes \nabla_w(\Reg_{T^{wF}}))$. This monodromy can be extracted from iterated nearby cycles with respect to  the identity map $(\Spd \bOE)^I \to (\Spd \bOE)^I$ applied to the cohomology sheaves. We will show that $p_!$ commutes with nearby cycles (see Section \ref{Section_pcomm} for details), so it is enough to control the monodromy on shtuka cohomology. Furthermore, by Assumption \ref{Assum_sm}, $\epsilon^*$ commutes with nearby cycles, so it suffices to control the monodromy on the local Hecke stack. On the local Hecke stack, we can compute the monodromy of the corresponding sheaves by reducing to the torus case, using twisted hyperbolic localization (see Section \ref{Section_monodromy} for details).

	\section{$p_!\Psi=\Psi p_!$ }\label{Section_pcomm}
	In this section, we prove that the map \eqref{Eq_commute?} is an isomorphism (Theorem \ref{Thm_pPsi}).

	\subsection{The statement}
	Recall the following general situation. Let $S$ be a curve with a closed point $s$. Let $I=\{1, 2, ..., n\}$ be a finite set. Given a morphism $f: Z \to W$  of small $v$-stacks over $S^I$ and $A \in \scrD_{\et}(Z|_{(S-s)^I})$, we ask whether the natural map
	$$f_!\Psi_1...\Psi_nA \to \Psi_1...\Psi_nf_!A$$
	is an isomorphism.
	
	We fix an element $b \in B(G)$. In our case, $S=\Spd \breve{O}_E, s=\Spd \overline{\bbF}_p$, 
	$$Z=\Sht^{{\infty=b}}_{\calG_{C}, (\Spd \bOE)^{I \sqcup \{0\}}}|_{S^I \times s}$$ 
	with the structure map $p$ to $S^I$ being the projection to the first $I$-legs. In other words,
	$$p: \Sht^{{\infty=b}}_{\calG_{C}, S^{I \sqcup \{0\}}} |_{S^I \times s} \to S^{I} \times s \to S^I,$$
	where the first map is the leg map and the second map is the projection onto the first $I$ factors. Let $W=S^I$ with the structure map to $S^I$ being the identity map.
	Let $$f=p: Z \to W.$$  Let $\eta=S-s=\Spd \breve{E}$. Set $A:=\emb_!\epsilon^*(\Sat V \boxtimes \nabla) \in \scrD_{\et}(Z|_{\eta^I})$, where $\nabla=\nabla_w(\on{Reg}_{T^{wF}})$. We ask whether the natural map
	$$p_!\Psi_1...\Psi_nA \to \Psi_1...\Psi_np_!A$$
	is an isomorphism.
	
	\begin{theorem}\label{Thm_pPsi}
		For every $V \in \Rep(\widehat{G}^I)$, the canonical map
		$$p_!\Psi_1...\Psi_n\emb_!\epsilon^*(\Sat V \boxtimes \nabla_w(\on{Reg}_{T^{wF}})) \to \Psi_1...\Psi_np_!\emb_!\epsilon^*(\Sat V \boxtimes \nabla_w(\on{Reg}_{T^{wF}}))$$
		is an isomorphism.
	\end{theorem}

	\subsection{The proof}\label{Subsection_p_proof}
	We now prove Theorem \ref{Thm_pPsi}. Write $\nabla=\nabla_w(\on{Reg}_{T^{wF}})$.
	Note that we have a canonical map
	\begin{equation}
		p_!\Psi \to \Psi p_!.
	\end{equation}
	Following \cite[Section 5]{salmon2023unipotent} and \cite[Section 4.8]{eteve2023monodromiques}, let us construct the candidate inverse
	\begin{equation}\label{Eq_candidate}
		\Psi_I p_!\emb_!\epsilon^*(\Sat V \boxtimes \nabla) \to p_!\Psi_I \emb_!\epsilon^*(\Sat V \boxtimes \nabla),
	\end{equation}
	where $I=\{1, ..., n\}$ is a finite set, $\Psi_I=\Psi_1...\Psi_n$ denotes the iterated nearby cycle. In the proof, we will use auxiliary notations $I_1 \cong I_2 \cong I_3 \cong I$.
	\begin{align}
		&\Psi_I p_!\emb_!\epsilon^*(\Sat V \boxtimes \nabla)\\
		\cong \;&\Psi_{I_1} p_!\Psi_{I_2}\emb_!\epsilon^*(\Sat (V \boxtimes 1) \boxtimes \nabla) \quad \text{Lemma \ref{Lemma_B1}}\\
		\xrightarrow{\coev}\;&\Psi_{I_1} p_!\Psi_{I_2}\emb_!\epsilon^*(\Sat (V \boxtimes (V^* \otimes V)) \boxtimes \nabla) \\
		\cong \;&\Psi_{I_1} p_!\Psi_{I_2}\Psi_{I_3}\emb_!\epsilon^*(\Sat (V \boxtimes V^* \boxtimes V) \boxtimes \nabla) \quad \text{Lemma \ref{Lemma_B2}}\\
		\xrightarrow{\can} \;&\Psi_{I_1} \Psi_{I_2}p_!\Psi_{I_3}\emb_!\epsilon^*(\Sat (V \boxtimes V^* \boxtimes V) \boxtimes \nabla) \quad \text{canonical map}\\
		\cong  \;&\Psi_I p_!\Psi_{I_3}\emb_!\epsilon^*(\Sat ((V \otimes V^*) \boxtimes V) \boxtimes \nabla) \quad \text{Lemma \ref{Lemma_nearby_nabla}}\\
		\xrightarrow{\ev}  \;&\Psi_I p_!\Psi_{I_3}\emb_!\epsilon^*(\Sat (1 \boxtimes V) \boxtimes \nabla) \\
		\cong   \;&p_!\Psi_{I_3}\emb_!\epsilon^*(\Sat V \boxtimes \nabla) \quad \text{Lemma \ref{Lemma_C_nabla}}
	\end{align}
	
	Without loss of generality we may assume that $V \in \Rep(\widehat{G}^I)$ is the irreducible representation associated to the cocharacter $\mu \in X_*(T)^I$. Let $\{E_i\}_{i \in I}$ be the fields of definition of $\mu$. In the following, we will denote 
	$\Sht^{\infty=b}_{\calG, \prod_{i \in I}\Spd \breve{E}_i/(\DY)^I, \leq \mu}$ by $\Sht_{\calG, b, I, V}$. After adding level-$\sbar$ structure, we will denote $\Sht^{\infty=b}_{\calG, \prod_{i \in I}\Spd \breve{E}_i/ (\DY)^I, \sbar, \leq \mu}$ by $\Sht_{\calG, b, I, \sbar, V}$. In the following, we base change everything to the fields of definition $E_i$ whenever needed, but we will omit $\prod_{i \in I}\Spd \breve{E}_i$ from the notation.

	\begin{lemma}\label{Lemma_B1}
		$$\emb_!\epsilon^*(\Sat V \boxtimes \nabla)
		\cong \Psi_{I_2}\emb_!\epsilon^*(\Sat (V \boxtimes 1) \boxtimes \nabla).$$
	\end{lemma}
	
	\begin{proof}
		This follows from
		$$\left(\Sht_{\calG_C, I_1 \sqcup I_2 \sqcup 0, \sbar, V \boxtimes 1}\right)|_{(\Spd \bE)^{I_1 \sqcup I_2} \times \sbar} \cong \left(\Sht_{\calG_C, I_1 \sqcup 0,  V}\right)|_{(\Spd \bE)^{I_1} \times \sbar} \times (\Spd \bE)^{I_2},$$
		under which $\emb_!\epsilon^*(\Sat(V \boxtimes 1) \boxtimes \nabla)$ corresponds to $\emb_!\epsilon^*(\Sat V  \boxtimes \nabla) \boxtimes \underline{\Lambda}_{(\Spd \breve{E})^{I_2}}$. 
	\end{proof}
	
	\begin{lemma}\label{Lemma_B2}
		$$\Psi_{I_2}\emb_!\epsilon^*(\Sat (V \boxtimes (V^* \otimes V)) \boxtimes \nabla)  \; \cong \; \Psi_{I_2}\Psi_{I_3}\emb_!\epsilon^*(\Sat (V \boxtimes V^* \boxtimes V) \boxtimes \nabla).$$
	\end{lemma}
	
	\begin{proof}
		This follows from monoidality of the twisted central functor. More precisely, by proper base change 
		on the generic fiber, we may commute $\emb_!$ and $\epsilon^*$, so it suffices to show that 
		$$\Psi_{I_2}\epsilon^*\emb_!(\Sat (V \boxtimes (V^* \otimes V)) \boxtimes \nabla)  \; \cong \; \Psi_{I_2}\Psi_{I_3}\epsilon^*\emb_!(\Sat (V \boxtimes V^* \boxtimes V) \boxtimes \nabla).$$
		Let $\iota: \frac{\Iw^0\backslash L^+\calG_{\sigma}/\Iw^0}{Ad_FT_M} \to \frac{\Iw^0\backslash LG/\Iw^0}{Ad_FT_M}$ be the map induced by the inclusion $L^+\calG_{\sigma} \subseteq LG$. It suffices to show that 
		$$\Psi_{I_2}\epsilon^*(\Sat (V \boxtimes (V^* \otimes V)) \boxtimes \iota_!\nabla)  \; \cong \; \Psi_{I_2}\Psi_{I_3}\epsilon^*(\Sat (V \boxtimes V^* \boxtimes V) \boxtimes \iota_!\nabla).$$
		Now $\epsilon: \Sht_{\calG_C, b, J \sqcup 0, N^x} \to \Hk^{\loc}_{\calG_C^0, J \sqcup 0}/Ad_FT_M$ is defined integrally over $S^J$, so we can commute the nearby cycles with $\epsilon^*$ by Assumption \ref{Assum_sm} . Therefore, it suffices to show that 
		$$\Psi_{I_2}(\Sat (V \boxtimes (V^* \otimes V)) \boxtimes \iota_!\nabla)  \; \cong \; \Psi_{I_2}\Psi_{I_3}(\Sat (V \boxtimes V^* \boxtimes V) \boxtimes \iota_!\nabla).$$
		This is exactly the monoidality of the twisted (by $\nabla$) central functor, which can be proved by an argument similar to \cite[Section 5.2, Theorem 4]{gaitsgory2001central}. 
	\end{proof}
	
	\begin{lemma}\label{Lemma_nearby_nabla}
		$$\Psi_{I_1} \Psi_{I_2}p_!\Psi_{I_3}\emb_!\epsilon^*(\Sat (V \boxtimes V^* \boxtimes V) \boxtimes \nabla) \;
		\cong \; \Psi_I p_!\Psi_{I_3}\emb_!\epsilon^*(\Sat ((V \otimes V^*) \boxtimes V) \boxtimes \nabla)$$
	\end{lemma}
	
	\begin{proof}
		Let $A:=p_!\Psi_{I_3}\emb_!\epsilon^*(\Sat (V \boxtimes V^* \boxtimes V) \boxtimes \nabla)$. The left hand side is taking iterated nearby cycle  on $A$, and the right hand side is first restricting $A$ to the (partial) diagonal and then taking nearby cycle. Therefore, it suffices to show that $A$ is locally constant over $\eta^{I_1 \sqcup I_2}$. This is Proposition \ref{Prop_lc_nabla}.
	\end{proof}
	
	\begin{proposition}\label{Prop_lc_nabla}
		Let $\eta=\Spd \breve{E}$. For every $V_i \in \Rep(\widehat{G}^{I_i}), i= 1, 2, 3$, the complex $$A:=p_!\Psi_{I_3}\emb_!\epsilon^*(\Sat (V_1 \boxtimes V_2 \boxtimes V_3) \boxtimes \nabla) \in \scrD_{\et}(\eta^{I_1 \sqcup I_2})$$ 
		is locally constant along $\eta^{I_1 \sqcup I_2}$, i.e., $A$ lies in the essential image of pullback $\scrD_{\et}([*/I_E^{I_1 \sqcup I_2}]) \to \scrD_{\et}(\eta^{I_1 \sqcup I_2})$.
	\end{proposition}
	
	\begin{proof}
		Let us first recall the maps. Denote $s=\Spd \overline{\bbF}_p$, $N=s$ the level structure. We have
		\[\begin{tikzcd}
			{\Hk^{\loc}_{\eta^{I_1 \sqcup I_2 \sqcup I_3}} \times \frac{Iw^0\backslash M_{\sigma}/U_M}{Ad_FT_M}} & {\Sht_{\calG_{\sigma}, b, \eta^{I_1 \sqcup I_2 \sqcup I_3}, N} \times^{\calG_{\sigma}(k_E)} M_{\sigma}/B_M} & {\Sht_{\calG_C, b, \eta^{I_1 \sqcup I_2 \sqcup I_3}\times s}}
			\arrow["\epsilon"', from=1-2, to=1-1]
			\arrow["\emb", from=1-2, to=1-3],
		\end{tikzcd}\]
		$$\Psi_{I_3}: \scrD_{\et}(\Sht_{\calG_C, b, \eta^{I_1 \sqcup I_2 \sqcup I_3} \times s}) \to  \scrD_{\et}(\Sht_{\calG_C, b, \eta^{I_1 \sqcup I_2} \times s^{I_3 \sqcup \{0\}}}),$$
		and 
		$$p: \Sht_{\calG_C, b, \eta^{I_1 \sqcup I_2} \times s^{I_3 \sqcup \{0\}}} \to \eta^{I_1 \sqcup I_2} \times s^{I_3 \sqcup \{0\}} \cong \eta^{I_1 \sqcup I_2}.$$
		
		Recall that we have the following Cartesian diagram (Lemma \ref{Lemma_Sht_Cart})
		\[\begin{tikzcd}
			{{\Sht_{\calG_{\sigma}, b, \eta^{I_1 \sqcup I_2 \sqcup I_3}, N} \times^{\calG_{\sigma}(k_E)} M_{\sigma}/B_M}} & {{\Sht_{\calG_C, b, \eta^{I_1 \sqcup I_2 \sqcup I_3}\times s}}} \\
			{{\Hk^{\loc}_{\calG^0_C, \eta^{I_1 \sqcup I_2 \sqcup I_3}} \times \frac{Iw^0\backslash M_{\sigma}/U_M}{Ad_FT_M}} } & {\Hk^{\loc}_{\calG_C^0, \eta^{I_1 \sqcup I_2 \sqcup I_3} \times s}/Ad_FT_M}
			\arrow["\emb", from=1-1, to=1-2]
			\arrow["\epsilon"', from=1-1, to=2-1]
			\arrow["{\epsilon_1}"', from=1-2, to=2-2]
			\arrow["{\emb_1}", from=2-1, to=2-2].
		\end{tikzcd}\]
		In the following, we denote $V_{123}=V_1 \boxtimes V_2 \boxtimes V_3$. By proper base change, 
		$$\emb_!\epsilon^*(\Sat V_{123} \boxtimes \nabla) \cong \epsilon_1^*\emb_{1, !}(\Sat V_{123} \boxtimes \nabla) \cong \epsilon_1^*(\Sat V_{123} \boxtimes  \iota_!\nabla),$$
		where $\iota: \frac{\Iw^0\backslash L^+\calG_{\sigma}/\Iw^0}{Ad_FT_M} \to \frac{\Iw^0\backslash LG/\Iw^0}{Ad_FT_M}$ is induced by the inclusion $L^+\calG_{\sigma} \subseteq LG$. So it suffices to show that
		$$p_!\Psi_{I_3}\epsilon_1^*(\Sat V_{123} \boxtimes K)$$
		is locally constant along $\eta^{I_1 \sqcup I_2}$, where $K:=\iota_!\nabla$. 
		
		Note that $\epsilon_1$ is actually defined integrally over $S^{I_3}$, i.e., there is a map
		$$\epsilon_2: \Sht_{\calG_C, b, \eta^{I_1 \sqcup I_2} \times S^{I_3} \times s} \to {\Hk^{\loc}_{\calG_C^0, \eta^{I_1 \sqcup I_2} \times S^{I_3} \times s}/Ad_FT_M}$$
		such that the pullback of $\epsilon_2$ along $\eta^{I_1 \sqcup I_2 \sqcup I_3} \times s \hookrightarrow \eta^{I_1 \sqcup I_2} \times S^{I_3} \times s$ is $\epsilon_1$. By Assumption \ref{Assum_sm}, nearby cycles commutes with $\epsilon_2^*$, so it suffices to show that 
		$$p_!\epsilon_s^*\Psi_{I_3}(\Sat V_{123} \boxtimes K)$$
		is locally constant along $\eta^{I_1 \sqcup I_2}$, where 
		$$\epsilon_s: \Sht_{\calG_C, b, \eta^{I_1 \sqcup I_2} \times s^{I_3 } \times s} \to \Hk^{\loc}_{\calG_C^0, \eta^{I_1 \sqcup I_2} \times s^{I_3} \times s}/Ad_F T_M$$ 
		is the pullback of $\epsilon_2$ along $$\eta^{I_1 \sqcup I_2} \times s^{I_3 } \times s \hookrightarrow \eta^{I_1 \sqcup I_2} \times S^{I_3} \times s.$$
		
		Note that in $\Psi_{I_3}(\Sat V_{123} \boxtimes K)$, $\Psi_{I_3}=\Psi_f$ is the iterated nearby cycle with respect to the map 
		$$f: \Hk^{\loc}_{\calG^0_C, \eta^{I_1 \sqcup I_2} \times S^{I_3} \times s} \to S^{I_3}.$$ 
		Let $\alpha$ denote the map $$\alpha: \Hk^{\loc}_{\calG_C^0, S^{I_3} \times s} \to S^{I_3}.$$
		Note that $\Psi_{I_3}(\Sat V_{123} \boxtimes K)$ corresponds to $\Sat(V_1 \boxtimes V_2) \boxtimes \Psi_{\alpha}(\Sat V_3 \boxtimes K)$ under the isomorphism 
		$$\Hk^{\loc}_{\calG_C^0, \eta^{I_1 \sqcup I_2} \times s^{I_3} \times s} \cong \Hk^{\loc}_{\calG_C^0, \eta^{I_1 \sqcup I_2}} \times \Hk^{\loc}_{\calG_C^0, s^{I_3} \times s}.$$
		
		We have a commutative diagram
		\[\begin{tikzcd}
			& {\Sht_{\calG_C, b, \eta^{I_1 \sqcup I_2} \times s^{I_3} \times s}} & {\eta^{I_1 \sqcup I_2} \times \Sht_{\calG_C, b, s^{I_3} \times s}} & {\eta^{I_1 \sqcup I_2}} \\
			{\Hk^{\loc}_{\calG_C^0, \eta^{I_1 \sqcup I_2}}} & {\Hk^{\loc}_{\calG_C^0, \eta^{I_1 \sqcup I_2} \times s^{I_3} \times s}/Ad_FT_M} & {\Hk^{\loc}_{\calG_C^0, s^{I_3} \times s}/Ad_FT_M}
			\arrow["{\tilde{p}}", from=1-2, to=1-3]
			\arrow["{\epsilon_{s, 1}}"', from=1-2, to=2-1]
			\arrow["{\epsilon_s}", from=1-2, to=2-2]
			\arrow["\pi", from=1-3, to=1-4]
			\arrow["{\epsilon_{s, 2}}", from=1-3, to=2-3]
			\arrow["{p_1}"', from=2-2, to=2-1]
			\arrow["{p_2}", from=2-2, to=2-3],
		\end{tikzcd}\]
		where $\tilde{p}$ is given by restricting the $\calG_C$-torsor near $\{p=0\}$.
		Therefore, we have
		\begin{align}
			&p_!\epsilon_s^*\Psi_{f}(\Sat V_{123} \boxtimes K)\\
			\cong \;&p_!\epsilon_s^*(\Sat(V_1 \boxtimes V_2) \boxtimes \Psi_{\alpha}(\Sat V_3 \boxtimes K))\\
			\cong \;&p_!\epsilon_s^*(p_1^*\Sat(V_1 \boxtimes V_2) \otimes p_2^*\Psi_{\alpha}(\Sat V_3 \boxtimes K))\\
			\cong \;&p_!(\epsilon_s^*p_1^*\Sat(V_1 \boxtimes V_2) \otimes \epsilon_s^*p_2^*\Psi_{\alpha}(\Sat V_3 \boxtimes K))\\
			\cong \;&\pi_!\tilde{p}_!(\epsilon_s^*p_1^*\Sat(V_1 \boxtimes V_2) \otimes \tilde{p}^*\epsilon_{s, 2}^*\Psi_{\alpha}(\Sat V_3 \boxtimes K))\\
			\cong \;&\pi_!(\tilde{p}_!\epsilon_s^*p_1^*\Sat(V_1 \boxtimes V_2) \otimes \epsilon_{s, 2}^*\Psi_{\alpha}(\Sat V_3 \boxtimes K)) \quad \text{Projection formula}
		\end{align}
		Note that $\epsilon_{s, 2}^*\Psi_{\alpha}(\Sat V_3 \boxtimes K)$ is locally constant along $\eta^{I_1 \sqcup I_2}$. Therefore, it suffices to show that $$\tilde{p}_!\epsilon_s^*p_1^*\Sat(V_1 \boxtimes V_2)$$ 
		is locally constant along $\eta^{I_1 \sqcup I_2}$.
		
		Let $\Gr_{G}^{\tw, \infty=b}$ be the twisted affine Grassmannian associated to $b \in B(G)$ (\cite[Section 23]{scholze2020berkeley}). Note that in \cite[Section 23]{scholze2020berkeley}, the element $b \in B(G)$ is fixed throughout, so they simply denote $\Gr_{G}^{\tw, \infty=b}$ by $\Gr_{G}^{\tw}$. The space $\Gr_{G}^{\tw, \infty=b}$ classifies data similar to ${\Sht_{\calG_C, b, \eta^{I_1 \sqcup I_2}}}$ but restricts everything to $\calY_S(0, \infty)$. Note that we have a commutative diagram
		\[\begin{tikzcd}
			{\Hk^{\loc}_{\calG_C^0, \eta^{I_1 \sqcup I_2}}} & {\Sht_{\calG_C, b, \eta^{I_1 \sqcup I_2} \times s^{I_3} \times s}} & {\eta^{I_1 \sqcup I_2} \times \Sht_{\calG_C, b, s^{I_3} \times s}} \\
			{\Hk^{\loc}_{G, \eta^{I_1 \sqcup I_2}}} & {\Gr^{\tw, \infty=b}_{G, \eta^{I_1 \sqcup I_2}}} & {\eta^{I_1 \sqcup I_2} \times \Bun_G}
			\arrow["{=}", from=1-1, to=2-1]
			\arrow["{\epsilon_{s, 1}}"', from=1-2, to=1-1]
			\arrow["{\tilde{p}}", from=1-2, to=1-3]
			\arrow[from=1-2, to=2-2]
			\arrow[from=1-3, to=2-3]
			\arrow["E"', from=2-2, to=2-1]
			\arrow["P", from=2-2, to=2-3]
		\end{tikzcd}\]
		where the two maps 
		$$\Sht_{\calG_C, b, s^{I_3} \times s} \to \Bun_G, \quad {\Gr^{\tw, \infty=b}_{G, \eta^{I_1 \sqcup I_2}}} \to \Bun_G$$ 
		are given by restriction near $\{p=0\} \subseteq \calY_S$ followed by descent along $Y_S \to X_S$. The right square is Cartesian by Beauville--Laszlo lemma.  It therefore suffices to show that 
		$$P_!E^*\Sat(V_1 \boxtimes V_2)$$
		is locally constant along $\eta^{I_1 \sqcup I_2}$. 
		
		Note that we have a map $\Phi: \Gr^{\tw}_{G, \eta^J} \to \Hk^{\glob, \infty=b}_{G, (\DX)^J}|_{\eta^{J}}$ given by restricting the $G$-torsor near $0$ and $\infty$. 
		Moreover, we have a commutative diagram
		\[\begin{tikzcd}
			{\Hk^{\loc, \tw}_{G, \eta^{I_1 \sqcup I_2}}} & {\Gr^{\tw, \infty=b}_{G, \eta^{I_1 \sqcup I_2}}} & {\eta^{I_1 \sqcup I_2} \times \Bun_G} \\
			{(\Hk^{\loc}_{(\DX)^{I_1 \sqcup I_2}})|_{\eta^{I_1 \sqcup I_2}}} & {(\Hk^{\glob, \infty=b}_{G, (\DX)^{I_1 \sqcup I_2}})|_{\eta^{I_1 \sqcup I_2}}} & {((\DX)^{I_1 \sqcup I_2} \times \Bun_G)|_{\eta^{I_1 \sqcup I_2}}}
			\arrow["{\Phi^{\on{pre}}}"', from=1-1, to=2-1]
			\arrow["{\pi'}"', from=1-2, to=1-1]
			\arrow["P", from=1-2, to=1-3]
			\arrow["\Phi"', from=1-2, to=2-2]
			\arrow["\cong"', from=1-3, to=2-3]
			\arrow["\pi"', from=2-2, to=2-1]
			\arrow["{\on{res}_0}", from=2-2, to=2-3]
		\end{tikzcd}\]
		where ${\Hk^{\loc, \tw}_{G, \eta^{I_1 \sqcup I_2}}}$ classifies similar data as $\Gr^{\tw}$ but we restrict everything to the completion at the legs.
		
		Recall we have Satake sheaves $\Sat_{\eta}(V) \in \scrD_{\et}(\Hk^{\loc}_{G, \eta^J})$ and $\Sat_{\DX}(V) \in \scrD_{\et}(\Hk^{\loc}_{G, (\DX)^J})$. Note that $\Phi$ is ind-proper and $\Phi$ is an isomorphism away from the twisted diagonal
		$$\tilde{\Delta}:=\{(x_1, ..., x_n, y_1, ..., y_n)\;|\; x_i=\varphi^k(y_j) \; \text{for some $i, j=1, ..., n$ and $k \in \bbZ$}\} \subseteq \eta^{I_1 \sqcup I_2}.$$  
		This implies that 
		$$\Phi_!E^*\Sat_{\eta}(V_1 \boxtimes V_2)=\pi^*\Sat_{\DX}(V_1 \boxtimes V_2)$$
		because both of them are ULA over $\eta^{I_1 \sqcup I_2}$ and they agree away from $\tilde{\Delta}$. Therefore, it suffices to show that $\on{res}_{0, !}\pi^*\Sat(V_1 \boxtimes V_2)$ is locally constant over $\eta^{I_1 \sqcup I_2}$. But this follows from \cite[Corollary IX.2.3]{fargues2021geometrization}.
	\end{proof}
	
	\begin{lemma}\label{Lemma_C_nabla}
		For the trivial representation $1 \in \Rep(\widehat{G}^{I})$ and $V \in \Rep(\widehat{G}^{I_3})$, we have 
		$$\Psi_I p_!\Psi_{I_3}\emb_!\epsilon^*(\Sat (1 \boxtimes V) \boxtimes \nabla) 
		\cong   p_!\Psi_{I_3}\emb_!\epsilon^*(\Sat V \boxtimes \nabla).$$
	\end{lemma}
	
	\begin{proof}
		The key is to note that we have an isomorphism
		$$(\Sht_{\calG_C, I \sqcup I_3 \sqcup 0, 1 \boxtimes V})|_{\Spd \breve{E}^{I} \times \Spd \breve{E}^{I_3} \times \sbar} \cong (\Sht_{\calG_C, I_3 \sqcup 0, V})|_{\Spd \breve{E}^{I_3} \times \sbar} \times \Spd \breve{E}^{I},$$
		under which the complex $\emb_!\epsilon^*(\Sat(1 \boxtimes V) \boxtimes \nabla)$ corresponds to the complex $\emb_!\epsilon^*(\Sat V \boxtimes \nabla) \boxtimes \underline{\Lambda}_{\Spd \breve{E}^I}$.
	\end{proof}
	
	Arguing as in \cite{salmon2023unipotent}, we see that the candidate inverse \eqref{Eq_candidate} is indeed the inverse of the canonical map. The proof of Theorem \ref{Thm_pPsi} is now complete.

	\section{Monodromy of twisted central sheaves}\label{Section_monodromy}
	In this section, we 
	\begin{enumerate}
		\item introduce the constant term functor for twisted sheaves on the local Hecke stack with pro-$p$ Iwahori level,
		\item introduce the twisted Wakimoto sheaves and twisted central sheaves,
		\item compute the $(I_E)^I$-action on the twisted central sheaves.
	\end{enumerate}
	Throughout the section, we work over $k=\overline{k}_E$. So that the Iwahori group scheme $\calI$ should be understood as over $\bOE$, etc.

	\subsection{Setting}
	In this section, we will work over $\bOE$ instead of $O_E$. Therefore, we use the following group-theoretic setting in place of the setup introduced in Subsection \ref{section-notation}.
	
	Let $\calG$ be a connected reductive group over $\bOE$ with generic fiber $G:=\calG_{\bE}$. Note that $\calG$ is split since $\bOE$ is strictly Henselian. Let $\calT \subseteq \calG$ be a split maximal torus over $\bOE$. Let $T:=\calT_{\bE}$.  Let $\eta: \bbG_{m, \bOE} \to \calT \subseteq \calG$ be a regular cocharacter, i.e., $\calT$ is the centralizer of $\eta$ in $\calG$.
	
	Let $\calP$ be the attractor locus of $\eta$. The generic fiber $P:=\calP_{\bE}$ is a Borel subgroup of $G$. We can and will assume that the Borel $P$ defines the same notion of dominance for characters as in Subsection \ref{section-notation}. By \cite[Lemma 4.5]{haines2021test}, we have a semidirect product 
	$$\calP=\calT \rtimes \calU,$$
	where $\calU$ is the kernel of $\calP \to \calT$.
	Let $U:=\calU_{\bE}$.

	The integral model $\calG$ of $G$ determines a hyperspecial vertex $x_0$ in the Bruhat-Tits building $\calB(G, \bE)$ of $G$. The Borel subgroup $P \subseteq G$ determines an alcove $C$ containing $x_0$ in its closure. Let $\calI:=\calG_C$ be the Iwahori group scheme corresponding to the alcove $C$. Let $\calI^0 \subseteq \calI$ be the pro-$p$ Iwahori group scheme. Let $I=L^+\calI$ and $I^0=L^+\calI^0$.
	Let $\calU'$ be the schematic closure of $U$ in $\calI$. Note that $\calU'$ is canonically isomorphic to $\calU$, so we will use $\calU$ to denote both.
	
	\subsection{Constant terms}\label{Subsection_Constant_terms}

	Following \cite[Section 2]{anschutz2023gaitsgory} and \cite[Section 5, 6]{aglr2022}, we introduce the constant term functor for twisted sheaves on  $$\Gr_{\calI^0}=LG/L^+\calI^0=LG/I^0.$$

	There is a map
	\begin{equation}
		\pi: LG/I^0 \to LG/I, 
	\end{equation}
	which is a torsor under $I/I^0 \cong \mathsf{T}=\calT_{k}$, the special fiber of $\calT$. We refer to \cite[Section 2]{anschutz2023gaitsgory} for the basics on the affine flag variety $$\Gr_{\calI}:=\Fl_{\calI}=LG/L^+\calI=LG/I.$$
	Unlike \cite{anschutz2023gaitsgory}, we will use the notation $\Gr_{\calI}$ instead of $\Fl_{\calI}$.

	We have an action of $\mathbb{G}_m$ on $LG/I$ by conjugation via the regular cocharacter $\eta$.
	This places us in the situation of hyperbolic localization as in \cite[Section IV.6]{fargues2021geometrization}.
	
	Recall that the $\mathbb{G}_m$-fixed points
	\begin{equation}
		\Gr_{\calI}^0:=\Gr_{\calI}^{\mathbb{G}_m}=\coprod\limits_{w\in \tilde{W}}wI/I=\coprod\limits_{w\in \tilde{W}}w
	\end{equation} 
	is a disjoint union of points indexed by the extended affine Weyl group $\tilde{W}$. We also have the attractor locus
	\begin{equation}
		\Gr_{\calI}^+=\coprod\limits_{w\in \tilde{W}} \calS_w,\text{ where } \calS_w:=LU\cdot w\subseteq \Gr_{\calI}.
	\end{equation}
	
	\begin{definition}
		We define the spaces $\Gr_{\calI^0}^0$ and $\Gr_{\calI^0}^+$ via pullback along $\pi$:
		\begin{equation}
			\begin{tikzcd}
				{\Gr_{\calI^0}} & {\Gr_{\calI^0}^+} & {\Gr_{\calI^0}^0} \\
				{\Gr_{\calI}} & {\Gr_{\calI}^+} & {\Gr_{\calI}^0}
				\arrow["\pi"', from=1-1, to=2-1]
				\arrow["{\tilde{q}^+}"', from=1-2, to=1-1]
				\arrow["\lrcorner"{anchor=center, pos=0.125, rotate=-90}, draw=none, from=1-2, to=2-1]
				\arrow["{\pi^+}"', from=1-2, to=2-2]
				\arrow["{\tilde{i}^+}"', from=1-3, to=1-2]
				\arrow["\lrcorner"{anchor=center, pos=0.125, rotate=-90}, draw=none, from=1-3, to=2-2]
				\arrow["{\pi^0}", from=1-3, to=2-3]
				\arrow["{q^+}", from=2-2, to=2-1]
				\arrow["{i^+}", from=2-3, to=2-2].
			\end{tikzcd}
		\end{equation}
	\end{definition}
	
	As a warning, we note that $\Gr_{\calI^0}^0$ is no longer the fixed-point locus of $\Gr_{\calI^0}$ under the $\bbG_m$-action.
	
	We have 
	\begin{equation}\label{Equation_Fl^Gm}
		\Gr_{\calI^0}^{0}=\coprod\limits_{w\in \tilde{W}}wI/I^0=\coprod\limits_{w\in \tilde{W}}w\mathsf{T}.
	\end{equation}
	Likewise,
	\begin{equation}\label{Eq_Gr^+}
		\Gr_{\calI^0}^+=\coprod\limits_{w\in \tilde{W}} \tilde{\calS_w},\text{ where } \tilde{\calS_w}:=LU\cdot w\mathsf{T}\subseteq \Gr_{\calI^0}.
	\end{equation}

	Note that we also have the map
	$$p^+: \Gr_{\calI}^+ \to \Gr_{\calI}^0.$$
	By Equations \eqref{Equation_Fl^Gm} and \eqref{Eq_Gr^+}, we have the following Cartesian diagram 
	\begin{equation}
		\begin{tikzcd}
			{\Gr_{\calI^0}^+} & {\Gr_{\calI^0}^0} \\
			{\Gr_{\calI}^+} & {\Gr_{\calI}^0}
			\arrow["{\tilde{p}^+}", from=1-1, to=1-2]
			\arrow["{\pi^+}"', from=1-1, to=2-1]
			\arrow["\lrcorner"{anchor=center, pos=0.125}, draw=none, from=1-1, to=2-2]
			\arrow["{\pi^0}", from=1-2, to=2-2]
			\arrow["{p^+}"', from=2-1, to=2-2].
		\end{tikzcd}
	\end{equation}

	\begin{definition}
		Consider the maps
		\begin{equation}
			\Hk_{\calI} = L^+\calI \backslash \Gr_{\calI} \xleftarrow{q^+} L^+\calT \backslash \Gr_{\calI}^+ \xrightarrow{p^+} L^+\calT \backslash \Gr_{\calI}^0.
		\end{equation}

		Define
		\begin{equation}
			\CT_B := p^+_! \circ q^{+, *}: \scrD_{\et}(\Hk_\calI) \to \scrD_{\et}(L^+\calT \backslash \Gr_{\calI}^0).
		\end{equation}
		
	\end{definition}

	The essential datum in the shtuka-cohomology computation is a tame local system over $\sfT$. We introduce this datum and the associated terminology.
	
	\begin{definition}[{\cite[Section 4.7]{eteve2023monodromiques}}]\label{Def_CH}
		Let $\sfT'$ be a torus over $\overline{\bbF}_p$.
		\begin{enumerate}
			\item Let $\pi_1^t(\sfT')$ be the geometric tame fundamental group of $\sfT'$ at the point $1$. It is canonically isomorphic to 
			\begin{equation}
				\pi_1^t(\sfT')=\bbX_{\bullet}(\sfT') \otimes \pi_1^t(\bbG_m)=\bbX_{\bullet}(\sfT') \otimes \hat{\bbZ}^{(p)}(1).
			\end{equation}
			
			\item Let $\on{CH}_{\Lambda}(\sfT')=\Hom(\pi_1^t(\sfT'), \Lambda^*)^{\on{tors}}$, where $\on{tors}$ denotes the group of characters of finite order. When there is no confusion about the coefficient $\Lambda$, we denote $\on{CH}(\sfT')=\on{CH}_{\Lambda}(\sfT')$.
			
			\item For each $\chi \in \on{CH}_{\Lambda}(\mathsf{T}')$, there is a corresponding tame local system $\calL_{\chi} \in \scrD_{\on{cons}}(\mathsf{T}')$, called the Kummer sheaf associated to $\chi$.
		\end{enumerate}
	\end{definition}
	
	To prepare for the next definition, we review the notion of $(\sfT, \calL_{\chi})$-equivariant sheaves following \cite[Section 2]{lusztig2020endoscopy}. Let $H$ be a connected algebraic group acting on a scheme $X$. We have the action map $\alpha: H \times X \to X$. Let $\calL$ be a character sheaf on $H$. We can define the category 
	$$\scrD_{\et}(X/(H, \calL))$$ 
	of $(H, \calL)$-equivariant sheaves on $X$. An object in $\scrD_{\et}(X/(H, \calL))$ is an object $A \in \scrD_{\et}(X)$ together with an isomorphism
	$$\alpha^*A \cong \calL \boxtimes A$$
	and higher coherence data. 
	
	Note that we can view $\calL_{\chi}$ as a character sheaf on $I=L^+\calI$ and $L^+\calT$ via pullback, which we still denote by $\calL_{\chi}$.
	
	\begin{definition}\label{Def_CT}
		Let $\sfT:=I/I^0$. For every character ${\chi} \in \on{CH}_{\Lambda}(\mathsf{T})$, we define the constant term functor $\on{CT}_B^{\chi}:=\tilde{p}^+_! \circ \tilde{q}^{+, *}$ 
		by the following composition
		$$\scrD_\et((I, \calL_{\chi})\backslash\Gr_{\calI^0}/(\mathsf{T}, \calL_{\chi})) \xrightarrow{\tilde{q}^{+, *}} \scrD_\et((L^+\calT, \calL_{\chi})\backslash\Gr_{\calI^0}^+/(\mathsf{T}, \calL_{\chi})) \xrightarrow{\tilde{p}^+_!} \scrD_\et((L^+\calT, \calL_{\chi})\backslash\Gr_{\calI^0}^0/(\mathsf{T}, \calL_{\chi}))$$
		between categories of equivariant sheaves twisted by the character $\chi$. 
	\end{definition}
	
	We can define $\Gr_{\calI^0}^-$, $\tilde{q}^-, \tilde{i}^-$, and $\tilde{p}^-$ similarly by considering the repeller locus.
	
	From now on, we mostly consider $\calI^0$ and write  $p^+, q^+, i^+$ for $\tilde{p}^+, \tilde{q}^+, \tilde{i}^+$.
	
 As in \cite[Definition IV.6.4]{fargues2021geometrization}, there is a natural map
	\begin{equation}\label{Eq_nat_map}
		p_*^{-} \circ q^{-, !} \to p^+_! \circ q^{+, *}.
	\end{equation}

	We record the following lemma which will be used repeatedly.

	\begin{lemma}[{\cite[Lemma 2.12]{anschutz2023gaitsgory}}]\label{Lemma_2.12}
		Let $\calU$ be the schematic closure of $U$ in $\calI$ and $w \in \bbX_\bullet^+$. We have
		$$IwI=L^+\calU wI=L^+\calU wL^+\calU \times ^{L^+\calU}I.$$
	\end{lemma}

	\subsection{Assumption on the constant term functor}\label{Subsection_CT_Ass}
	We now state our assumption on hyperbolic localization for twisted sheaves. 
	
	Work over $\Spd \overline{\bbF}_p$. Let $f: X \to S$ be a map of small v-stacks satisfying \cite[Hypothesis IV.6.1]{fargues2021geometrization}. In particular, we have a $\bbG_m$-action on $X$ over $S$. We have maps
	\begin{equation}
		X \xleftarrow{q^+} X^+ \xrightarrow{p^+} X^0 \quad \text{and} \quad X^+ \xleftarrow{i^+} X^0.
	\end{equation}

	Let $\sfT$ be a split torus over $\overline{\bbF}_p$. Let $\pi': \tilde{X} \to X$ be a $\sfT$-torsor. Assume that $\tilde{X}$ is equipped with a $\bbG_m$-action that commutes with the $\sfT$-action, and such that $\pi'$ is $\bbG_m$-equivariant.
	
	Define $\tilde{X}^0, \tilde{X}^+$ via the following Cartesian diagram:
	\begin{equation}\label{Diagram_twistCT1}
		\begin{tikzcd}
			{\tilde{X}} & {\tilde{X}^+} & {\tilde{X}^0} \\
			X & {X^+} & {X^0}
			\arrow["{\pi'}"', from=1-1, to=2-1]
			\arrow["{\tilde{q}^+}"', from=1-2, to=1-1]
			\arrow["\lrcorner"{anchor=center, pos=0.125, rotate=-90}, draw=none, from=1-2, to=2-1]
			\arrow[from=1-2, to=2-2]
			\arrow["{\tilde{i}^+}"', from=1-3, to=1-2]
			\arrow["\lrcorner"{anchor=center, pos=0.125, rotate=-90}, draw=none, from=1-3, to=2-2]
			\arrow[from=1-3, to=2-3]
			\arrow["{q^+}", from=2-2, to=2-1]
			\arrow["{i^+}", from=2-3, to=2-2].
		\end{tikzcd}
	\end{equation}

	Now assume that we have a $(\sfT \times \bbG_m)$-equivariant map $\tilde{p}^+: \tilde{X}^+ \to \tilde{X}^0$ making the following diagram commutative
	\begin{equation}\label{Diagram_twistedCT2}
		\begin{tikzcd}
			{\tilde{X}^+} & {\tilde{X}^0} \\
			{X^+} & {X^0}
			\arrow["{\tilde{p}^+}", from=1-1, to=1-2]
			\arrow[from=1-1, to=2-1]
			\arrow[from=1-2, to=2-2]
			\arrow["{p^+}"', from=2-1, to=2-2].
		\end{tikzcd}
	\end{equation}
	Note that Diagram \eqref{Diagram_twistedCT2} is automatically Cartesian because it induces a map from $\tilde{X}^+$ to the actual fiber product, which must be an isomorphism because every map between $\sfT$-torsors is an isomorphism.
	
	We can define 
	$$\tilde{X} \xleftarrow{\tilde{q}^-} \tilde{X}^- \xleftarrow{\tilde{i}^-} \tilde{X}^0$$
	similarly. Assume we have a $(\sfT \times \bbG_m)$-equivariant map $\tilde{p}^-: \tilde{X}^- \to \tilde{X}^0$.
	
	We can now state our assumption on hyperbolic localization for twisted sheaves. It is analogous to \cite[Theorem IV.6.5]{fargues2021geometrization}.
	As in \cite[Theorem IV.6.5]{fargues2021geometrization}, there are natural maps 
	$$(\tilde{p}^-)_*(\tilde{q}^-)^!A \to (\tilde{i}^-)^*(\tilde{q}^-)^!A \to (\tilde{i}^+)^!(\tilde{q}^+)^*A \to (\tilde{p}^+)_! (\tilde{q}^+)^*A$$
	for $A\in \scrD_{\et}(\tilde{X}/(\sfT, \calL_{\chi}), \Lambda)$.

	\begin{assumption}\label{Assum_CT}
		Let $\chi \in \on{CH}_{\Lambda}(\sfT)$. For any $\bbG_m$-equivariant object $A \in \scrD_{\et}(\tilde{X}/(\sfT, \calL_{\chi}), \Lambda)$, the above maps are isomorphisms
		$$(\tilde{p}^-)_*(\tilde{q}^-)^!A \cong (\tilde{i}^-)^*(\tilde{q}^-)^!A \cong (\tilde{i}^+)^!(\tilde{q}^+)^*A \cong (\tilde{p}^+)_! (\tilde{q}^+)^*A.$$
	\end{assumption}

	\subsection{Integral constant terms with multiple legs}
	Let $S_0=\Spd \bOE, x=\Spd k, \eta=\Spd \bE$. Let $J$ be a finite set.
	Let $S=S_0^J \times x$. Let $X=\Gr_{\calI, S}$. Let $f: X \to S$ be the natural map. After restricting to the bounded-by-$\mu$ locus, the map $f$ satisfies \cite[Hypothesis IV.6.1]{fargues2021geometrization} (see \cite[Lemma 5.1]{aglr2022}).
	
	Let $\tilde{X}=\Gr_{\calI^0, S}$, and let $\pi: \tilde{X} \to X$ be the natural map. We can define $\tilde{X}^0, \tilde{X}^+$ via pullback along $\pi$.
	To use hyperbolic localization for twisted sheaves on $\tilde{X}$ (Assumption \ref{Assum_CT}), we need to define a $\sfT$-equivariant map $\tilde{p}^+: \tilde{X}^+ \to \tilde{X}^0$. The proof of the following lemma were communicated to the author by Xiangqian Yang.
	
	\begin{lemma}\label{Lemma_vsurj}
		Let $LU$ be the loop group of $U$ over $S=S_0^J \times x$.
		\begin{enumerate}
			\item The action map $LU \times X^0 \to X^+$ is surjective as a map of v-sheaves.
			\item The action map $LU \times \tilde{X}^0 \to \tilde{X}^+$ is surjective as a map of v-sheaves.
		\end{enumerate}
	\end{lemma}
	
	\begin{proof}
		Note that (2) follows from (1) because a base change of a v-surjection is still a v-surjection.  
		
		It remains to prove (1). Note that the set of connected components
		$$\pi_0(\Gr^+_{\calI, S_0^J \times x}) = \pi_0(\Gr^+_{\calI, x^J \times x}) = \pi_0(\Gr^+_{\calI, x})$$
		is in bijection with the extended affine Weyl group $\tilde{W}$. For every $\tilde{w} \in \tilde{W}$, let $X^+_{\tilde{w}}$ be the connected component of $X^+$ corresponding to $\tilde{w}$. 
		
		If $\tilde{w} \in X_*(T) \subseteq \tilde{W}$, a similar argument as \cite[Theorem 5.2]{aglr2022} shows that $X^+_{\tilde{w}}$ is isomorphic to the corresponding connected component of $\Gr_{\calP, S_0^J \times x}$, where $\calP$ is the attractor locus of the regular cocharacter $\eta$ in $\calI$. We have 
		\begin{equation}\label{Eq_semidirect}
			\calP \cong \calU \rtimes \calT.
		\end{equation}
		The map 
		$$LU \times X_{\tilde{w}}^0 \to X_{\tilde{w}}^+$$
		corresponds to the map
		$$LU \times \Gr_{\calT, S_0^J \times x} \to \Gr_{\calP, S_0^J \times x},$$
		which is visibly v-surjective by Equation \eqref{Eq_semidirect}.
		
		Now for a general element $\tilde{w} \in \tilde{W}$. We can write $\tilde{w}=w\nu$, where $w$ is an element of the finite Weyl group and $\nu \in X_*(T) \subseteq \tilde{W}$. Left multiplication by a lift of $w$ defines an isomorphism
		$$\Gr_{\calI, S_0^J \times x} \to \Gr_{\calI, S_0^J \times x},$$
		which is $\bbG_m$-equivariant with $\bbG_m$ acting on the left hand side by conjugation through $\eta$, and with $\bbG_m$ acting on the right hand side by conjugation via $\eta^w=w\eta w^{-1}$. This induces an isomorphism
		$$X_{w\nu}^{+, \eta} \cong X_{\nu}^{+, \eta^w},$$
		where the superscript $\eta$ and $\eta^w$ record the $\bbG_m$-action. This reduces the claim to the case that $\tilde{w}=\nu \in X_*(T) \subseteq \tilde{W}$, proved in the preceding paragraph.
	\end{proof}

	By Lemma \ref{Lemma_vsurj}, the second projection map 
	$$p_2: LU \times \tilde{X}^0 \to \tilde{X}^0$$
	induces a $(\sfT \times \bbG_m)$-equivariant map
	\begin{equation}
		\tilde{p}^+: \tilde{X}^+ \to \tilde{X}^0
	\end{equation}
	making Diagram \eqref{Diagram_twistedCT2} commute.
	
	Therefore, for every $\chi \in \on{CH}(\sfT)$, under Assumption \ref{Assum_CT}, we have the twisted constant functor 
	\begin{equation}
		\on{CT}_B^{\chi}:=\tilde{p}^+_!\tilde{q}^{+, *} \cong \tilde{p}^-_*\tilde{q}^{-, !}.
	\end{equation}

	\subsection{Wakimoto sheaves}
	\label{Subsection_Wakimoto}
	
	Throughout this subsection, the Hecke stack we will use is the local Hecke stack. Following \cite[Section 3]{anschutz2023gaitsgory}, we want to study the category of sheaves on the Hecke stack $$\Hk_{\calI^0}:=\Hk^{\loc}_{\calI^0, \sbar}=I^0\backslash LG/I^0.$$ 
	In particular, we will
	\begin{enumerate}
		\item introduce twisted Wakimoto-filtered complexes in mixed characteristic (Definition \ref{Definition_Wakimoto}), 
		\item calculate the constant terms of twisted Wakimoto sheaves (Lemma \ref{Lemma_3.24'}),
		\item show that twisted central objects for convolution are twisted Wakimoto-filtered (Lemma \ref{Lemma_3.15'}).
	\end{enumerate}

	We fix a collection of lifts $\dot{w}$ of $w \in \tilde{W}$ as in \cite[Appendix 4.A]{eteve2023monodromiques}. For every element $w$ in the extended affine Weyl group $\tilde{W}$ and $\chi \in \operatorname{CH}(\mathsf{T})$, we define $\Delta_{w, \chi}$ and $\nabla_{w, \chi}$ as in \cite[Section 4.7.4]{eteve2023monodromiques}.
	
	\begin{lemma}[{\cite[Lemma 4.A.5]{eteve2023monodromiques}}]\label{Lemma_convolution}
		With the choice of $\overset{.}{w}$ for any $w \in \tilde{W}$, we have canonical isomorphisms:
		\begin{enumerate}
			\item $\Delta_{w, w'\chi'} \;*\; \Delta_{w', \chi'} \;\cong \; \Delta_{ww', \chi'}$ if $l(w)+l(w')=l(ww'),$
			\item $\nabla_{w, w'\chi'} \;*\; \nabla_{w', \chi'} \;\cong \; \nabla_{ww', \chi'}$ if $l(w)+l(w')=l(ww'),$
			\item $\Delta_{w^{-1}, w\chi}\; *\nabla_{w, \chi} \;\cong \; \Delta_{e, \chi} \;\cong \; \nabla_{w^{-1}, w\chi}\; *\Delta_{w, \chi}.$
		\end{enumerate}
	\end{lemma}
	
	As in \cite[Section 4.7.3]{eteve2023monodromiques}, let $$\scrD_{\et, \chi, \chi}:=\scrD_{\et, (\mathsf{T} \times \mathsf{T}, \mathcal{L}_{\chi}\boxtimes \mathcal{L}_{\chi})}(I^0\backslash LG/I^0, \Lambda)$$ (which is denoted $D_{Ad, \chi}$ in \textit{loc. cit}) be the category of \'etale sheaves on $I^0\backslash LG/I^0$ with support contained in finitely many $I$-double cosets and equivariant under $(\mathsf{T} \times \mathsf{T}, \mathcal{L}_{\chi}\boxtimes \mathcal{L}_{\chi})$, where the action is given by $(t, t').x=txt'$. Let 
	$$\scrD_{\chi, \chi}:=\scrD_{\on{cons}, \chi, \chi} \subseteq \scrD_{\et, \chi, \chi}$$
	be the full subcategory of constructible sheaves. Let $\scrP_{\chi, \chi}$ be the corresponding category of perverse sheaves. 
	
	Note that $\scrD_{\et, \chi, \chi}$ is the same as the category $\scrD_\et((I, \calL_{\chi})\backslash\Gr_{\calI^0}/(\mathsf{T}, \calL_{\chi}))$, which we used to define the twisted constant term functor in Definition \ref{Def_CT}.
	
	The action of the finite Weyl group on $\sfT$ induces an action of the extended affine Weyl group $\tilde{W}$ on $\mathsf{T}$. This induces an action of $\tilde{W}$ on $\on{CH}_{\Lambda}(\mathsf{T})$. Let $W_{\chi}:=\Stab_{\tilde{W}}(\chi)$ be the stabilizer of $\chi$ in $\tilde{W}$. 
	
	We can endow $\scrD_{\et, \chi, \chi}$ with a monoidal structure by convolution, as follows. Consider the following convolution diagram
	\[\begin{tikzcd}
		& {I^0\backslash LG \times^{I^0}LG/I^0} & {I^0\backslash LG \times^{I}LG/I^0} & {I^0\backslash LG/I^0} \\
		{I^0\backslash LG/I^0} && {I^0\backslash LG/I^0}
		\arrow["f", from=1-2, to=1-3]
		\arrow["{\on{pr}_1}", from=1-2, to=2-1]
		\arrow["{\on{pr}_2}"', from=1-2, to=2-3]
		\arrow["m", from=1-3, to=1-4]
	\end{tikzcd}\]
	For $A, B \in \scrD_{\et, \chi, \chi}$, the sheaf $\on{pr}_1^*A \otimes \on{pr}_2^*$B descends to a sheaf
	\begin{equation}
		A \tilde{\boxtimes} B:=  \on{Des}_f(\on{pr}_1^*A \otimes \on{pr}_2^*B)\in \scrD_{\et}(I^0\backslash LG \times ^{I}LG/I^0).
	\end{equation}
	The convolution product is defined as 
	\begin{equation}
		A * B:=m_!(A \tilde{\boxtimes}B) \in \scrD_{\et, \chi, \chi}.
	\end{equation}
	
	For $w \in \tilde{W}$, let $j_w: IwI/I^0 \subseteq \Gr_{\calI^0}$ be the locally closed embedding.
	
	\begin{lemma}
		The map
		\begin{equation}\label{Theta}
			\Theta: K_0(\scrD_{\chi, \chi}) \to \mathbb{Z}[W_{\chi}], \qquad [\mathcal{F}] \mapsto (-1)^{\dim \sfT}\sum_{w \in W_{\chi}}\sum_{n \in \mathbb{Z}}(-1)^n\Rank(\mathcal{H}^n(j^*_{w}\mathcal{F}))w
		\end{equation}
		is a ring isomorphism, with inverse ring isomorphism
		$$\psi: \mathbb{Z}[W_{\chi}] \to K_0(\scrD_{\chi, \chi}), \qquad w \mapsto (-1)^{l(w)}{[\Dwchi]}.$$
	\end{lemma}
	
	\begin{proof}
		An argument similar to \cite[Lemma 5.2.1]{achar2025central}  works. We sketch the proof here. 
		
		We first show that $\Theta$ is a well-defined isomorphism of $\mathbb{Z}$-modules.
		The point is that $K_0(\scrD_{\chi, \chi})$ is generated by $\{[\Dwchi]\; | \;
		w \in W_{\chi}\}$ and that the set $\{[\Dwchi]\; | \;
		w \in W_{\chi}\}$ is $\mathbb{Z}$-linearly independent in $K_0(\scrD_{\chi, \chi})$.
		Indeed, the latter can be seen using the maps
		$$\phi_w: K_0(\scrD_{\chi, \chi}) \to \mathbb{Z} \qquad
		[\mathcal{F}] \mapsto \sum_{n \in \mathbb{Z}}(-1)^n\Rank(\mathcal{H}^n(j^*_{w}\mathcal{F})),$$
		because
		$$\phi_{w'}([\Dwchi])=(-1)^{l(w)+\dim \sfT} \delta_{w}^{w'}$$
		for $w, w' \in W_{\chi}$.
		
		Next, we show that $\Theta$ is a ring map, i.e., for any $w_1, w_2 \in W_{\chi}$,
		$$[\Delta_{w_1, \chi} * \Delta_{w_2, \chi}]=[\Delta_{w_1w_2, \chi}].$$ Indeed, one can show more generally that for any $w_1, w_2 \in \tilde{W}$ and any $\chi$,
		$$[\Delta_{w_1, w_2\chi}*\Delta_{w_2, \chi}]=[\Delta_{w_1w_2, \chi}]$$
		in $K_0(\scrD_{\on{cons}}(\Hk_{\calI^0}))$. This can be done using Lemma \ref{Lemma_convolution} and induction on the length of $w_2$.

	\end{proof}
	
	\begin{lemma}\label{Lemma_perverse}
		Let $w_1, w_2$ be in $W_{\chi}$. Then $\Delta_{w_1, \chi}*\nabla_{w_2, \chi}, \nabla_{w_1, \chi}*\Delta_{w_2, \chi} \in \Perv_{\chi, \chi}$.
	\end{lemma}
	
	\begin{proof}
		A similar argument to \cite[Lemma 4.1.7]{achar2025central} works. 
	\end{proof}
	
	Recall we have the natural map $\pi: \Gr_{\calI^0} \to \Gr_{\calI}$. Also recall that $\Gr_{\calI}$ is denoted by $\Fl_{\calI}$ in \cite{anschutz2023gaitsgory}.
	
	\begin{lemma}\label{Lemma_support}
		Let $w_1, w_2 \in W_{\chi}$. 
		\begin{enumerate}
			\item The perverse sheaves $\Delta_{w_1, \chi}*\nabla_{w_2, \chi}, \nabla_{w_1, \chi}*\Delta_{w_2, \chi}$ are both supported on $\pi^{-1}(\Gr_{\calI, \leq w_1w_2})$.
			\item The perverse sheaves $\Delta_{w_1, \chi}*\nabla_{w_2, \chi}, \nabla_{w_1, \chi}*\Delta_{w_2, \chi}$ restrict to $\underline{\Lambda} \boxtimes \mathcal{L}_{\chi}[l(w_1w_2)+\dim \sfT]$  on $\pi^{-1}(\Gr_{\calI, w_1w_2}) \cong \mathbb{A}^{l(w_1w_2)} \times \mathsf{T}$. 
		\end{enumerate}

	\end{lemma}
	
	\begin{proof}
		The proof is similar to the equal characteristic case, cf. \cite[Lemma 4.1.10]{achar2025central}. Recall that we have the map $\Theta$ as in (\ref{Theta}), sending $[\Delta_{w, \chi}]=[\nabla_{w, \chi}]$ to $(-1)^{l(w)} w \in \mathbb{Z}[W_{\chi}].$
		
		Let $\mathcal{F}=\Delta_{w_1, \chi}*\nabla_{w_2, \chi}$. Let $w \in W_{\chi}$ such that $\pi^{-1}(\Gr_{\calI, w})$ is open in $\supp(\mathcal{F})$. Since $\mathcal{F}$ is perverse (Lemma \ref{Lemma_perverse}), $j_w^*\mathcal{F}$ is concentrated in one degree (\cite[Lemma 3.1.6]{achar2021perverse}). Hence the coefficient of $w$ in $\Theta([\mathcal{F}])$ is nonzero. Now since $\Theta([\mathcal{F}])=(-1)^{l(w_1w_2)} w_1w_2$, we must have $w=w_1w_2$. This finishes the proof of (1).
		
		For (2), note that by (1), the sheaf $j_{w_1w_2}^*(\mathcal{F})$ on $\pi^{-1}(\Gr_{\calI, w_1w_2}) \cong \mathbb{A}^{l(w_1w_2)} \times \mathsf{T}$ is concentrated in one degree, which must be $l(w_1w_2)+\dim \sfT$ by perversity. By $(\mathsf{T} \times \mathsf{T}, \mathcal{L}_{\chi}\boxtimes \mathcal{L}_{\chi})$-equivariance, we must have
		$$j_{w_1w_2}^*(\mathcal{F}) \cong \Lambda^{\oplus m} \boxtimes \mathcal{L}_{\chi}[l
		(w_1w_2)+\dim \sfT]$$ for some $m \in \mathbb{Z}_{\geq 0}$. 
		Now $\Theta([\mathcal{F}])=(-1)^{l(w_1w_2)} w_1w_2$ implies that $m=1$. Hence
		$$j_{w_1w_2}^*(\mathcal{F}) \cong \Lambda \boxtimes \mathcal{L}_{\chi}[l
		(w_1w_2)+\dim \sfT].$$
	\end{proof}
	
	Analogously to \cite[Prop 3.7]{anschutz2023gaitsgory}, we have the following lemma.
	
	\begin{lemma}\label{Lemma_finite_support}
		Let $\mathcal{F} \in \scrD_{\chi, \chi}$. There exists a finite subset $A_{\mathcal{F}} \subseteq \tilde{W}$ such that for any $w \in W_\chi$,
		\begin{enumerate}
			\item $$\supp(\Delta_{w, \chi}*\mathcal{F}) \subseteq \pi^{-1}(\cup_{w' \in w.A_{\mathcal{F}}}\Gr_{\calI, w'}).$$
			\item $$\cosupp(\nabla_{w, \chi}*\mathcal{F}) \subseteq \pi^{-1}(\cup_{w' \in w.A_{\mathcal{F}}}\Gr_{\calI, w'}).$$
			\item  $$\supp(\mathcal{F}*\Delta_{w, \chi}) \subseteq \pi^{-1}(\cup_{w' \in A_{\mathcal{F}}.w}\Gr_{\calI, w'}).$$
			\item  $$\cosupp(\mathcal{F}*\nabla_{w, \chi}) \subseteq \pi^{-1}(\cup_{w' \in A_{\mathcal{F}}.w}\Gr_{\calI, w'}).$$
		\end{enumerate}
		
	\end{lemma}
	
	\begin{proof}
		We only prove (3), as the remaining cases are similar. Let $m_I$ be the multiplication for $\Gr_{\calI}$ and $m_{I^0}$ the multiplication for $\Gr_{\calI^0}$. Let $X \subseteq \Gr_{\calI}$ be a closed finite union of $I$-orbits such that $\supp( \mathcal{F}) \subseteq \pi^{-1}(X)$. By \cite[Lemma 2.7]{anschutz2023gaitsgory}, there is a finite subset $S_X \subseteq \tilde{W}$ such that for any $w \in \tilde{W}$,
		$$m_{I}(X \tilde{\times} \Gr_{\calI, w}) \subseteq \cup_{y \in S_X.w}\Gr_{\calI, y}.$$
		Therefore,
		$$m_{I^0}(\pi^{-1}(X) \tilde{\times} \pi^{-1}(\Gr_{\calI, w})) \subseteq \pi^{-1}(\cup_{y \in S_X.w}\Gr_{\calI, y}).$$
		By proper base change, 
		$$\supp(\mathcal{F} * \Delta_{w, \chi}) \subseteq m_{I^0}(\pi^{-1}(X) \tilde{\times} \pi^{-1}(\Gr_{\calI, w})) \subseteq \pi^{-1}(\cup_{y \in S_X.w}\Gr_{\calI, y}).$$
	\end{proof}

	Next, we introduce the twisted Wakimoto sheaves following \cite[Section 4.7.4]{eteve2023monodromiques}. We recall the features needed below.
	
	Let us fix a uniformizer $\pi_E$ of $E$. For any $\nu \in \bbX_\bullet$,  we have an element 
	$$\nu(\pi_E) \in T(\breve{E}) \subseteq G(\breve{E})=LG(k).$$ This induces a canonical isomorphism
	\begin{equation}
		I\nu I/I^0 \xrightarrow{\beta(\nu)} I\nu I/I \times I/I^0 \quad [x^0\nu(\pi_E)y] \mapsto ([x^0\nu(\pi_E)y],[y]),
	\end{equation}
	where $x^0 \in I^0$ and $y \in I$ (here we used $I\nu I=I^0\nu I$).
	
	\begin{definition}\label{Definition_Wakimoto}
		For every $\chi \in \operatorname{CH}(\mathsf{T})$ and $\nu \in \bbX_\bullet \cap W_{\chi} =\bbX_{\bullet}$, we define the \textbf{twisted Wakimoto sheaves}
		$J_{\nu, \chi}$ as follows:
		\begin{enumerate}
			\item For $\nu \in \bbX_\bullet^+$, define
			$$J_{\nu, \chi}:=\nabla_{\nu, \chi}=_{\operatorname{def}}  j_{\nu, *}\left(\beta(\nu)^*(\Lambda \boxtimes \mathcal{L}_{\chi})\right)[l(\nu)+\dim \sfT] \in \scrD_{\chi, \chi}.$$
			\item We extend the above definition linearly to define $J_{\nu, \chi}$ for any $\nu \in \bbX_\bullet $ as in \cite[Section 3.2]{anschutz2023gaitsgory}. In other words, $J_{\nu, \chi}$ is characterized by (1) and $$J_{\nu_1+\nu_2, \chi} \cong J_{\nu_1, \chi}*J_{\nu_2, \chi}$$
			for every $\nu_1, \nu_2 \in {\bbX_{\bullet}}(T).$
		\end{enumerate}
		
		\begin{remark}
			In particular, we have for $\nu \in -\bbX_\bullet^+$ anti-dominant, $J_{\nu, \chi} \cong \Delta_{\nu, \chi}$. And for $\nu=\nu_1-\nu_2$ with $\nu_1, \nu_2 \in \bbX_\bullet^+$, we have $J_{\nu, \chi} \cong \nabla_{\nu_1, \chi} * \Delta_{-\nu_2, \chi}$.
		\end{remark}
		
	\end{definition}
	
	Moreover, it upgrades to a functor which we still denote by $J_{\nu, \chi}$:
	\begin{equation}\label{Eq_Wak}
		J_{\nu, \chi}: \scrD
		_{\on{cons}}(*/{\mathsf{T}}) \cong \scrD_{\chi, \chi}(I^0 \backslash I/I^0) \to \scrD_{\chi, \chi}, \quad M \mapsto J_{\nu, \chi}*M.
	\end{equation}
	We stress that in the definition of the convolution on $\scrD_{\chi, \chi}$, the construction includes a descent step from $I^0\backslash LG \times^{I^0}LG/I^0$ to $I^0\backslash LG \times^{I}LG/I^0$.
	We define 
	$$J_{\Omega, \chi}:=\oplus_{\nu \in \Omega}J_{\nu, \chi}: \scrD_{\on{cons}}([\Omega/{\mathsf{T}}]) \to \scrD_{\chi, \chi}$$ 
	for any subset $\Omega \subseteq \bbX_\bullet $. We write $J_{\chi}$ for $J_{{\bbX_{\bullet}} , \chi}$.
	
	\begin{lemma}
		\begin{enumerate}
			\item For any $\nu \in {\bbX_{\bullet}} $, $J_{\nu, \chi}$ is $t$-exact for the perverse $t$-structure.
			\item For any $\nu \in {\bbX_{\bullet}}$, $J_{\nu, \chi}$ is supported on $\pi^{-1}(\Gr_{\calI, \leq \nu})$.
			\item For any $\mu, \nu \in {\bbX_{\bullet}} $, there exists a canonical isomorphism $J_{\mu, \chi}*J_{\nu, \chi} \cong J_{\mu+\nu, \chi}$.
			\item For any $\mu, \nu \in {\bbX_{\bullet}} $ with $\nu \; \nleq \;\mu$, we have $\on{RHom}(J_{\mu, \chi}, J_{\nu, \chi})=0.$
		\end{enumerate}
	\end{lemma}
	
	\begin{proof}
		The proof is similar to the non-twisted case, cf. \cite[Lemma 3.12]{anschutz2023gaitsgory}.
		\begin{enumerate}
			
			\item This follows from Lemma \ref{Lemma_perverse}.
			
			\item This follows from Lemma \ref{Lemma_support}.
			
			\item When $\mu, \nu$ dominant and lengths add, this follows from Lemma \ref{Lemma_convolution} (2). For general $\mu, \nu$, it follows from the definition of Wakimoto sheaves and Lemma \ref{Lemma_convolution} (3).
			
			\item This follows from the closure relation of the Schubert stratification on $\Gr_{\calI}$.

		\end{enumerate}
	\end{proof}
	
	\begin{definition}
		For $\chi \in \on{CH}(\sfT)$, define the \textbf{$\chi$-Wakimoto category} as the full subcategory $\on{Wak}_{\chi}:=\langle J_{\chi} \rangle$ of $\scrD_{\et, \chi, \chi}$  generated by the essential image of $J_{\chi}$ under cones and shifts. An object $\mathcal{F} \in {\on{Wak}}_{\chi}$ is called \textbf{$\chi$-Wakimoto filtered}.
	\end{definition}
	
	We have the following simple criterion for determining whether an object of $\scrD_{\chi, \chi}$ lies in ${\on{Wak}}_{\chi}$ analog to \cite[Prop 3.15]{anschutz2023gaitsgory}.
	
	\begin{lemma}\label{Lemma_3.15'}
		Let $\mathcal{F} \in \scrD_{\chi, \chi}$ ($\mathcal{F}$ is constructible). 
		\begin{enumerate}
			\item $\mathcal{F}$ is $\chi$-Wakimoto filtered if and only if 
			$$\supp(J_{-\nu, \chi} * \mathcal{F}) \subseteq \pi^{-1}(\{t_{-\mu}\;|\; \mu \in {\bbX_\bullet^+}\})$$ for all $\nu \gg 0$, the cocharacters that are sufficiently dominant.
			\item If $\mathcal{F}$ satisfies $J_{\nu, \chi} * \mathcal{F} \cong \mathcal{F} * J_{\nu, \chi}$ for all $\nu \in \bbX_{\bullet}$, then $\calF$ is $\chi$-Wakimoto filtered.
		\end{enumerate}
	\end{lemma}
	
	\begin{proof}
		A similar argument as \cite[Prop 3.15]{anschutz2023gaitsgory} works.
		\begin{enumerate}
			\item Assume $\mathcal{F}$ is $\chi-$Wakimoto filtered, then without loss of generality $\mathcal{F}=J_{\nu', \chi}$ for some $\nu' \in \bbX_{\bullet}$.
			If now $\nu \gg 0$, then
			$$J_{-\nu, \chi} * \mathcal{F} \cong \Delta_{-\nu+\nu', \chi}$$
			and the support claim follows.

			Conversely, assume that 
			$$\supp(J_{-\nu, \chi} * \mathcal{F}) \subseteq \pi^{-1}(\{t_{-\mu}\;|\; \mu \in {\bbX_\bullet^+}\})$$ for all $\nu \gg 0$. Then $J_{-\nu, \chi} * \mathcal{F}$ lies in the subcategory $\langle \Delta_{t_{-\mu}, \chi}\;|\;\mu \in {\bbX_\bullet^+} \rangle$. But $\langle \Delta_{t_{-\mu}, \chi}\;|\;\mu \in {\bbX_\bullet^+}\rangle \subseteq {\on{Wak}}_{\chi}$. So $J_{-\nu, \chi}*\mathcal{F} \in {\on{Wak}}_{\chi}$, hence $\mathcal{F} \in {\on{Wak}}_{\chi}$.
			
			\item By Lemma \ref{Lemma_finite_support}, for all $\nu \in \bbX_{\bullet}$,
			$$\supp(\Delta_{\nu, \chi}*\mathcal{F})=\supp(\mathcal{F}*\Delta_{\nu, \chi}) \subseteq \pi^{-1}(\nu A_{\mathcal{F}} \cap A_{\mathcal{F}} \nu).$$
			Now for $\nu \gg 0$, $J_{-\nu, \chi}=\Delta_{-\nu, \chi}$. Therefore,
			$$\supp(J_{-\nu, \chi}*\mathcal{F}) \subseteq \pi^{-1}((-\nu) A_{\mathcal{F}} \cap A_{\mathcal{F}} (-\nu)) \subseteq \pi^{-1}(-{\bbX_\bullet^+})$$
			for $\nu \gg 0$. Hence we conclude that $\mathcal{F} \in {\on{Wak}}_\chi$ by (1).
		\end{enumerate}
	\end{proof}
	
	\begin{lemma}If $\Omega \subseteq {\bbX_{\bullet}} $ is a lower set, the inclusion ${\on{Wak}}_{\Omega, \chi} \to {\on{Wak}}_{\chi}$ has a right adjoint ${\on{Wak}}_{\chi} \to {\on{Wak}}_{\Omega, \chi}, \mathcal{F} \mapsto \mathcal{F}_{\Omega}$ such that the cone $\mathcal{F}'$ of the adjunction unit $\mathcal{F}_{\Omega} \to \mathcal{F}$ lies in ${\on{Wak}}_{\chi}$ and satisfies $\mathcal{F}'_{\Omega}=0$.
		
	\end{lemma}
	
	\begin{proof}
		A similar argument as \cite[Lemma 3.17]{anschutz2023gaitsgory} works. It follows by induction on the length of a Wakimoto filtration. 
	\end{proof}
	
	We can now define twisted Wakimoto grading as in \cite[Lemma 3.18]{anschutz2023gaitsgory}.
	
	\begin{definition}
		For any $\mu \in {\bbX_{\bullet}} $, we define the endofunctor
		$$\on{gr}_{\mu}: {\on{Wak}}_{\chi} \to {\on{Wak}}_{\chi}, \quad \mathcal{F} \mapsto \text{cone}(\mathcal{F}_{< \mu} \to \mathcal{F}_{\leq \mu}).$$
		We also define $\on{gr}:=\oplus_{\mu \in {\bbX_{\bullet}} }\on{gr}_{\mu}: {\on{Wak}}_{\chi} \to {\on{Wak}}_{\chi}$.
	\end{definition}
	
	Note that $\on{gr}_{\nu}(\mathcal{F})$ lies in the essential image of the functor $J_{\nu, \chi}$. Since convolution with $J_{\nu, \chi}$ is invertible (with inverse isomorphic to convolution with $J_{-\nu, \chi}$), we see that
	$$J_{\nu, \chi}: \scrD_{\on{cons}}([*/{\mathsf{T}}]) \to {\on{Wak}_{\chi}}$$
	is fully faithful. In particular, we can lift $\on{gr}_{\nu}(\mathcal{F})$ to an element of $\scrD_{\on{cons}}([*/{\mathsf{T}}])$, uniquely up to a unique isomorphism.  We make the following definition as in \cite[Definition 3.19]{anschutz2023gaitsgory}.
	
	\begin{definition}
		Let $\mathcal{F} \in {\on{Wak}}_{\chi}$, define
		$$\on{Grad}^{\chi}_{\nu}(\calF) \in \scrD_{\on{cons}}([*/{\mathsf{T}}])$$
		to be the canonical object such that $J_{\nu, \chi}(\on{Grad}^{\chi}_{\nu}(\calF))$ identifies with $\on{gr}_{\nu}\mathcal{F}$.
	\end{definition}

	\subsection{Constant terms vs Wakimoto gradings}\label{Subsection_CT_Wak}
	
	The goal of this section is to prove Lemma \ref{Lemma_3.24'} -- Wakimoto gradings are computed by stalks of constant terms. The main lemma is Lemma \ref{Lemma_3.23'} -- an inductive formula for constant terms applied to $J_{\nu, \chi}*\mathcal{F}$ for some $\mathcal{F} \in \scrD_{\et, \chi, \chi}$.
	
	\subsubsection{Statement of Lemma \ref{Lemma_3.23'}}
	
	Recall from Equation \eqref{Equation_Fl^Gm} that 
	$$\Gr_{\calI^0}^0=\coprod\limits_{w\in \tilde{W}}w\mathsf{T}.$$
	For any $w \in \tilde{W}$, let $\iota_{w}: w\mathsf{T} \to \Gr_{\calI^0}^0$ denote the inclusion. We have the functor $\on{CT}_{B^-}^{\chi}(-)_{w}:=\iota_{w}^* \circ \on{CT}^{\chi}_{B^-}$:
	$$\scrD_{\et, \chi, \chi} \xrightarrow{\on{CT}_{B^-}^{\chi}} \scrD_{\et}((L^+\calT, \chi)\backslash \Gr_{\calI^0}^0/(\mathsf{T}, \chi)) \xrightarrow{\iota_{w}^*} \scrD_{\et}((L^+\calT, \chi)\backslash  w\mathsf{T}/(\mathsf{T}, \chi)).$$
	
	Note that 
	$$\scrD_{\et}((L^+\calT, \chi)\backslash w \mathsf{T}/(\mathsf{T}, \chi)) \cong \scrD_{\et}((\mathsf{T}, \chi)\backslash w\mathsf{T}/(\mathsf{T}, \chi))$$ 
	because $\ker({L^+\calT} \to T)$ is pro-unipotent by \cite[Proposition VI.4.1]{fargues2021geometrization}. Moreover, for any $\nu \in \bbX_{\bullet} $ and $w \in \tilde{W}$ we have an identification 
	$$\nu(\pi_E)_*: \scrD_{\et}((\mathsf{T}, \chi)\backslash w\mathsf{T}/(\mathsf{T}, \chi)) \cong \scrD_{\et}((\mathsf{T}, \chi)\backslash \nu w\mathsf{T}/(\mathsf{T}, \chi)).$$
	
	\begin{lemma}[{\cite[Proposition 3.23]{anschutz2023gaitsgory}}]\label{Lemma_3.23'}
		For every $\nu \in {\bbX_{\bullet}} $, $w \in \tilde{W}$, and $\mathcal{F} \in \scrD_{\et, \chi, \chi}$, we have
		$$\on{CT}_{B^-}^{\chi}(J_{\nu, \chi} * \mathcal{F})_{\nu w} \cong \nu({\pi_E})_*\on{CT}_{B^-}^{\chi}(\mathcal{F})_w[d_{\nu}],$$
		where $d_{\nu}=\langle 2\rho, \nu \rangle +\dim \sfT$.
	\end{lemma}
	
	We will prove Lemma \ref{Lemma_3.23'} in the next subsubsection. Before that, we need to introduce some notation and preliminary lemmas.
	
	We have the natural projections $f, f_{\nu}$ and the multiplication maps $m, m_{\nu}$:
	\begin{equation}\label{Diagram_f}
		\begin{tikzcd}
			{LG \times^{I^0}LG/I^0} & {LG \times^{I}LG/I^0} & {LG/I^0} \\
			{I\nu I \times^{I^0}LG/I^0} & {I\nu I \times^{I}LG/I^0} & {LG/I^0}
			\arrow["f", from=1-1, to=1-2]
			\arrow["m", from=1-2, to=1-3]
			\arrow[from=2-1, to=1-1]
			\arrow["{f_{\nu}}", from=2-1, to=2-2]
			\arrow[from=2-2, to=1-2]
			\arrow["{m_{\nu}}", from=2-2, to=2-3]
			\arrow[from=2-3, to=1-3]
		\end{tikzcd}
	\end{equation}

	By Lemma \ref{Lemma_2.12}, for $\nu \in {\bbX_\bullet^+}$, we have isomorphisms $\alpha$ and $\alpha_0$ that fit into a commutative diagram 
	\begin{equation}\label{Diagram_phi}
		\begin{tikzcd}
			{I\nu I \times^{I^0}LG/I^0} & {I\nu I \times^{I}LG/I^0} \\
			{{L^+\calU}\nu({\pi_E}){L^+\calU} \times^{{L^+\calU}}I \times^{I^0}LG/I^0} & {{L^+\calU}\nu({\pi_E}) {L^+\calU} \times^{{L^+\calU}}LG/I^0}
			\arrow["{f_{\nu}}", from=1-1, to=1-2]
			\arrow["{\alpha_0}"', from=2-1, to=1-1]
			\arrow["{\phi_{\nu}}"', from=2-1, to=2-2]
			\arrow["\alpha"', from=2-2, to=1-2],
		\end{tikzcd}
	\end{equation}
	where $\phi_{\nu}$ is induced from the multiplication map
	$$I \times^{I^0} LG/I^0 \to LG/I^0, \quad (a, b) \mapsto ab.$$

	The following lemma is well-known.
	
	\begin{lemma}
		Given the lift $\nu({\pi_E}) \in G(\bE)$ of $\nu \in \bbX_{\bullet}$, we have an isomorphism
		\begin{equation}
			\beta_{\nu}: I\nu I/I^0 \xrightarrow{\cong} I\nu I/I \times I/I^0 \quad [x^0\nu({\pi_E})y] \mapsto ([x^0\nu({\pi_E})y], [y]),
		\end{equation}
		where $x^0 \in I^0$ and $y \in I$. 
	\end{lemma}

	Recall that $j_{\nu}: I\nu I/I^0 \to \Gr_{\calI^0}$ is the locally closed embedding. 
	
	In the next lemma, we use the following notation: Let $G'$ be a group object. Let $X'$, $Y'$ be two stacks with $G'$-actions. Let $A' \in \scrD_{\et}(X')$ and $B' \in \scrD_{\et}(Y')$ such that $A' \boxtimes B'$ is equivariant with respect to the diagonal action of $G'$ on $X' \times Y'$. Then we denote $A' \tilde{\boxtimes}_{G'} B'$ to be the unique object in $\scrD_{\et}(X' \times^{G'} Y')$ whose pulback is $A' \boxtimes B' \in \scrD_{\et}(X' \times Y')$.
	
	\begin{lemma}\label{Lemma_beta_revised}
		Let $\nu$ in ${\bbX_\bullet^+}$. Then Lemma \ref{Lemma_2.12} induces an isomorphism
		$$I\nu I/I^0 \xrightarrow[]{\cong}{L^+\calU}\nu({\pi_E}){L^+\calU} \times^{{L^+\calU}} I/I^0,\; [u\nu({\pi_E})x] \mapsto ([u\nu({\pi_E})], [x])$$ where $u \in {L^+\calU}$ and $x \in I$, such that $j_{\nu}^*J_{\nu, \chi}$ corresponds to $\underline{\Lambda} \tilde{\boxtimes}_{{L^+\calU}}\mathcal{L}_{\chi}[d_{\nu}].$ 
	\end{lemma}
	
	\begin{proof}
		This follows from the definition of $J_{\nu, \chi}$ and the following commutative diagram
		\[\begin{tikzcd}
			{I\nu I/I^0} & {I\nu I/I \times I/I^0} \\
			{{L^+\calU}\nu({\pi_E}){L^+\calU} \times^{{L^+\calU}} I/I^0 } \\
			{{L^+\calU}\backslash I/I^0} & {I/I^0}
			\arrow["\beta_{\nu}", from=1-1, to=1-2]
			\arrow["\cong", from=1-1, to=2-1]
			\arrow[from=1-2, to=3-2]
			\arrow[from=2-1, to=3-1]
			\arrow[from=3-2, to=3-1]
		\end{tikzcd}\]
	\end{proof}
	
	Let us recall the notation $\on{Des}$ from Section \ref{section-notation}.
	
	\begin{lemma}\label{Lemma_Descent_revised}
		Let $\nu \in \bbX_{\bullet}^+$ dominant. Let $\calF \in \scrD_{\et, \chi, \chi}$. We have
		\begin{equation}
			\on{Des}_{\phi_{\nu}}(\underline{\Lambda} \tilde{\boxtimes}_{{L^+\calU}}\mathcal{L}_{\chi} \tilde{\boxtimes}_{I^0} \mathcal{F}) \cong \underline{\Lambda}\tilde{\boxtimes}_{{L^+\calU}}\mathcal{F}
		\end{equation}
		as a sheaf on $ I\nu I\times^ILG/I^0 \xrightarrow[\cong]{\alpha^{-1}} {L^+\calU}\nu({\pi_E}){L^+\calU} \times^{{L^+\calU}}LG/I^0$ .
	\end{lemma}
	
	\begin{proof}
		Recall that the map 
		$$\phi_{\nu}: {{L^+\calU}\nu({\pi_E}){L^+\calU} \times^{{L^+\calU}}I \times^{I^0}LG/I^0} \to  {{L^+\calU}\nu({\pi_E}) {L^+\calU} \times^{{L^+\calU}}LG/I^0}$$ 
		is induced from the multiplication map
		$$\mu: I \times^{I^0} LG/I^0 \to LG/I^0, \quad (a, b) \mapsto ab.$$
		
		This follows from
		\begin{equation}
			\mu^*\mathcal{E} \cong \mathcal{L}_{\chi} \tilde{\boxtimes}_{I^0} \mathcal{E},
		\end{equation} 
		for every $\mathcal{E} \in \scrD_{\et, \chi, \chi}$.
	\end{proof}
	
	Recall that we have the maps $q^+: \Gr_{\calI^0}^+ \to \Gr_{\calI^0}$ and $p^+: \Gr_{\calI^0}^+ \to \Gr_{\calI^0}^0$ (Section \ref{Subsection_Constant_terms}). For any $\nu' \in \tilde{W}$, we denote $q_{\nu'}: LU\nu' I/I^0 \to LG/I^0$ the restriction of $q^+$ to $LU\nu' I/I^0$. Define $p_{\nu'}: LU\nu'I/I^0 \to \nu'I/I^0$ the base change of $p^+$ along $\nu'I/I^0 \to \Gr_{\calI^0}^0$.

	\begin{lemma}\label{Lemma_Cartesian3}
		For every $w \in \tilde{W}$, we have the following Cartesian diagram
		\[\begin{tikzcd}
			{{L^+\calU}\nu({\pi_E}){L^+\calU} \times^{{L^+\calU}}LU.wI/I^0} && {LU.\nu wI/I^0} \\
			{{L^+\calU}\nu({\pi_E}){L^+\calU} \times^{{L^+\calU}}LG/I^0} & {I\nu I \times^ILG/I^0} & {LG/I^0}
			\arrow["{m_\nu'}", from=1-1, to=1-3]
			\arrow["{q_{\nu w}'}", from=1-1, to=2-1]
			\arrow["{q_{\nu w}}"', from=1-3, to=2-3]
			\arrow["\alpha", from=2-1, to=2-2]
			\arrow["{m_{\nu}}", from=2-2, to=2-3]
		\end{tikzcd}\]
	\end{lemma}
	
	\begin{proof}
		Recall that $\alpha$ is an isomorphism. We notice that for every $w \in \tilde{W}$, the multiplication map $m_{\nu}$ sends ${{L^+\calU}\nu({\pi_E}){L^+\calU}  \times^{{L^+\calU}}LU.wI/I^0}$ to ${LU.\nu wI/I^0}$ because $\nu(\pi_E)$ normalizes $LU$. For distinct $w \in \tilde{W}$, the subspaces ${LU.\nu wI/I^0}$ are disjoint. Therefore, ${{L^+\calU}\nu({\pi_E}){L^+\calU}  \times^{{L^+\calU}}LU.wI/I^0}$ is exactly the preimage of ${LU.\nu wI/I^0}$.
	\end{proof}
	
	\begin{lemma}\label{Lemma_Cartesian4}
		We have the following Cartesian diagram
		\[\begin{tikzcd}
			{{L^+\calU}\nu({\pi_E}){L^+\calU} \times^{{L^+\calU}}LU.wI/I^0} & {{L^+\calU}\backslash LU.wI/I^0} \\
			{{L^+\calU}\nu({\pi_E}){L^+\calU} \times^{{L^+\calU}}LG/I^0} & {{L^+\calU}\backslash LG/I^0}
			\arrow["{p'_2}", from=1-1, to=1-2]
			\arrow["{q_{\nu w}'}", from=1-1, to=2-1]
			\arrow["{q_w}"', from=1-2, to=2-2]
			\arrow["{p_2}", from=2-1, to=2-2]
		\end{tikzcd}\]
	\end{lemma}
	
	\begin{proof}
		Here $p_2$ and $p_2'$ are the second projection maps. The diagram is clearly Cartesian.
	\end{proof}
	
	\begin{lemma}\label{Lemma_isom2}
		Let ${\calU_{\nu}}:={L^+\calU} \cap \nu({\pi_E}){L^+\calU}\nu({\pi_E})^{-1}$. We have an isomorphism
		\begin{equation}
			{	{L^+\calU}\nu({\pi_E}){L^+\calU} \times^{{L^+\calU}}LU.wI/I^0} \;\; \xrightarrow{\gamma} \;\; {{L^+\calU} \times^{{\calU_{\nu}}} 		\nu({\pi_E})LU.wI/I^0     } 
		\end{equation}
		\begin{equation}
			{(u_1\nu({\pi_E})u_2, u_3w)}  \;\; \mapsto \;\; {(u_1, \nu({\pi_E})u_2u_3w)}
		\end{equation}
		under which $\underline{\Lambda} \tilde{\boxtimes}_{{L^+\calU}} q_w^!\mathcal{F} \in \scrD_{\et}({L^+\calU}\nu({\pi_E}){L^+\calU} \times^{{L^+\calU}}LU.wI/I^0)$ corresponds to $$\underline{\Lambda} \tilde{\boxtimes}_{{\calU_{\nu}}} \nu({\pi_E})_*q_w^!\mathcal{F} \in \scrD_{\et}({L^+\calU} \times^{{\calU_{\nu}}}
		\nu({\pi_E})LU.wI/I^0),$$
		where $\nu({\pi_E}): LU.wI/I^0 \to \nu({\pi_E})LU.wI/I^0=LU.\nu wI/I^0$ is left multiplication by $\nu({\pi_E})$.
	\end{lemma}
	
	\begin{proof}
		The isomorphism $\gamma$ is induced from the isomorphism
		$$L^+\calU \nu(\pi_E) L^+\calU/L^+\calU \cong L^+\calU/\calU_{\nu}.$$
		The rest follows by tracking the isomorphisms.
	\end{proof}
	
	\begin{lemma}\label{Lemma_commute_I0}
		We have the following commutative diagram
		\[\begin{tikzcd}
			{{L^+\calU}\nu({\pi_E}){L^+\calU} \times^{{L^+\calU}}LU.wI/I^0} & {LU.\nu wI/I^0} \\
			{{L^+\calU} \times^{{\calU_{\nu}}}\nu({\pi_E})LU.wI/I^0} \\
			{{\calU_{\nu}}\backslash\nu({\pi_E})LU.wI/I^0} & {\nu wI/I^0}
			\arrow["{m_\nu'}", from=1-1, to=1-2]
			\arrow["\cong", from=1-1, to=2-1]
			\arrow["{p_{\nu w}}"', from=1-2, to=3-2]
			\arrow["{p_U}", from=2-1, to=3-1]
			\arrow["{p_{{\calU_{\nu}}}}", from=3-1, to=3-2]
		\end{tikzcd}\]
	\end{lemma}
	
	\begin{proof}
		Here $p_{\nu w}$ is the base change of 
		$$p^+: \Gr_{\calI^0}^+ = \sqcup_{w \in \tilde{W}}LU.wI/I^0 \to \sqcup_{w \in \tilde{W}} wI/I^0 = \Gr_{\calI^0}^0$$
		along the map
		$$\nu wI/I^0 \to \sqcup_{w \in \tilde{W}} wI/I^0.$$
		
		Note that $\nu(\pi_E)$ normalizes $LU$. Tracking the maps, the lemma follows by noticing that $\nu(\pi_E)LU.wI/I^0=LU.\nu wI/I^0$ and that the map
		$$p_{\nu w}: LU.\nu wI/I^0 \to \nu wI/I^0$$
		factors through $\calU_{\nu} \backslash LU.\nu wI/I^0$.
	\end{proof}
	
	\begin{lemma}\label{Lemma_prounipotent}
		$p_{U, *}p_U^*=\on{id}$.
	\end{lemma}
	
	\begin{proof}
		This holds because $p_U$ is a ${L^+\calU}$-torsor and ${L^+\calU}$ is pro-unipotent, hence has trivial cohomology. See \cite[Proof of Proposition VI.4.1]{fargues2021geometrization} for details.
	\end{proof}

	\subsubsection{Proof of Lemma \ref{Lemma_3.23'} and its corollary}
	
	\begin{proof}[Proof of Lemma \ref{Lemma_3.23'}]
		The proof is similar to \cite[Lemma 3.23]{anschutz2023gaitsgory}. Without loss of generality $\nu$ is dominant. Note that the Wakimoto sheaf $J_{\nu, \chi}$ implicitly depends on the choice of the lift $\nu(\pi_E) \in G(\bE)$ of $\nu$. We will denote $J_{\nu, \chi}$ by $J_{\chi}(\nu(\pi_E))$ in the following. We refer the reader to the preceding discussions for the maps used in the following, notably, Diagram \eqref{Diagram_f} and \eqref{Diagram_phi}.
		\begin{align}
			&\on{CT}_{B^-}^{\chi}(J_{\chi}(\nu({\pi_E})) * \mathcal{F})_{\nu w}\\
			=\;& p_{\nu w, *}q_{\nu w}^!(J_{\chi}(\nu({\pi_E})) * \mathcal{F})\\
			=\;& p_{\nu w, *}q_{\nu w}^!m_*\on{Des}_f(J_{\chi}(\nu({\pi_E})) \tilde{\boxtimes}_{I^0} \mathcal{F}) \\
			=\;& p_{\nu w, *}q_{\nu w}^!m_{\nu, *}\on{Des}_{f_{\nu}}(j_{\nu}^*J_{\chi}(\nu({\pi_E})) \tilde{\boxtimes}_{I^0} \mathcal{F}) \\
			=\;& p_{\nu w, *}q_{\nu w}^!m_{\nu, *}\on{Des}_{\phi_{\nu}}((\underline{\Lambda} \tilde{\boxtimes}_{{L^+\calU}} \mathcal{L}_{\chi}) \tilde{\boxtimes}_{I^0} \mathcal{F})[d_{\nu}] && \text{See Lemma \ref{Lemma_beta_revised}} \\
			=\;& p_{\nu w, *}q_{\nu w}^!m_{\nu, *}(\underline{\Lambda}\tilde{\boxtimes}_{{L^+\calU}}\mathcal{F})[d_{\nu}] && \text{Lemma \ref{Lemma_Descent_revised}}\\
			=\;& p_{\nu w, *}m'_{\nu, *}(q'_{\nu w})^{!}(\underline{\Lambda}\tilde{\boxtimes}_{{L^+\calU}}\mathcal{F})[d_{\nu}] && \text{Lemma \ref{Lemma_Cartesian3}}\\
			=\;& p_{\nu w, *}m'_{\nu, *}(\underline{\Lambda}\tilde{\boxtimes}_{{L^+\calU}}q_{ w}^{!}\mathcal{F})[d_{\nu}] && \text{Lemma \ref{Lemma_Cartesian4}}\\
			=\;& p_{\nu w, *}m'_{\nu, *}(\underline{\Lambda}\tilde{\boxtimes}_{{\calU_{\nu}}}\nu({\pi_E})_*q_{ w}^{!}\mathcal{F})[d_{\nu}] && \text{Lemma \ref{Lemma_isom2}}\\
			=\;& p_{{\calU_{\nu}}, *}p_{U, *}(\underline{\Lambda}\tilde{\boxtimes}_{{\calU_{\nu}}}\nu({\pi_E})_*q_{ w}^{!}\mathcal{F})[d_{\nu}] && \text{Lemma \ref{Lemma_commute_I0}}\\
			=\;&p_{{\calU_{\nu}}, *}p_{U, *}p_U^* \nu({\pi_E})_*q_w^!\mathcal{F}[d_{\nu}] \\
			=\;&p_{{\calU_{\nu}}, *} \nu({\pi_E})_*q_w^!\mathcal{F}[d_{\nu}] && \text{see Lemma \ref{Lemma_prounipotent}}\\
			=\;&p_{\nu w, *} \nu({\pi_E})_*q_w^!\mathcal{F}[d_{\nu}] && \text{$\mathcal{F}$ is left-$I^0$-equivariant}\\
			=\;& \nu({\pi_E})_*p_{w, *}q_w^!\mathcal{F}[d_{\nu}]\\
			=\;& \nu({\pi_E})_*\on{CT}^{\chi}_{B^-}(\mathcal{F})_w[d_{\nu}].
		\end{align}
	\end{proof}
	
	As a corollary of Lemma \ref{Lemma_3.23'}, we deduce the following lemma.
	
	\begin{lemma}[{\cite[Corollary 3.24]{anschutz2023gaitsgory}}]\label{Lemma_3.24'}
		Assume Assumption \ref{Assum_CT}. For any $\mu \in {\bbX_{\bullet}} $ and $\mathcal{F} \in {\on{Wak}}_{\chi}$, we have
		$$\mu({\pi_E})_*\on{Grad}_{\mu}^{\chi}(\mathcal{F})[d(\mu)] \cong \on{CT}_{B^-}^{\chi}(\mathcal{F})_{\mu} \in \scrD_{\on{cons}}((\mathsf{T}, \chi)\backslash\mu({\pi_E}) T/(\mathsf{T}, \chi)),$$
		where $d(\mu)=\langle 2\rho, \mu \rangle+\dim \sfT$.
	\end{lemma}
	
	\begin{proof}
		Since both functors are exact ($\on{CT}_{B^-}^{\chi}$ is both a left and right adjoint by Assumption \ref{Assum_CT}), we are reduced to the case that the length of a Wakimoto filtration is one, i.e., $\mathcal{F}=J_{\nu, \chi}*M$ for some $\nu \in \bbX_{\bullet}$ and $M \in \scrD_{\on{cons}}(*/{\mathsf{T}})$. By Lemma \ref{Lemma_3.23'}, we reduce to the case $\mathcal{F}=J_{0, \chi}*M$ for some $M \in \scrD_{\on{cons}}(*/{\mathsf{T}})$. Then we conclude by a direct computation.
	\end{proof}

	\subsection{Tameness of twisted central sheaves}\label{Subsection_tame}
	In this subsection, we introduce the twisted central sheaves and describe the $\Gal(\bar{\breve{E}}/\breve{E}) \cong I_E$-action on the twisted central sheaves.
	
	\subsubsection{Twisted central sheaves}
	
	Following \cite[Section 4.1]{anschutz2023gaitsgory}, we introduce the spaces that underlie the construction of twisted central sheaves in mixed characteristic.
	
	In Subsection \ref{section-notation}, for $\calH \in \{\calI, \calI^0\}$, we defined the schematic local Hecke stack $$\Hk^{\on{sch}}_{\mathcal{H}}:=\Hk_{\mathcal{H}}=[L^+\mathcal{H}\backslash LG/L^+\calH].$$
	Here $\Hk^{\on{sch}}_{\mathcal{H}}$ is a temporary notation for $\Hk_{\mathcal{H}}$. There is also an analytic local Hecke stack $\Hk_{\mathcal{H}, \Spd \bOE}^{\operatorname{an}}$, which is the $v$-stack $\Hk_{\mathcal{H}, \bOE}$ defined in \cite[Definition 4.1]{anschutz2023gaitsgory}. The space $\Hk_{\mathcal{H}, \Spd \bOE}^{\operatorname{an}}$ is also the space $\Hk^{\loc}_{\calH, \Spd \bOE}$ we introduced in Definition \ref{Def_HklocY}.
	By \cite[Proposition 4.3]{anschutz2023gaitsgory}, there is an equivalence of derived category of sheaves
	\begin{equation}
		c \colon \scrD_{\mathrm{\acute{e}t}}(\Hk_{\calI,k}^{\mathrm{sch}} ) \to \scrD_{\mathrm{\acute{e}t}}(\Hk_{\calI, k}^{\mathrm{an}})
	\end{equation}
	carrying $\scrD_{\mathrm{cons}}(\Hk_{\calI,k}^{\mathrm{sch}} ) $ to $\scrD_{\mathrm{ula}}(\Hk_{\calI, k}^{\mathrm{an}})$. The same is true for $(\sfT \times \sfT, \chi \boxtimes \chi)$-equivariant sheaves on $\Hk_{\calI^0}$. 
	Henceforth we do not distinguish notationally between the  schematic and analytic versions of $\Hk_{\calH}$. Unless stated otherwise, $\Hk_{\calH}$ denotes the analytic version.  Similarly, all the notations $LZ$, $L^+X$, $\Fl_{\calH}$, and $\Gr_{\calH}$ introduced before should be understood as the schematic version and should have been decorated with the superscript $^{\on{sch}}$. There are also analytic version of these spaces, which should be decorated with the superscript $^{\on{an}}$. From now on, unless otherwise stated, we will only consider the analytic version, and we will simply omit the superscript.
	
	We also need the multi-leg analytic local Hecke stack $\Hk_{\mathcal{H}, S}$ for $S$ over a product of $\Spd \bOE$, as defined in \cite[Definition 4.10]{anschutz2023gaitsgory} and Definition \ref{Def_HklocY}.
	
	We choose a complete algebraically closed extension $C$ of $\bE$ with residue field $k=\overline{k}$, and let $\overline{\eta}=\Spd(C), \overline{s}=\Spd(\overline{k}), S_0=\Spd(O_C)$.
	
	To define twisted central sheaves, we consider $${\Hk_{\mathcal{I}^0, S_0 \times \overline{s}}},$$ the two-leg Hecke stack for $\calI^0$ with one leg fixed at $\overline{s}$. We have the following nearby-cycles diagram:
	\begin{equation}\label{Diagram_nearby}
		\begin{tikzcd}
			{\Hk_{\mathcal{I}^0, \overline{s} \times \overline{s}}} & {\Hk_{\mathcal{I}^0, S_0 \times \overline{s}}} & {\Hk_{\mathcal{I}^0, \overline{\eta} \times \overline{s}}} \\
			{\overline{s}} & {S_0} & {\overline{\eta}}
			\arrow["i", from=1-1, to=1-2]
			\arrow[from=1-1, to=2-1]
			\arrow["f", from=1-2, to=2-2]
			\arrow["j"', from=1-3, to=1-2]
			\arrow[from=1-3, to=2-3]
			\arrow[from=2-1, to=2-2]
			\arrow[from=2-3, to=2-2]
		\end{tikzcd}
	\end{equation}
	Note that 
	$$\Hk_{\calI^0, \overline{\eta} \times \overline{s}} \cong \Hk_{G, \overline{\eta}} \times \Hk_{\calI^0, \overline{s}}$$
	because the legs are disjoint.
	Recall that for any $\chi \in \on{CH}(\mathsf{T})$, we have the Kummer sheaf $\calL_{\chi} \in \scrD_{\on{cons}}(\mathsf{T})=\scrD_{\on{cons}}(I/I^0)$. We can view $\calL_{\chi}$ as a sheaf on $LG/I^0$ by pushforward along the closed embedding $I/I^0 \subseteq LG/I^0$.  We can further view $\calL_{\chi}$ as a sheaf on $\Hk_{\calI^0, \overline{s}}$ by descent. Following \cite[Definition 4.4]{anschutz2023gaitsgory} and \cite[Section 4.7.3]{eteve2023monodromiques}, we make the following definition. 
	
	\begin{definition}\label{Def_twisted_central}
		\begin{enumerate}
			\item We define the \textbf{$\chi$-twisted central functor} as the composition
			\begin{equation}
				\scrZ_{\chi}: \Rep(\hat{G}) {\xrightarrow{\Sat}}\scrP_{\mathrm{ula}}(\Hk_{G,C}) {\xrightarrow{\Psi_f(-\boxtimes\calL_{\chi}[\dim \sfT])}}
				\scrD_{\chi, \chi}
			\end{equation}
			\begin{equation*}
				V \mapsto \Sat(V) \mapsto {\Psi_f(\Sat(V) \boxtimes\calL_{\chi}[\dim \sfT])}.
			\end{equation*}
			Here, the first arrow comes from the geometric Satake equivalence of \cite[Chapter VI]{fargues2021geometrization}. The second arrow is nearby cycle for $f$ as in Diagram (\ref{Diagram_nearby}). We have implicitly used the fact that nearby cycles respect the ULA property by \cite[Corollary 6.14]{aglr2022}. 
			
			\item More generally, for every finite set $J$, we can define the \textbf{$\chi$-twisted central functor} as the composition
			$$\scrZ_{\chi}: \Rep(\widehat{G}^J) \xrightarrow{\on{Sat}} \scrP_{\on{ula}}(\Hk_{G, (\Spd C)^J}) \xrightarrow{\Psi(- \boxtimes \calL_{\chi}[\dim \sfT])} \scrD_{\chi, \chi},$$
			where $\Psi$ is the iterated nearby cycle with respect to the map 
			\begin{equation}
				f: \Hk_{\calI^0, S_0^J \times \overline{s}} \to S_0^J.
			\end{equation}

			\item We call $$\scrZ_{\chi}(V)={\Psi_f(\Sat(V) \boxtimes\calL_{\chi}[\dim \sfT])}$$ the \textbf{$\chi$-twisted central sheaf} associated to $V \in \Rep(\hat{G}^J)$.
		\end{enumerate}
	\end{definition}
	
	Note that the maps $i$ and $j$ in Diagram (\ref{Diagram_nearby}) are $(\mathsf{T} \times \mathsf{T})$-equivariant, hence we have $\scrZ_{\chi}(V) \in \scrD_{\chi, \chi}$.

	\subsubsection{Twisted Wakimoto filtration on twisted central sheaves}
	
	In this subsection, we apply Lemma \ref{Lemma_3.24'} to twisted central sheaves.
	
	Recall that $G=\calG_{\bE}$ and $T=\calT_{\bE}$. We denote the pro-$p$ Iwahori of $\calT$ by $\calT_1:=\calI^0_T$. It is obtained from $\calT$ over $\bOE$ by dilating the trivial subgroup on the special fiber. $\calT_1$ is the unique integral model for $T$ whose $\bOE$-points are $$\calT_1(\bOE)=\ker(\calT(\bOE) \to \calT(k)).$$
	Let $f_T$ be the map
	\begin{equation}
		f_T: \Hk_{\calT_1, S_0^J \times \overline{s}} \to S_0^J.
	\end{equation}

	The next theorem says that a twisted central sheaf admits a filtration with explicit graded pieces.
	
	\begin{theorem}[{\cite[Theorem 4.17]{anschutz2023gaitsgory}}]\label{Theorem_central}
		Assume Assumption \ref{Assum_CT}. For every finite set $J$ and $V \in \Rep(\hat{G}^J)$, the complex $$\scrZ_{\chi}(V)=\Psi_f(\Sat(V)\boxtimes \mathcal{L}_{\chi}[\dim T])$$ is $\chi$-Wakimoto filtered, and its associated graded is isomorphic to  $\Psi_{f_T}(\Sat(V|_{\widehat{T}})\boxtimes \mathcal{L}_{\chi}[\dim T])$.
	\end{theorem}
	
	\begin{proof}
		By an argument similar to \cite[Theorem 1 (b)]{gaitsgory2001central}, we have 
		$$\scrZ_{\chi}(V) * J_{\nu, \chi} \cong J_{\nu, \chi} * \scrZ_{\chi}(V)$$ for any $\nu \in \bbX_{\bullet}$. Then by Lemma \ref{Lemma_3.15'}, $\scrZ_{\chi}(V)$ is $\chi$-Wakimoto filtered.
		
		To compute its associated graded, we need to compute $\on{CT}_{B^-}^{\chi}(\scrZ_{\chi}(V))$ by Lemma \ref{Lemma_3.24'}.

		By Assumption \ref{Assum_CT}, the functor $\CT_{B^-}^{\chi}$ commutes with nearby cycles: 
		$$\on{CT}_{B^-}^{\chi}(\Psi_f(\Sat(V)\boxtimes \mathcal{L}_{\chi}[\dim T])) \cong \Psi_{f_T}(\on{CT}_{B^-, \overline{\eta}}^{\chi}(\Sat(V)\boxtimes \mathcal{L}_{\chi}[\dim T])),$$
		where on the right hand side,  $\on{CT}_{B^-, \overline{\eta}}^{\chi}$ is the twisted constant term functor for the geometric generic fiber $\Hk_{\mathcal{I}^0, \overline{\eta}^J \times \overline{s}}$. Applying Geometric Satake on the generic fiber (see Lemma \ref{Lemma_CT_Sat}), we see that 
		$$\Psi_{f_T}(\on{CT}_{B^-, \overline{\eta}}^{\chi}(\Sat(V)\boxtimes \mathcal{L}_{\chi}[\dim T])) \cong \Psi_{f_T}(\Sat(V|_{\widehat{T}})\boxtimes \mathcal{L}_{\chi}[\dim T]).$$
	\end{proof}

	\begin{lemma}\label{Lemma_CT_Sat}
		We have
		$$\on{CT}_{B^-, \overline{\eta}}^{\chi}(\Sat(V)\boxtimes \mathcal{L}_{\chi}) \cong \Sat(V|_{\widehat{T}})\boxtimes \mathcal{L}_{\chi}.$$
	\end{lemma}
	
	\begin{proof}
		The functor $\on{CT}_{B^-, \overline{\eta}}^{\chi}$ in the statement is the constant term functor for the generic fiber $\Hk_{\calI^0, \overline{\eta}^J \times \overline{s}} \cong \Hk_{G, \overline{\eta}^J} \times \Hk_{\calI^0, \overline{s}} \to \overline{\eta}^J$. So we have 
		$$\on{CT}_{B^-, \overline{\eta}}^{\chi}(\Sat(V) \boxtimes \calL_{\chi}) \cong \on{CT}_{B^-, \overline{\eta}}(\Sat V) \boxtimes \on{CT}_{B^-}^{\chi}(\calL_{\chi}).$$
		Note that $\on{CT}_{B^-, \overline{\eta}}(\Sat V)=\Sat(V|_{\widehat{T}})$ by geometric Satake (\cite[Section VI.11]{fargues2021geometrization}). 
		
		The lemma then follows from the fact that
		$$\on{CT}_{B^-}^{\chi}(\calL_{\chi}) \cong \calL_{\chi}.$$
		Indeed, here $\on{CT}_{B^-}^{\chi}$ is the twisted constant term functor (for the opposite Borel) as defined in Subsection \ref{Subsection_Constant_terms}. We have
		$$\on{CT}_{B^-}^{\chi}(\calL_{\chi}) \cong p^+_{B^-, !}q^{+, *}_{B^-}\calL_{\chi} \cong p^-_!\calL_{\chi} = p^-_!\iota_!\calL_{\chi} \cong \calL_{\chi}.$$
		Here for the last two isomorphisms, we recall that $\calL_{\chi}$ is the abbreviation for $\iota_!\calL_{\chi}$, where $\iota: I/I^0 \to LG/I^0$, and the composition $I/I^0 \xrightarrow{\iota} LU^-I/I^0 \xrightarrow{p^-} I/I^0$ is the identity.
	\end{proof}

	\begin{remark}
		As a corollary, we see that assuming Assumption \ref{Assum_CT}, the twisted central sheaf $\scrZ_{\chi}(V)$ is perverse because it admits a filtration whose graded pieces are perverse sheaves. 
	\end{remark}

	\subsubsection{Tameness and monodromy: Statement}
	
	Note that the complex $\Sat(V) \boxtimes\calL_{\chi}[\dim T]$ is defined over $\Hk_{\calI^0, \eta^J \times \overline{s}}$, where $\eta=\Spd(\breve{E})$. So the iterated nearby cycle ${\Psi_f(\Sat(V) \boxtimes\calL_{\chi}[\dim T])}$ carries a $\Gal(\overline{\eta}/\eta)^J$-action.
	
	Recall that we have canonical isomorphisms
	\begin{equation}
		I_E/P_E \cong \widehat{\bbZ}^{(p)}(1) \cong I_{k_E((t))}/P_{k_E((t))}.
	\end{equation}
	And we have a map
	\begin{equation}
		a: [*/I_{k_E((t))}] \cong \on{Spec} \overline{k}_E((t)) \to \bbG_{m, \overline{k}_E}=\on{Spec} \overline{k}_E[u, u^{-1}]
	\end{equation}
	sending $u$ to $t$.
	So for every cocharacter $\nu: \bbG_m \to \sfT$, we can view $\calL_{\chi}$ as a character of $(I_E/P_E)$ via pullback along $\nu$ and $a$.

	\begin{theorem}\label{Thm_tame}
		Assume Assumption \ref{Assum_CT}. For every finite set $J$, $V \in \Rep(\widehat{G}^J)$ and $\chi \in \on{CH}_{\Lambda}(T)$, we have
		\begin{enumerate}
			\item The $\Gal(\overline{\eta}/\eta)^J$-action on $\scrZ_{\chi}(V)$ is tame, i.e., trivial on $P_E^J$. In particular, we have an induced action of $(I_E/P_E)^J$ on $\scrZ_{\chi}(V)$.
			\item The $I_E^J=\Gal(\overline{\eta}/\eta)^J$-action on $\scrZ_{\chi}(V)$ is given by $\chi$. More precisely, for every $\mu \in \bbX_{\bullet}(\sfT)$, the graded piece $\on{gr}_{\mu}^{\chi}\scrZ_{\chi}(V)$ decomposes into a direct sum
			\begin{equation}
				\on{gr}_{\mu}^{\chi}\scrZ_{\chi}(V) \cong \bigoplus_{\nu}A_{\nu}, 
			\end{equation} 
			where $\nu=(\nu_j)_{j \in J}$ runs over elements in $\bbX_{\bullet}(\sfT)^J$ such that $\sum_{j \in J}\nu_j=\mu$. And the $(I_E/P_E)^J$-action on $A_{\nu}$ is given by $\boxtimes_{j \in J}a^*\nu_j^*\calL_{\chi}$. 
		\end{enumerate}
	\end{theorem}

	The key computations needed to prove \ref{Thm_tame} will appear in subsections  \ref{Subsection_compute_equal}, \ref{Subsection_compute_multi}, and \ref{Subsection_compute_mixed}. Let us explain the rest of the proof assuming these computations. 
	
	\begin{proof}
		First notice that (1) follows from (2). Indeed, since there are no nontrivial continuous homomorphisms from the wild inertia (it is pro-$p$) to a unipotent $\overline{\bbF}_{\ell}$-group, it suffices to show that the $\Gal(\overline{\eta}/\eta)^J$-action on the graded pieces are tame. But this follows from (2) because $\calL_{\chi}$ is tame.
		
		Note that 
		$$\on{gr}^{\chi}_{\mu}\scrZ_{\chi}(V) \cong \Psi_{f_T}(\Sat V_{\mu} \boxtimes \calL_{\chi}[\dim T])^{\oplus \dim V[\mu]},$$
		where $V_{\mu}$ is the irreducible representation of $\widehat{T}$ corresponding to $\mu$, and $V[\mu]$ is the $\mu$-weight space of $V$.
		The statement (2) follows from an explicit computation. 
		We refer the reader to Subsection  \ref{Subsection_compute_equal}, \ref{Subsection_compute_multi}, and \ref{Subsection_compute_mixed} for more details on the computation. We construct an explicit section $z_{\nu}$ to the map $f_T$. 
		We then use nearby cycle commutes with proper pushforward to compute the monodromy. 
		This reduces the calculation to nearby cycles for a constant family. 
	\end{proof}

	The above theorem has the following consequences.

	\begin{theorem}[Theorem \ref{Theorem_action}]\label{Thm_act2}
		Assume Assumptions \ref{Assum_sm} and \ref{Assum_CT}.
		\begin{enumerate}
			\item For every finite set $I$, $V \in \Rep(\widehat{G}^I)$, $b \in B(G)$, and a polysimplex $\sigma$ in the Bruhat-Tits building of $G$ over $E$, the $(I_E)^I$-module $(\HV)^{V_{\sigma}(k_E)}$ is tame, i.e., the $(I_E)^I$-action is trivial on the wild inertia $(P_E)^I$.
			\item For every $s \in (\widehat{T}_{M_{\sigma}}\sslash W_{M_{\sigma}})^{\widehat{F}}(\Lambda)$, the $\calO(Z^1(W_E, \widehat{G})/\widehat{G})$-module 
			$$e_s(\HV)^{V_{\sigma}(k_E)}$$ 
			is supported on $\on{ev}_{\tau_E}^{-1}(s)$. 
		\end{enumerate}
	\end{theorem}
	
	\begin{proof}
		The same proof as \cite[Section 4.8]{eteve2023monodromiques} applies. Indeed, (1) follows from Assumption \ref{Assum_sm}, Theorem \ref{Thm_pPsi}, and Theorem \ref{Thm_tame}. For (2), the same computation as \cite[Section 4.8.3]{eteve2023monodromiques} applies.
	\end{proof}

	\begin{theorem}[{Theorem \ref{Thm_LLC}}]
		Assume Assumptions \ref{Assum_sm} for $b=1$ and Assumption \ref{Assum_CT}. Let $\pi \in \Irr_{\Lambda}^0G(E)$ be an irreducible depth-zero representation.
		\begin{enumerate}
			\item The corresponding $L$-parameter $\varphi_{\pi}^{\on{FS}}$ is tame. In particular, we can view $\varphi_{\pi}^{\on{FS}}$ as a map from $W_E/P_E$.
			\item The restriction of $\varphi_{\pi}^{\on{FS}}$ to $I_E/P_E$ is determined by its value on the topological generator $\tau_E \in I_E/P_E$, which is computed by the following commutative diagram
			\[\begin{tikzcd}
				{\Irr_{\Lambda}^0G(E)} & {(Z^1(W_E/P_E, \widehat{G})\sslash \widehat{G})(\Lambda)} \\
				{(\widehat{T}\sslash W)^{\widehat{F}}(\Lambda)} & {(\widehat{T}\sslash W)^{\widehat{F}}(\Lambda)}
				\arrow["{\on{LLC}^{\on{FS}}}", from=1-1, to=1-2]
				\arrow["{\on{LS}}"', from=1-1, to=2-1]
				\arrow["{\on{ev}_{\tau_E}}", from=1-2, to=2-2]
				\arrow["{=}", from=2-1, to=2-2].
			\end{tikzcd}\]
		\end{enumerate}
	\end{theorem}
	
	\begin{proof}
		This follows from the previous Theorem \ref{Thm_act2}.
	\end{proof}

	\subsection{Computation in equal characteristic: one-leg version}\label{Subsection_compute_equal}

	Let $k=\overline{\bbF}_p$. We first assume that $\on{char}(E)=p$ and the finite set $J=\{1\}$. Using the chosen uniformizer $\pi_E=t$ of $E$, we identify
	$$\bOE=k[[t]], \quad \bE=k((t)).$$ 
	Put
	$$X=\Spd \bOE = \Spd k[[t]], \;\; \eta= \Spd \bE=\Spd k((t)), \;\; x=\Spd k.$$
	Let $\chi \in \on{CH}_{\Lambda}(\sfT)$. We are going to compute the $\Gal(\overline{\eta}/\eta)^J=\Gal(\overline{\eta}/\eta)$-action on $\scrZ_{\chi}(V).$ 
	
	For a $k$-torus $\sfT$, let $\sfT^{\Diamond}$ denote the
	$v$-sheaf defined on affinoid perfectoid spaces over $k$ by
	\[
	\sfT^{\Diamond}(\Spa(R,R^+))=\sfT(R).
	\]
	Thus $\Diamond$ denotes the big diamond construction.
	In particular,
	\[
	\bbG_m^{\Diamond}=\Spd(k[u,u^{-1}],k),
	\]
	where $k[u,u^{-1}]$ is equipped with the discrete topology.
	We use the same notation $\calL_{\chi}$ for the pullback
	of the Kummer sheaf to $\sfT^{\Diamond}$.
	
	Define
	\[
	a:\eta\longrightarrow\bbG_m^{\Diamond}
	\]
	by the homomorphism
	\[
	k[u,u^{-1}]\longrightarrow k((t)),\qquad u\longmapsto t.
	\]

	\begin{theorem}
		For every cocharacter $(\nu: \bbG_m \to \sfT) \in \bbX_{\bullet}(\sfT)$, the $\Gal(\overline{\eta}/\eta)$-action on 
		$$\Psi_{f_T}(\Sat V_{\nu} \boxtimes \calL_{\chi}[\dim \sfT])$$ 
		is given by $a^*\nu^*\calL_{\chi}$.
	\end{theorem}
	
	The proof of this theorem will occupy the rest of Subsection \ref{Subsection_compute_equal}. From now on, we suppress the cohomological shift $[\dim \sfT]$ that does not affect the monodromy.

	\subsubsection{The case $\sfT=\bbG_m$ and $\nu=1$}
	
	We begin with the case $\sfT=\bbG_m=\on{Spec} k[u, u^{-1}], \nu=1$, i.e., $\nu$ is the map $\bbG_m \to \bbG_m, u \mapsfrom u$. Note that we have $\bbX_{\bullet}(\bbG_m) \cong \bbZ$. For $n \in \bbZ$, we denote 
	\begin{equation}
		\nu_n: \bbG_m \to \bbG_m, \quad u^n \mapsfrom u.
	\end{equation}

	Recall that $\calT_1$ over $\on{Spec} k[[t]]$ is the integral model of $\bbG_{m, k((t))}$ such that
	$$\calT_1(k[[t]])=1+tk[[t]].$$
	
	Recall that we have the map $f_T: \Hk_{\calT_1, X \times x} \to X$ remembering the first leg. 
	We can view it as a map $\Gr_{\calT_1, X \times x} \to X$, which we still denote by $f_T$.  
	Note that the map
	$$\pi: \Gr_{\calT_1, X \times x} \to \Hk_{\calT_1, X \times x}$$
	can be factored as the composition of a smooth map and a pro-unipotent map, so $\pi^*$ commutes with nearby cycles. Therefore, it suffices to compute the nearby cycle with respect to the diagram
	\[\begin{tikzcd}
		{\Gr_{\calT_1, x}} & {\Gr_{\calT_1, X \times x}} & {\Gr_{\calT_1, \eta \times x}} \\
		x & X & \eta
		\arrow[from=1-1, to=1-2]
		\arrow[from=1-1, to=2-1]
		\arrow["{f_T}", from=1-2, to=2-2]
		\arrow[from=1-3, to=1-2]
		\arrow[from=1-3, to=2-3]
		\arrow[from=2-1, to=2-2]
		\arrow[from=2-3, to=2-2].
	\end{tikzcd}\]

	Recall that when $\on{char}{E}=p$, we have for every perfectoid space $S=\Spa(R, R^+)$ over $k$, the relative curve 
	\begin{equation}
		\calY_{S}:=\calY_{S, E} \subseteq \Spd R^+[[t]].
	\end{equation}
	The space $X=\Spd \bOE$ classifies degree one Cartier divisors on $\calY_S$ in the following sense. A map 
	\begin{equation}
		\varphi: S=\Spa (R, R^+) \to \Spd k[[t]]=X
	\end{equation}
	is determined by the element
	\begin{equation}
		a_{\varphi}:=\varphi^{\#}(t) \in R^+.
	\end{equation}
	This gives a degree one Cartier divisor on $\calY_S$ via pullback along
	$$(a_{\varphi}-t)R^+[[t]] \subseteq R^+[[t]].$$
	In particular, the constant map
	\begin{equation}
		\underline{x}: S \to X
	\end{equation}
	gives a degree one Cartier divisor $\Gamma_{x} \subseteq \calY_S$. The completion of $\calY_S$ along $\Gamma_{x}$ is 
	\begin{equation}
		\hat{\Gamma}_{x} \cong \Spa R[[t]].
	\end{equation}

	Recall that for every perfectoid space $S=\Spa(R, R^+)$ over $k$, the groupoid $\Gr_{\calT_1, x}(S)$ classifies tuples $\{(\calE, \beta, \psi_x)\}$, where
	\begin{enumerate}
		\item $\calE$ is a $GL_1$-torsor over $\calY_S$,
		\item $\beta: \calE|_{\calY_S-\Gamma_x} \to \calE^0|_{\calY_S-\Gamma_x}$ is a trivialization of $\calE$ over ${\calY_S-\Gamma_x}$,
		\item $\psi_x: \calE|_{\Gamma_x} \to \calE^0|_{\Gamma_x}$ is a trivialization of $\calE$ over $\Gamma_x \cong x \times S \cong S$.
	\end{enumerate}
	
	Similarly, the groupoid $\Gr_{\calT_1, X \times x}(S)$ classifies tuples $\{(y, \calE, \beta, \psi_x)\}$, where
	\begin{enumerate}
		\item $y=(y_1, y_2) \in X(S)^2$ such that $y_2=\underline{x}$ is the constant map,
		\item $\calE$ is a $GL_1$-torsor over $\calY_S$,
		\item $\beta: \calE|_{\calY_S-\Gamma_y} \to \calE^0|_{\calY_S-\Gamma_y}$ is a trivialization of $\calE$ over $\calY_S-\Gamma_y$, where $\Gamma_y=\Gamma_{y_1} \cup \Gamma_{y_2}$,
		\item $\psi_x: \calE|_{\Gamma_x} \to \calE^0|_{\Gamma_x}$ is a trivialization of $\calE$ over $\Gamma_x \cong x \times S \cong S$.
	\end{enumerate}
	
	\begin{lemma}\label{Lemma_GrT}
		We have an isomorphism (explicitly defined in the proof)
		\begin{equation}
			\theta: \Gr_{\calT_1, x} \xrightarrow{\cong} \underline{\bbZ}_x \times \sfT^{\Diamond}.
		\end{equation}
	\end{lemma}
	
	\begin{proof}
		We are going to define the map v-locally in $S$. Let $S$ be a perfectoid space over $k$ and $(\calE, \beta, \psi_x) \in \Gr_{\calT_1, x}(S)$. After localizing, we may assume that $S=\Spa(R, R^+)$ is affinoid and connected. Note that v-locally in $S$, the $GL_1$-torsor $\calE$ can be trivialized. So we may assume that $(\calE, \beta, \psi_x)$ is isomorphic to $(\calE^0, \beta^0, \psi_x^0)$, where 
		\begin{enumerate}
			\item $\calE^0$ is the trivial $GL_1$-torsor over $\calY_S$,
			\item  $\beta^0 \in \on{Aut}(\calE^0|_{\calY_S-\Gamma_x})$,
			\item $\psi_x^0 \in \on{Aut}(\calE^0|_S) \cong R^*$.
		\end{enumerate}
		Restricting to 
		$$\hat{\Gamma}_{x}^*:=\hat{\Gamma}_{x}-\Gamma_x \cong \Spa R((t)),$$
		we have
		\begin{equation}
			\beta^0|_{\hat{\Gamma}_{x}^*} \in \on{Aut}(\calE^0|_{\hat{\Gamma}_{x}^*}) \cong R((t))^*.
		\end{equation}
		So we can write 
		\begin{equation}
			\beta^0|_{\hat{\Gamma}_{x}^*}=t^n. u^0(t),
		\end{equation}
		for some $n \in \bbZ$ and $u^0(t) \in R[[t]]^*$. We call such $n \in \bbZ$ the \textit{type of modification of $\calE$ near $x$}. Evaluating at $0$ gives an element
		$u^0(0) \in R^*$.
		We define 
		\begin{equation}
			\theta((\calE, \beta, \psi_x))=(n, u^0(0)(\psi_x^0)^{-1}) \in \bbZ \times R^* =(\bbZ_x \times \sfT^{\Diamond})(S).
		\end{equation}
		
		The inverse map is given by 
		\begin{equation}
			(\calE, \beta, \psi_x) \mapsfrom (n, \alpha)
		\end{equation}
		where $n \in \bbZ, \alpha \in R^*$ and $\calE=\calE^0, \beta=t^n \in R((t))^*, \psi_x=\alpha^{-1} \in R^*$.
	\end{proof}

	\begin{lemma}
		We have isomorphisms 
		\begin{equation}
			\Gr_{\calT_1, \eta \times x} \xrightarrow{\cong} \Gr_{\calT_1, \eta} \times \Gr_{\calT_1, x} \xrightarrow{\cong} \underline{\bbZ}_{\eta} \times (\underline{\bbZ}_x \times \sfT^{\Diamond}). 
		\end{equation}
	\end{lemma}
	
	\begin{proof}
		Over $\eta \times x$, the two legs $y_1, y_2$ are disjoint, so we have an isomorphism
		\begin{equation}
			\Gr_{\calT_1, \eta \times x} \xrightarrow{\cong} \Gr_{\calT_1, \eta} \times \Gr_{\calT_1, x}
		\end{equation}
		by restricting to $\hat{\Gamma}_{y_1}$, $\hat{\Gamma}_{y_2}$, respectively.
		
		There is an isomorphism 
		\begin{equation}
			\Gr_{\calT_1, \eta} \xrightarrow{} \underline{\bbZ}_{\eta}
		\end{equation}
		recording the type of modification near $y_1$.
		
		Finally, we have an isomorphism
		$$\Gr_{\calT_1, x} \cong (\underline{\bbZ}_x \times \sfT^{\Diamond})$$
		by Lemma \ref{Lemma_GrT}.
	\end{proof}

	\begin{lemma}\label{Lemma_section}
		We have a section (explicitly defined in the proof)
		\begin{equation}
			z:=z_1: X \to \Gr_{\calT_1, X \times x}
		\end{equation}
		such that
		\begin{enumerate}
			\item Over $x \in X$, the map $z_x: x \to \Gr_{\calT_1, x} \cong \underline{\bbZ}_x \times \sfT^{\Diamond} = \sqcup_{n \in \bbZ}\nu_n\sfT^{\Diamond}$ factors through $\nu_1\sfT^{\Diamond}$ and maps $x$ to $\nu_1.(-1) \in \nu_1\sfT^{\Diamond}$.
			\item Over $\eta \in X$, the map $z_{\eta}: \eta \to \Gr_{\calT_1, \eta \times x} \cong \underline{\bbZ}_{\eta} \times \sqcup_{n \in \bbZ}\nu_n\sfT^{\Diamond}$ factors through $(1, \nu_0 \sfT^{\Diamond}) \cong \eta \times \sfT^{\Diamond}$, and the map $\Spd k((t)) = \eta \to \sfT^{\Diamond}=\Spa (k[u, u^{-1}], k)$ is given by $a: t \mapsfrom u$.
		\end{enumerate}
	\end{lemma}
	
	\begin{proof}
		Let $\Perf_{/ X}$ be the category of perfectoid spaces over $X$. We need to define functorially in $(\varphi: S \to X) \in \Perf_{/ X}$, sections 
		$$z({\varphi}) \in  \Gr_{\calT_1, X \times x}(S).$$
		
		Given a map
		\begin{equation}
			\varphi: S=\Spa(R, R^+) \to X= \Spd k[[t]]
		\end{equation}
		corresponding to the topologically nilpotent element $a_{\varphi}=\varphi^{\#}(t) \in R$. Define the element $z({\varphi})=(y_1, y_2, \calE, \beta, \psi_x) \in \Gr_{\calT_1, X \times x}(S)$ as follows:
		\begin{enumerate}
			\item $y_1=\varphi$, $y_2=\underline{x}$,
			\item $\calE:=\calO(-\Gamma_{y_1})$ the ideal sheaf corresponding to the graph $\Gamma_{y_1}$,
			\item $\beta: \calO(-\Gamma_{y_1}) \subseteq \calO_{\calY_S}$ which is the base change of $(a_{\varphi}-t)R^+[[t]] \subseteq R^+[[t]]$,
			\item let $\psi_x$ be the base change of the multiplication map by $(a_{\varphi}-t)^{-1}$: 
			\begin{equation}\label{Eq_psi}
				\psi_x: ((a_{\varphi}-t)R^+[[t]] \xrightarrow{(a_{\varphi}-t)^{-1}} R^+[[t]]) \otimes_{R^+[[t]]} R.
			\end{equation}
		\end{enumerate}

		Before continuing to prove (1) and (2), we note that there is a global trivialization 
		\begin{equation}
			\alpha: \calE \to \calE^0
		\end{equation}
		over $\calY_S$ given by the base change of the multiplication by $(a_{\varphi}-t)^{-1}$:
		\begin{equation}
			(a_{\varphi}-t)R^+[[t]] \to R^+[[t]].
		\end{equation}
		So the tuple $(\calE, \beta, \psi_x)$ is isomorphic to $(\calE^0, \beta^0, \psi_x^0)$ via $\alpha$, where
		\begin{enumerate}
			\item $\beta^0$ is the base change of the composite
			\begin{equation}
				R^+[[t]] \xrightarrow{\alpha^{-1}=(a_{\varphi}-t)} (a_{\varphi}-t)R^+[[t]] \xrightarrow{\beta=\text{inclusion}} R^+[[t]]
			\end{equation}
			along $\times_{\Spa R^+[[t]]}(\calY_S-\Gamma_y)$,
			\item $\psi_x^0$ is the base change of the composite
			\begin{equation}
				R^+[[t]] \xrightarrow{\alpha^{-1}=(a_{\varphi}-t)} (a_{\varphi}-t)R^+[[t]] \xrightarrow{\psi_x=(a_{\varphi}-t)^{-1}} R^+[[t]]
			\end{equation}
			along $\times_{\Spa R^+[[t]]}S$.
		\end{enumerate}
		In other words, $\beta^0=(a_{\varphi}-t) \in R((t))$ and $\psi_x^0=1 \in R^*$.
		
		Next, we prove (1) about $z_x$. Over $x \in X$, we have $y_1 = \underline{x}$, so $a_{\varphi}=\varphi^{\#}(t)=0$. So $\beta^0=t.u^0(t)$, where $u^0(t)=-1 \in R((t))^*$. Therefore, the tuple $(\calE, \beta, \psi_x)$ corresponds to 
		\begin{equation}
			(n, u^0(0)(\psi_x^0)^{-1}) = (1, -1)=\nu_1.(-1) \in \nu_1R^*=\nu_1\sfT^{\Diamond}(S).
		\end{equation}
		
		Finally, we prove (2) about $z_{\eta}$. Over $\eta \in X$, the two legs $y_1=\varphi$ and $y_2=\underline{x}$ are disjoint. We have $a_{\varphi}=\varphi^{\#}(t) \in R^*$. We are going to pin down the map
		$$z_{\eta}: \eta \to \Gr_{\calT_1, \eta \times x} \cong \Gr_{\calT_1, \eta} \times \Gr_{\calT_1, x} \cong \underline{\bbZ}_{\eta} \times (\underline{\bbZ}_x \times \sfT^{\Diamond}).$$
		
		The map $\eta \to \Gr_{\calT_1, \eta} \cong \underline{\bbZ}_{\eta}$ records the type of the modification $\beta$ near $y_1$. Since $\calO(-\Gamma_{y_1}) \subseteq \calO_{\calY_S}$ is a modification of type $1$, the map $\eta \to \Gr_{\calT_1, \eta} \cong \underline{\bbZ}_{\eta}$ is the isomorphism onto the 1-th copy $\eta \subseteq \underline{\bbZ}_{\eta}$.
		
		It remains to understand the map $\eta \to \Gr_{\calT_1, x} \cong \underline{\bbZ}_x \times \sfT^{\Diamond}$. Since $a_{\varphi} \in R^*$, we have $\beta^0=(a_{\varphi}-t) \in R[[t]]^*$. So $u^0(t)=\beta=(a_{\varphi}-t)$. Therefore, the tuple $(\calE, \beta, \psi_x)$ corresponds to 
		\begin{equation}
			(n, u^0(0)(\psi_x^0)^{-1})=(0, a_{\varphi}) = \nu_0a_{\varphi} \in \nu_0R^*=\nu_0\sfT^{\Diamond}(S).
		\end{equation} 
		From this, we see that the map $\Spd k((t)) = \eta \to \nu_0\sfT^{\Diamond} \cong \sfT^{\Diamond}=\Spa (k[u, u^{-1}], k)$ is given by $a: t \mapsfrom u$.
	\end{proof}
	
	Note that we have a $\sfT^{\Diamond}$-action on $\Gr_{\calT_1, X \times x}$ by altering the trivialization $\psi_x$:
	\begin{equation}
		\sfT^{\Diamond} \times \Gr_{\calT_1, X \times x} \to \Gr_{\calT_1, X \times x} \quad (t', (y, \calE, \beta, \psi_x)) \mapsto (y, \calE, \beta, (t')^{-1} \circ \psi_x).
	\end{equation}
	Applying this action to $z$ yields a map $\tilde{z}: X  \times \sfT^{\Diamond} \to \Gr_{\calT_1, X \times x}$ satisfying the following lemma.
	
	\begin{lemma}\label{Lemma_section2}
		We have a map (explicitly defined in the proof)
		\begin{equation}
			\tilde{z}: X \times \sfT^{\Diamond} \to \Gr_{\calT_1, X \times x} 
		\end{equation}
		such that
		\begin{enumerate}
			\item Over $x \in X$, the map $\tilde{z}|_{x \times \sfT^{\Diamond}}: x \times \sfT^{\Diamond} \to \Gr_{\calT_1, x} \cong \sqcup_{n \in \bbZ}\nu_n \sfT^{\Diamond}$ maps $x \times \sfT^{\Diamond} \cong \sfT^{\Diamond}$ to $\nu_1 \sfT^{\Diamond} \cong \sfT^{\Diamond}$, and the map $\sfT^{\Diamond} \to \sfT^{\Diamond}$ is given by $-\on{id}$.
			\item Over $\eta \in X$, the map $\tilde{z}_{\eta}: \eta \times \sfT^{\Diamond} \to \Gr_{\calT_1, \eta \times x} \cong \Gr_{\calT_1, \eta} \times \Gr_{\calT_1, x} \cong \underline{\bbZ}_{\eta} \times \sqcup_{n \in \bbZ}\nu_n \sfT^{\Diamond}$ maps $\eta \times \sfT^{\Diamond}$ to $(1, \nu_0\sfT^{\Diamond})$, and the map $\eta \times \sfT^{\Diamond} \to \eta \times \sfT^{\Diamond}$ is $(p_1, m)$, where $p_1: \eta \times \sfT^{\Diamond} \to \eta$ and $m$ is the composite
			\begin{equation}
				 \eta \times \sfT^{\Diamond} \xrightarrow{(a, \on{id}_{\sfT^{\Diamond}})} \sfT^{\Diamond} \times \sfT^{\Diamond} \xrightarrow{\mu} \sfT^{\Diamond},
			\end{equation}
			where $\mu$ is the multiplication map.
		\end{enumerate}
	\end{lemma}
	
	\begin{proof}
		We are going to define functorially in $(\varphi: S \to X) \in \Perf_{/X}$ elements
		\begin{equation}
			\tilde{z}({\varphi}) \in  \Gr_{\calT_1, X \times x}(S).
		\end{equation}
		Let $(\varphi, t') \in X(S) \times \sfT^{\Diamond}(S)$. Define $\tilde{z}({\varphi})=(y_1, y_2, \calE, \beta, \psi_x) \in \Gr_{\calT_1, X \times x}(S)$, where
		\begin{enumerate}
			\item $y_1=\varphi, y_2=\underline{x}$,
			\item $\calE=\calO(-\Gamma_{y_1})$,
			\item $\beta: \calO(-\Gamma_{y_1}) \subseteq \calO_{\calY_S}$ as before,
			\item $\psi_x: (a_{\varphi}-t)R^+[[t]] \otimes_{R^+[[t]]}R \xrightarrow{(t'(a_{\varphi}-t))^{-1}} R^+[[t]] \otimes_{R^+[[t]]} R$.
		\end{enumerate}
		It is straightforward to check that the map $\tilde{z}$ satisfies the properties (1) and (2).
	\end{proof}
	
	We have the following diagram
	\[\begin{tikzcd}
		{x \times \sfT^{\Diamond}} & {X \times \sfT^{\Diamond}} & {\eta \times \sfT^{\Diamond}} \\
		{\Gr_{\calT_1, x}} & {\Gr_{\calT_1, X \times x}} & {\Gr_{\calT_1, \eta \times x}} \\
		x & X & \eta
		\arrow[from=1-1, to=1-2]
		\arrow[from=1-1, to=2-1]
		\arrow["{\tilde{z}}", from=1-2, to=2-2]
		\arrow[from=1-3, to=1-2]
		\arrow[from=1-3, to=2-3]
		\arrow[from=2-1, to=2-2]
		\arrow[from=2-1, to=3-1]
		\arrow["{f_T}", from=2-2, to=3-2]
		\arrow[from=2-3, to=2-2]
		\arrow[from=2-3, to=3-3]
		\arrow[from=3-1, to=3-2]
		\arrow[from=3-3, to=3-2].
	\end{tikzcd}\]
	Note that the composition $f_T \circ \tilde{z}$ is the first projection map $p_1: X \times \sfT^{\Diamond} \to X$. Let $\Sat V_1 \in \scrD_{\et}(\Gr_{T, \eta})$ be the Satake sheaf corresponding to $\nu_1$. Explicitly, $\Sat V_1$ is supported on the $1$-th component $\nu_1\eta$ of $\Gr_{T, \eta}=\underline{\bbZ}_{\eta}$, and it is the constant sheaf on $\nu_1\eta$.  
	
	Note that the map $\tilde{z}$ is a closed immersion, hence proper. On the special fiber, the map $\tilde{z}$ is an isomorphism onto a connected component. So we can compute $\Psi_{f_T}(\Sat V_{1} \boxtimes \calL_{\chi})$ using the fact that $\tilde{z}_*$ commutes with nearby cycles. 
	
	\begin{lemma}\label{Lemma_zLchi}
		On the generic fiber, we have
		$$\tilde{z}_*(a^*\calL_{\chi} \boxtimes \calL_{\chi})=\Sat V_1 \boxtimes \calL_{\chi}.$$ 
	\end{lemma}
	
	\begin{proof}
		Indeed, $\Sat V_1 \boxtimes \calL_{\chi}$ is supported on an irreducible component $Z$ of $\Gr_{\calT_1, \eta \times x}$. The irreducible component $Z$ is isomorphic to $\eta \times \sfT$, with $\Sat V_1 \boxtimes \calL_{\chi}$ corresponding to $\underline{\Lambda} \boxtimes \calL_{\chi}$. The map $\tilde{z}: \eta \times \sfT \to Z$ is identified with the map $(p_1, m): \eta \times \sfT \to  \eta \times \sfT$ by Lemma \ref{Lemma_section2}. So it suffices to show that $(p_1, m)_*(a^*\calL_{\chi} \boxtimes \calL_{\chi})=\underline{\Lambda} \boxtimes \calL_{\chi}$. Note that the map $(p_1, m): \eta \times \sfT \to  \eta \times \sfT$ is an isomorphism. Hence it suffices to show that
		$$a^*\calL_{\chi} \boxtimes \calL_{\chi}=(p_1, m)^*(\underline{\Lambda} \boxtimes \calL_{\chi}).$$
		This follows from $\mu^*\calL_{\chi} \cong \calL_{\chi} \boxtimes \calL_{\chi}$, where $\mu: \sfT \times \sfT \to \sfT$ is the multiplication map.
	\end{proof}
	
		We can now compute that
	\begin{align}
		\;& \Psi_{f_T}(\Sat V_{1} \boxtimes \calL_{\chi})\\
		\cong \;& \Psi_{f_T}(\tilde{z}_*(a^*\calL_{\chi} \boxtimes \calL_{\chi}))\\
		\cong \;& \tilde{z}_*\Psi_{p_1}(a^*\calL_{\chi} \boxtimes \calL_{\chi}) \quad{\text{$\tilde{z}$ is proper}}\\
		\cong \;& (-\on{id})_*\Psi_{p_1}(a^*\calL_{\chi} \boxtimes \calL_{\chi}) \quad{\text{Lemma \ref{Lemma_section2}}}\\
		\cong \;& \Psi_{\on{id}_X}(a^*\calL_{\chi}) \boxtimes (-\on{id})_*\calL_{\chi} \quad \text{Nearby cycles for the constant family $p_1: X \times \sfT \to X$}
	\end{align}
	
	Note that the constant factor $(-\on{id})_*\calL_{\chi}$ does not affect the monodromy. Therefore, we get that the $\Gal(\overline{\eta}/\eta)=\pi_1(\eta)$-action on $\Psi_{f_T}(\Sat V_{1} \boxtimes \calL_{\chi})$ is given by $a^*\calL_{\chi}$.
	
	\subsubsection{The case for general torus}
	Let $\sfT$ be a torus over $k$ ($\sfT$ is split) and $\nu \in \bbX_{\bullet}(\sfT)$.
	
	\begin{lemma}\label{Lemma_section_T}
		For every $\nu \in \bbX_{\bullet}(\sfT)$, there is a section (explicitly defined in the proof)
		\begin{equation}
			z_{\nu}: X \to \Gr_{\calT_1, X \times x}
		\end{equation}
		such that
		\begin{enumerate}
			\item Over $x$, the map $z_{\nu}|_x: x \to \Gr_{\calT_1, x \times x} \cong \sqcup_{\nu' \in \bbX_{\bullet}(\sfT)}\nu'\sfT^{\Diamond}$ maps to the $\nu$-th component.
			\item Over $\eta$, the map $z_{\nu}|_{\eta}: \eta \xrightarrow{(f_1, f_2)} \Gr_{\calT_1, \eta} \times \Gr_{\calT_1, x} \cong \underline{\bbX_{\bullet}(\sfT)}_{\eta} \times \sqcup_{\nu' \in \bbX_{\bullet}(\sfT)}\nu'\sfT^{\Diamond}$ satisfies the following properties: the map $f_1$ maps $\eta$ to the $\nu$-th component of $\underline{\bbX_{\bullet}(\sfT)}_{\eta}$; the map $f_2$ maps $\eta$ to the $0$-th component of
			$\sqcup_{\nu' \in \bbX_{\bullet}(\sfT)}\nu'\sfT^{\Diamond}$, and the map $\eta \to \sfT^{\Diamond}$ is given by $\nu \circ a$.
		\end{enumerate}
	\end{lemma}
	
	\begin{proof}
		In fact, given $\nu: \bbG_m \to \sfT$, we have a map $\nu_*: \Gr_{\bbG_{m, 1}, X \times x} \to \Gr_{\calT_1, X \times x}$ by induction of torsors along $\nu$, where $\bbG_{m, 1}$ is the dilation of $\bbG_m$ at the trivial subgroup over the special fiber. In our case, we can define
		\begin{equation}
			z_{\nu}:=\nu_*z \in \Gr_{\calT_1, X \times x}(X),
		\end{equation}
		where $z: X \to \Gr_{\bbG_{m, 1}, X \times x}$ is the section we constructed in  Lemma \ref{Lemma_section}.
		
		Explicitly, for every $\varphi: S=\Spa(R, R^+) \to X$, the section $z_{\nu}(\varphi)$ is given by $(y_1, y_2, \calE, \beta, \psi_x) \in \Gr_{\calT_1, X \times x}(S)$, where
		\begin{enumerate}
			\item $y_1=\varphi$, $y_2=\underline{x}$.
			\item $\calE:=\calO(-\Gamma_{y_1}) \times_{\bbG_m, \nu} \sfT$, the induction of the $\bbG_m$-torsor $\calO(-\Gamma_{y_1})$ along $\nu: \bbG_m \to \sfT$.
			\item $\beta$ is the induction of the inclusion $\calO(-\Gamma_{y_1}) \subseteq \calO$ along $\nu$.
			\item $\psi_x$ is the induction of \eqref{Eq_psi} along $\nu$.
		\end{enumerate}
		
		It remains to verify the assertions over $x$ and $\eta$.
		\begin{enumerate}
			\item Over $x$: Note that the map $\nu_*: \Gr_{\bbG_m, x} \cong \underline{\bbZ}_x \to \Gr_{\sfT, x} \cong \underline{\bbX_{\bullet}(\sfT)}_x$ sends $1=\nu_1$ to $\nu$. Thus $z_{\nu}(x)$ lies in the $\nu$-th component. Moreover, since $z=z_{1}$ sends $x$ to $-1 \in \bbG_m^{\Diamond} \cong \nu_1 \bbG_m^{\Diamond}$, we have $z_{\nu}(x)=\nu.\nu(-1) \in \nu \sfT^{\Diamond}$.

			\item Over $\eta$: We look at the map 
			$$z_{\nu}|_{\eta}: \eta \xrightarrow{(f_1, f_2)} \Gr_{\calT_1, \eta} \times \Gr_{\calT_1, x} \cong \underline{\bbX_{\bullet}(\sfT)}_{\eta} \times \sqcup_{\nu' \in \bbX_{\bullet}(\sfT)}\nu'\sfT^{\Diamond}.$$
			As in (1), the map $f_1$ sends $\eta$ to the $\nu$-th component, and the map $f_2$ sends $\eta$ to the $0$-th component. Recall that for $(\sfT, \nu)=(\bbG_m, 1)$, the map $z=z_{1}$ induces $f_{2, 1}=a: \eta \to \bbG_m^{\Diamond}$. Therefore, we have $f_2=f_{2, \nu}=\nu \circ a$.
		\end{enumerate}
		
	\end{proof}
	
	Similarly, we have the analogue of Lemma \ref{Lemma_section2}.
	
	\begin{lemma}
		For every $\nu \in \bbX_{\bullet}(\sfT)$, there is a map 
		\begin{equation}
			\tilde{z}_{\nu}: X \times \sfT^{\Diamond} \to \Gr_{\calT_1, X \times x}
		\end{equation}
		such that 
		\begin{enumerate}
			\item The map $\tilde{z}_{\nu}|_{x \times \sfT^{\Diamond}}: x \times \sfT^{\Diamond} \to \Gr_{\calT_1, x \times x} \cong \sqcup_{\nu' \in \bbX_{\bullet}(\sfT)}\nu' \sfT^{\Diamond}$
			factors through $\nu \sfT^{\Diamond}$, and the map
			$x \times \sfT^{\Diamond} \cong \sfT^{\Diamond} \to \nu \sfT^{\Diamond} \cong \sfT^{\Diamond}$ is given by $\nu(-1)$, multiplication by $\nu(-1)$.
			\item The map $\tilde{z}_{\nu} |_{\eta \times \sfT^{\Diamond}}: \eta \times \sfT^{\Diamond} \xrightarrow{(f_1, f_2)} \Gr_{\calT_1, \eta \times x} \cong \Gr_{\calT_1, \eta} \times \Gr_{\calT_1, x}$ satisfies the following: the map $f_1$ factors through the $\nu$-th component of $\Gr_{\calT_1, \eta} \cong \underline{\bbX_{\bullet}(\sfT)}_{\eta}$, and the map $\eta \times \sfT^{\Diamond} \to \nu \eta \cong \eta$ is given by $p_1$, the projection to the first factor; the map $f_2$ factors through the $0$-th component of $\Gr_{\calT_1, x}$, and the map $\eta \times \sfT^{\Diamond} \to \sfT^{\Diamond}$ is given by
			\begin{equation}
				m_{\nu}:=\mu \circ (\nu \circ a, \id_{\sfT^{\Diamond}}): \quad  (t, u) \mapsto \nu(a(t)).u.
			\end{equation} 
		\end{enumerate}
	\end{lemma}
	
	\begin{proof}
		A similar proof as Lemma \ref{Lemma_section2} works.
	\end{proof}

	Over $\eta \times \sfT^{\Diamond}$, we have
	\begin{equation}
		(p_1, m_{\nu})^*(\underline{\Lambda} \boxtimes \calL_{\chi}) \cong a^*\nu^*\calL_{\chi} \boxtimes \calL_{\chi}.
	\end{equation}
	Now we can compute the monodromy on the twisted central sheaves:
	\begin{align}
		\;&\Psi_{f_T}(\Sat V_{\nu} \boxtimes \calL_{\chi})\\
		\cong \;&\Psi_{f_T}(\tilde{z}_{\nu, *}(a^*\nu^*\calL_{\chi} \boxtimes \calL_{\chi}))\\
		\cong \;&\tilde{z}_{\nu, *}\Psi_{p_1}((a^*\nu^*\calL_{\chi} \boxtimes \calL_{\chi}))\\
		\cong \;&\nu(-1)_*\Psi_{p_1}((a^*\nu^*\calL_{\chi} \boxtimes \calL_{\chi})).
	\end{align}
	As before, the functor $\nu(-1)_*$ and the constant factor $\boxtimes \calL_{\chi}$ does not affect the monodromy. Therefore, the complex $\Psi_{f_T}(\Sat V_{\nu} \boxtimes \calL_{\chi})$ has the same monodromy as $\Psi_{\on{id}_X}(a^*\nu^*\calL_{\chi})$, which is given by $a^*\nu^*\calL_{\chi}$.

   	\subsection{Computation in equal characteristic: multi-leg version}\label{Subsection_compute_multi}
   	Let $J=\{1, ..., n_0\}$ be a finite set, $\nu=(\nu_1, ..., \nu_{n_0}) \in \bbX_{\bullet}(\sfT)^J$. Let $V_{\nu}$ be the irreducible highest weight representation of $\widehat{T}^J$ of weight $\nu$. We want to compute the iterated nearby cycle 
   	$$\Psi_{f_T}(\Sat(V_{\nu}) \boxtimes \calL_{\chi})$$
   	with respect to the map
   	\begin{equation}
   		f_T: \Gr_{\calT_1, X^J \times x} \to X^J.
   	\end{equation}

   	\subsubsection{The case $\sfT=\bbG_m$}
   	
   	Let us first consider the case $\sfT=\bbG_m$. Identifying $\bbX_{\bullet}(\bbG_m) \cong \bbZ$, $\nu$ corresponds to a tuple of integers $(m_1, ... ,m_{n_0}) \in \bbZ^J$. In this subsubsection we treat the case that $m_j \geq 0$ for every $j \in J$. The general case will be treated in the next subsubsection. We are going to define a map 
   \begin{equation}
   	    \tilde{z}_{\nu}: X^J \times \sfT^{\Diamond} \to \Gr_{\calT_1, X^J \times x}.
   \end{equation}
   For $\varphi: S = \Spa(R, R^+) \to X^J \times \sfT^{\Diamond}$ corresponding to $a_1, a_2, ..., a_{n_0} \in X(R) \subseteq R^+$ and $t' \in \sfT^{\Diamond}(S)=R^*$, we define the element
   \begin{equation}
   	\tilde{z}_{\nu}(\varphi)=(y_1, ..., y_{n_0}, y_{n_0+1}, \calE, \beta, \psi_x) \in \Gr_{\calT_1, X^J \times x}(S)
   \end{equation}
   as follows:
   \begin{enumerate}
   	\item Define $y_j=\varphi_j \; (1 \leq j \leq n_0)$, where $\varphi_j$ is the composition of $\varphi$ with the projection onto the $j$-th $X$. Define $y_{n_0+1}=\underline{x}$.
   	\item Define $\calE=\calO(-\sum_{j \in J}m_j\Gamma_{y_j}) \cong \otimes_{j \in J} \calO(-m_j\Gamma_{y_j})$.
   	\item Define $\beta: \calO(-\sum_{j \in J}m_j\Gamma_{y_j}) \to \calO_{\calY_S}$ the natural inclusion corresponding to 
   	$$b_{\varphi}R^+[[t]] \subseteq R^+[[t]],$$
   	where 
   	\begin{equation}
   		b_{\varphi}=\prod_{j \in J}(a_j-t)^{m_j}  \in R^+[[t]].
   	\end{equation}
   	\item Define $\psi_x: \calE|_{S} \to \calE^0|_{S}$ as the map corresponding to multiplication by $(t'b_{\varphi})^{-1}$:
   	$$b_{\varphi}R^+[[t]] \to R^+[[t]].$$
   \end{enumerate}
   
   Similar to before, we have the following lemma. Denote 
   \begin{equation}
   	X^J_{\on{disj}} \subseteq X^J
   \end{equation}
   the locus of $(y_j)_{j \in J}$ such that the graphs $\Gamma_{y_j}$ are disjoint. For every space $X' \to X^J$, denote
   \begin{equation}
   	X'_{\on{disj}} = X' \times_{X^J} X_{\on{disj}}.
   \end{equation}

   \begin{lemma}
   The map $\tilde{z}_{\nu}$ has the following properties:
   \begin{enumerate}
   	\item Over $x^J \subseteq X^J$, the map 
   	$$\tilde{z}_{\nu, x^J}: x^J \times \sfT^{\Diamond} \cong \sfT^{\Diamond} \to \Gr_{\calT_1, x^J \times x} \cong \sqcup_{n' \in \bbZ}\nu_{n'}\sfT^{\Diamond}$$
   	factors through $\nu_m\sfT^{\Diamond} \cong \sfT^{\Diamond}$, where $m=\sum_{j \in J}m_j$. And the map $\sfT^{\Diamond} \to \sfT^{\Diamond}$ is given by multiplication by $c_{\nu} \in \sfT^{\Diamond}$, where $c_{\nu}=\prod_{j \in J}\nu_{m_i}(-1)=\prod_{j \in J}(-1)^{m_j}=(-1)^m$.
   	\item Over $U=(\eta^J)_{\on{disj}}$, the map 
   	$$\tilde{z}_{\nu, U}: U \times \sfT^{\Diamond} \to \Gr_{\calT_1, U \times x} \cong (\prod_{j \in J}\Gr_{\calT_1, \eta} \times \Gr_{\calT_1, x})|_U \cong (\prod_{j \in J}\underline{\bbZ}_{\eta} \times \sqcup_{n' \in \bbZ}\nu_{n'}\sfT^{\Diamond})|_U$$
   	factors through $((m_1, ..., m_{n_0}) \times \nu_0\sfT^{\Diamond})|_U \cong U \times \sfT^{\Diamond}$, and the map $U \times \sfT^{\Diamond} \to \sfT^{\Diamond}$ is given by the following composition:
   	\begin{equation}
   		\eta^J \times \sfT^{\Diamond} \xrightarrow{(a, ..., a, \id_{\sfT^{\Diamond}})}
   		 (\sfT^{\Diamond})^J \times \sfT^{\Diamond}
   		 \xrightarrow{(\nu_{m_1}, ..., \nu_{m_{n_0}}, \id_{\sfT^{\Diamond}})} (\sfT^{\Diamond})^J \times \sfT^{\Diamond} \xrightarrow{\mu} \sfT^{\Diamond},
   	\end{equation}
   	where $\mu$ is the multiplication.
   \end{enumerate}
   \end{lemma}
   
   \begin{proof}
   	A similar proof as Lemma \ref{Lemma_section2} works.
   \end{proof}

   \subsubsection{The case for general torus $\sfT$}
   Now let $\sfT$ be any torus over $k$. Let $\nu=(\nu_1, ..., \nu_{n_0}) \in \bbX_{\bullet}(\sfT)^J$ as before. Recall that for every map $b: G_1 \to G_2$ of groups and $\calF$ a $G_1$-torsor, we have $b_*\calF$, the induced $G_2$-torosr.
   
   We are going to define a map 
   \begin{equation}
   	\tilde{z}_{\nu}: X^J \times \sfT^{\Diamond} \to \Gr_{\calT_1, X^J \times x}.
   \end{equation}
   
   For $\varphi: S = \Spa(R, R^+) \to X^J \times \sfT^{\Diamond}$ corresponding to $a_1, a_2, ..., a_{n_0} \in X(R) \subseteq R^+$ and $t' \in \sfT^{\Diamond}(S)=\sfT(R)$, we define the element
   \begin{equation}
   	\tilde{z}_{\nu}(\varphi)=(y_1, ..., y_{n_0}, y_{n_0+1}, \calE, \beta, \psi_x) \in \Gr_{\calT_1, X^J \times x}(S)
   \end{equation}
   as follows:
   \begin{enumerate}
   	\item Define $y_j=\varphi_j \; (1 \leq j \leq n_0)$, where $\varphi_j$ is the composition of $\varphi$ with the projection onto the $j$-th $X$. Define $y_{n_0+1}=\underline{x}$.
   	\item Define $\calE=m_{J, *}(\oplus_{j \in J} \nu_{j, *}\calO(-\Gamma_{y_j}))$, where $m_J: \sfT^J \to \sfT$ is the multiplication map. Here we view $\oplus_{j \in J} \nu_{j, *}\calO(-\Gamma_{y_j})$ as a $\sfT^J$-torsor, so $\calE$ is the induced $\sfT$-torsor.
   	\item Let $\beta: m_{J, *}(\oplus_{j \in J} \nu_{j, *}\calO(-\Gamma_{y_j})) \to m_{J, *}(\oplus_{j \in J} \nu_{j, *}\calO_{\calY_S}) \cong \calE^0$ be the natural map, where $\calE^0$ is the trivial $\sfT$-torsor.
   	\item Let $\psi_x=(t')^{-1}m_{J, *}(\oplus_{j \in J}\nu_{j, *}((a_j-t)^{-1})): \calE|_S \to \calE^0|_S$ be the map induced by mulplication by $(t'b_{\varphi})^{-1}$, where $b_{\varphi}=\prod_{j \in J}(\nu_j(a_j-t))$.
   \end{enumerate}

   \begin{lemma}\label{Lemma_section3}
   	The map $\tilde{z}_{\nu}$ has the following properties:
   	\begin{enumerate}
   		\item Over $x^J \subseteq X^J$, the map 
   		$$\tilde{z}_{\nu, x^J}: x^J \times \sfT^{\Diamond} \cong\sfT^{\Diamond} \to \Gr_{\calT_1, x^J \times x} \cong \sqcup_{\nu' \in \bbX_{\bullet}(\sfT)}\nu' \sfT^{\Diamond}$$
   		factors through $(\nu_1 +...+\nu_{n_0})\sfT^{\Diamond} \cong \sfT^{\Diamond}$, and the map $\sfT^{\Diamond} \to \sfT^{\Diamond}$ is given by multiplication by $c_{\nu} \in \sfT^{\Diamond}$, where $c_{\nu}=\prod_{j \in J}\nu_j(-1)$.
   		\item Over $U=(\eta^J)_{\on{disj}}$, the map 
   		$$\tilde{z}_{\nu, U}: U \times \sfT^{\Diamond} \to \Gr_{\calT_1, U \times x} \cong (\prod_{j \in J}\Gr_{\calT_1, \eta} \times \Gr_{\calT_1, x})|_U \cong (\prod_{j \in J}\underline{\bbX_{\bullet}(\sfT)}_{\eta} \times \sqcup_{\nu' \in \bbX_{\bullet}(\sfT)}\nu'\sfT^{\Diamond})|_U$$
   		factors through $(\nu_1, ..., \nu_{n_0}, \nu_0\sfT^{\Diamond})|_U \cong U \times \sfT^{\Diamond}$, where $\nu_0 \in \bbX_{\bullet}(\sfT)$ is the trivial cocharacter. Moreover, the map $U \times \sfT^{\Diamond} \to \sfT^{\Diamond}$ is given by the following composition
   		 $$	\eta^J \times \sfT^{\Diamond} \xrightarrow{(a, ..., a, \id_{\sfT^{\Diamond}})}
   		 (\sfT^{\Diamond})^J \times \sfT^{\Diamond} \xrightarrow{(\nu_{1}, ..., \nu_{n_0}, \id_{\sfT^{\Diamond}})} (\sfT^{\Diamond})^J \times \sfT^{\Diamond} \xrightarrow{\mu} \sfT^{\Diamond},$$
   		 where $\mu$ is the multiplication map.
  \end{enumerate}
   \end{lemma}

   \begin{proof}
   	A similar proof as Lemma \ref{Lemma_section2} works. 
   \end{proof}

   \subsubsection{Computation of the iterated nearby cycle}
   
   Now we can compute the monodromy on the iterated nearby cycle $\Psi_{f_T}(\Sat V_{\nu} \boxtimes \calL_{\chi})$.
   
   Note that the map $\tilde{z}_{\nu}: X^J \times \sfT^{\Diamond} \to \Gr_{\calT_1, X^J \times x}$  is a closed immersion.
   
   In the following lemma, recall that we suppress the cohomological shift which does not change the resulting monodromy.
   
   \begin{lemma}\label{Lemma_zL2}
   Over $U=(\eta^J)|_{\on{disj}}$, we have an isomorphism 
   $$\tilde{z}_{\nu, *}((\boxtimes_{j \in J}a^*\nu_j^*\calL_{\chi}) \boxtimes \calL_{\chi}) \cong \Sat V_{\nu} \boxtimes \calL_{\chi}.$$
   \end{lemma}
   
   \begin{proof}
   	A similar proof as Lemma \ref{Lemma_zLchi} works.
   \end{proof}

    In the following, we ignore the issue of restricting to $(X^J)_{\on{disj}}$ as it does not change the resulting iterated nearby cycle. Indeed, the map $(X^J)_{\on{disj}} \to X^J$ is smooth, and smooth pullback commutes with iterated nearby cycles.
    
    Denote $a^*\nu^*\calL_{\chi}=\boxtimes_{j \in J}a^*\nu_j^*\calL_{\chi}$. We have
     \begin{align}
    	\;&\Psi_{f_T}(\Sat (V_{\nu}) \boxtimes \calL_{\chi})\\
    	\cong \;&\Psi_{f_T}(\tilde{z}_{\nu, *}(a^*\nu^*\calL_{\chi} \boxtimes \calL_{\chi}))  \quad \text{Lemma \ref{Lemma_zL2}}\\
    	\cong \;&\tilde{z}_{\nu, *}\Psi_{p_1}((a^*\nu^*\calL_{\chi} \boxtimes \calL_{\chi}))\\
    	\cong \;&c_{\nu, *}\Psi_{p_1}((a^*\nu^*\calL_{\chi} \boxtimes \calL_{\chi})) \quad \text{Lemma \ref{Lemma_section3}, where $c_{\nu}=\prod_{j \in J} \nu_j(-1) \in \sfT^{\Diamond}$}\\
    	\cong \;&\Psi_{{\id_{X^J}}}(a^*\nu^*\calL_{\chi}) \boxtimes c_{\nu, *} \calL_{\chi}
    \end{align}
    As the constant factor $c_{\nu, *} \calL_{\chi}$ does not affect the monodromy, the monodromy is same as the monodromy of 
    $$\Psi_{{\id_{X^J}}}(a^*\nu^*\calL_{\chi}) = \Psi_{{\id_{X^J}}}(\boxtimes_{j \in J}a^*\nu_j^*\calL_{\chi}),$$
    which is given by $\boxtimes_{j \in J}a^*\nu_j^*\calL_{\chi}$.

	\subsection{Computation in mixed characteristic}\label{Subsection_compute_mixed}
	Now let us assume that $E$ is a finite extension of $\bbQ_p$. 
	
	Let us start with the construction of the sections $z_{\nu}$ in mixed characteristic. 
	
	We start with the case $T=\bbG_m, J=\{1\}, \nu=1$. Recall that we would like to construct a section with respect to the map
	$$f_T: \Gr_{\calT_1, X \times x} \to X,$$
	where $X=\Spd \bOE$, $x=\Spd k=\Spd \overline{\bbF}_p$, and $\Gr_{\calT_1, X \times x}$ is the Beilinson--Drinfeld grassmannian in mixed characteristic. Explicitly, for a test perfectoid space $S \to X$ corresponding to the untilt $S^{\#}$ over $\bOE$, $\Gr_{\calT_1, X \times x}(S)$ classifies tuples $(y_1, y_2, \calE, \beta, \psi_x)$, where
	\begin{enumerate}
		\item $y_1, y_2  \in X(S)$, where $y_1$ corresponds to the untilt $S^{\#}$, $y_2$ corresponds to the trivial untilt.
		\item $\calE$ is a $T$-torsor over $\calY_S$.
		\item $\beta: \calE \xdashrightarrow{} \calE^0$ is a trivialization away from the legs.
		\item $\psi_x: \calE|_{x \times S} \to \calE^0|_{x \times S}$ is an isomorphism.
	\end{enumerate}
	
	Unfortunately, there is no section to the map 
	$$f_T: \Gr_{\calT_1, X \times x} \to X=\Spd \bOE$$
	with the desired properties. However, we can construct such a section after replacing $\bOE$ by an integral perfectoid extension of it.

	Let $\bE_{\infty}$ be the completion of a perfectoid field extension of $\bE$ with ring of integers $\breve{O}_{\infty}$. We are going to construct a section for the base change $\tilde{f_T}$ of $f_T$ along the map
	$\Spd \breve{O}_{\infty} \to X$, i.e., 
	\begin{equation}
		\tilde{f_T}: \Gr_{\calT_1, \Spd \breve{O}_{\infty} \times x} \to \Spd \breve{O}_{\infty}.
	\end{equation}

	Let $S=\Spa (R, R^+)$ be a perfectoid space over $\overline{\bbF}_p$. Let $f: S \to \Spd \breve{O}_{\infty}$ be a map which corresponds to an untilt
	$$f^{\#}: S^{\#}=\Spa (R^{\#}, R^{\#+}) \to \Spa \breve{O}_{\infty}.$$ We get a map
	$$u: \breve{O}_{\infty} \to R^{\#+}.$$
	Applying the tilting functor $(-)^{\flat}$, we get a map
	$$u^{\flat}: \breve{O}_{\infty}^{\flat} \to R^+.$$
	We have a commutative square relating the Fontaine's maps:
	\[\begin{tikzcd}
		{W(\breve{O}_{\infty}^b)} & {\breve{O}_{\infty}} \\
		{W(R^+)} & {R^{{\#}+}}
		\arrow["{\theta_{\breve{O}_{\infty}}}", from=1-1, to=1-2]
		\arrow["{W(u^{\flat})}"', from=1-1, to=2-1]
		\arrow["u", from=1-2, to=2-2]
		\arrow["{\theta_{f}}", from=2-1, to=2-2].
	\end{tikzcd}\]
	Let 
	$$\xi \in {W(\breve{O}_{\infty}^b)}$$
	be a generator of $\ker(\theta_{\breve{O}_{\infty}})$. Let 
	$$\xi_{f}:=W(u^{\flat})(\xi) \in W(R^+).$$
	Note that $\xi_{f}$ is a generator of $\ker(\theta_{f})$ (\cite[Remark 3.11 and proof of Lemma 3.13]{bhatt2018integral}).
	
	Now we construct a section for 
	$$\tilde{f_T}: \Gr_{\calT_1, \Spd \breve{O}_{\infty} \times x} \to \Spd \breve{O}_{\infty}.$$ 
	Namely, we need to construct for every map $f: S \to \Spd \breve{O}_{\infty}$ an element $z_f  \in \Gr_{\calT_1, \Spd \breve{O}_{\infty} \times x}(S)$ such that the collection $\{z_f\}_{f}$ is functorial in $f$: 
	for every commutative diagram
	\[\begin{tikzcd}
		{S_1} && {S_2} \\
		& {\Spd \breve{O}_{\infty}}
		\arrow["\alpha", from=1-1, to=1-3]
		\arrow["{f_1}"', from=1-1, to=2-2]
		\arrow["{f_2}", from=1-3, to=2-2]
	\end{tikzcd}\]
	we have $\alpha^*(z_{f_2})=z_{f_1}$. 
	
	Let $f: S=\Spa(R, R^+) \to \Spd \breve{O}_{\infty}$ be a map which corresponds to an untilt
	$$f^{\#}: S^{\#}=\Spa (R^{\#}, R^{\#+}) \to \Spa \breve{O}_{\infty}.$$ Let $u, u^{\flat}, \xi, \xi_{f}$ be as above. We define 
	$$z_f \in \Gr_{\calT_1, \Spd \breve{O}_{\infty} \times x}(S)$$
	by $(y_1, y_2, \calE, \beta, \psi_x)$, where
	\begin{enumerate}
		\item $y_1, y_2 \in (\Spd \breve{O}_{\infty})(S)$, where $y_1$ corresponds to the untilt $S^{\#}$, $y_2$ corresponds to the trivial untilt $S$.
		\item $\calE:=\calO(-\Gamma_{y_1})$ is the ideal sheaf corresponding to the degree $1$ Cartier divisor defined by $y_1$.
		\item $\beta: \calO(-\Gamma_{y_1}) \subseteq \calO$, i.e., $\beta$ is induced by the map $\ker (\theta_{f}) \subseteq W(R^+)$.
		\item $\psi_x: \calE|_{x \times S} \to \calE^0|_{x \times S}$ is defined as multiplication by $\xi_{f}^{-1}$.
	\end{enumerate}
	
	The collection $\{z_f\}_f$ is functorial in $f$, hence assembles to a section
	\begin{equation}\label{Eq_zmix}
		z=z_1: \Spd \breve{O}_{\infty} \to \Gr_{\calT_1, \Spd \breve{O}_{\infty} \times x}. 
	\end{equation}
	From here, we can repeat the construction in equal-characteristic case to get $z_{\nu}$ for every $T$, $J$, and $\nu \in \bbX_{\bullet}(\sfT)^J$.

	Let us identify $I_E$ with $\Gal(\overline{\bE}/\bE)$ and $P_E$ with $\Gal(\overline{\bE}/\bE^{\on{tm}})$, where $\bE^{\on{tm}}$ is the maximal tamely ramified extension of $\bE$. For every completed algebraic extension $E'$ of $\bE$, the group $\Gamma_{E'}:=\Gal(\overline{E'}/E')$ is a subgroup of $I_E$. Note that $P_E$ is the normal Sylow $p$-subgroup of $I_E$. Hence 
	\begin{equation}
		P_{E'}:=P_E \cap \Gamma_{E'}
	\end{equation}
	is the normal Sylow $p$-subgroup of $\Gamma_{E'}$.
	
	\begin{lemma}
		Let $\bE_{\infty}$ be the completion of a perfectoid field extension of $\bE$. The base change of the nearby cycle $\Psi_{f_T}(\Sat V_{\nu} \boxtimes \calL_{\chi})$ along $(\Spd \breve{O}_{\infty})^J \to X^J$ is tame, i.e., the $P_{\bE_{\infty}}$-action on $\Psi_{f_T}(\Sat V_{\nu} \boxtimes \calL_{\chi})$ is trivial.
	\end{lemma}

	\begin{proof}
		After base change to $(\Spd \breve{O}_{\infty})^J$, we have constructed a section $z=z_1$ as above in Equation \eqref{Eq_zmix}. Then the same computation as the equal-characteristic case applies.
	\end{proof}
	
	\begin{lemma}\label{Lemma_tame_descent}
		Let $\rho$ be a representation of $I_E=\Gal(\overline{E}/E^{\on{un}}) \cong \Gal(\overline{\bE}/\bE)$ with coefficients in $\Lambda$. Assume that for every completed perfectoid extension $\bE_{\infty}$ of $\breve{E}$, $\rho|_{\Gal(\overline{\bE}/\bE_{\infty})}$ is trivial on the wild inertia $P_{\bE_{\infty}}$, then $\rho$ is trivial on $P_E$. 
	\end{lemma}
	
	\begin{proof}
		First, note that it is enough to show that $\rho|_{I_{E'}}$ is trivial on $P_{E'}$ for some tamely ramified finite extension $E'$ of $E$.
		
		By structure theorem of $O_E^*$ (\cite[Proposition 5.7]{neukirch2007ant}), we have
		$$O_E^* \cong \bbZ_p^{\oplus d} \times (E^*)_{\on{tors}},$$
		where $d \in \bbZ_{\geq 1}$ and $(E^*)_{\on{tors}}$ denotes the torsion part of $E^*$. By passing to some sufficiently large tamely ramified extension $E'/E$, we may assume that $d \geq 2$. We choose a surjection 
		\begin{equation}
			h: O_E^* \to \bbZ_p^{\oplus 2}.
		\end{equation}
		
		By local class field theory, we have a surjection
		\begin{equation}
			\Gal(\overline{\bE}/\bE) \to O_E^*.
		\end{equation}
		Composing this with $h$, we get a surjection
		$$(f_1, f_2): \Gal(\overline{\bE}/\bE) \to \bbZ_p^{\oplus 2}.$$
		
		Let $\bE_i$ be the completion of the fixed field of $\ker f_i$, $i=1, 2$. Then $\bE_i/\bE$ is a totally ramified $\bbZ_p$-extension for $i=1, 2$. So $\bE_i$ is a perfectoid field for $i=1, 2$ (see Lemma \ref{Lemma_perfectoid} below).  By assumption, $\rho|_{\Gal(\overline{\bE}/\breve{E}_i)}$ is trivial on $P_{\bE_i}$ for $i=1, 2$.
		
		Moreover, by construction, $\Gal(\overline{\bE}/\breve{E}_1)$ and $\Gal(\overline{\bE}/\breve{E}_2)$ generate $\Gal(\overline{\bE}/\breve{E})$. Hence the normal Sylow $p$-subgroups $P_{\bE_1}$ and $P_{\bE_2}$ generate $P_E$. Therefore, $\rho$ is trivial on $P_E$.
	\end{proof}

	Let $L/\bE$ be a totally ramified $\bbZ_p$-extension. We are going to show that the completion  of $L$ is perfectoid. Let $L_n$ be the fixed field of $\ker(\bbZ_p \to \bbZ/p^n)$.

	\begin{lemma}\label{Lemma_perfectoid}
		Let $L/\bE$ be a totally ramified $\bbZ_p$-extension. Then the completion of $L$ is a perfectoid field. 
	\end{lemma}
	
	\begin{proof}
		Let us recall the definition of a perfectoid ring: A complete Tate ring $R$ is perfectoid if $R$ is uniform and there exists a pseudo-uniformizer $\varpi \in R$ such that $\varpi^p$ divides $p$ in $R^\circ$, and such that the $p$-th power Frobenius map
		$$\Phi: R^\circ/\varpi \to R^{\circ}/\varpi^p$$
		is an isomorphism.
		
		In our case, we take $R=\hat{L}$, the completion of $L$ with respect to the norm on $L$ (normalized by $|p|=p^{-1}$). Note that this is also the same as the $(p)$-adic completion, or the $(\varpi)$-adic completion for any pseudo-uniformizer $\varpi$. The ring $\hat{L}$ is by definition complete. $\hat{L}$ is Tate because $p$ is a topologically nilpotent unit. In our case, $R^{\circ}$ is the completion of $O_L$, hence bounded. So $R$ is uniform. Now we take $\varpi$ small enough such that $\frak{a}_L \subseteq (\varpi^p)$ (cf. proof of \cite[Lemma 13.3.1]{brinon2009cmi}). In particular, $\varpi^p$ divides $p$ because $p \in \frak{a}_L$.
		
		It remains to show that $\Phi$ is an isomorphism. It is enough to show that $\Phi$ is surjective (\cite[Remark 6.1.3]{scholze2020berkeley}). We choose a compatible system of uniformizers $\pi_n \in L_n$ as in Lemma \cite[Lemma 13.3.6]{brinon2009cmi}. Let $$x \in R^{\circ}/(\varpi^p)=\hat{L}^{\circ}/(\varpi^p)=L^{\circ}/(\varpi^p)=\cup_{n \geq 0}O_{L_n}/(\varpi^p),$$
		so we may assume $x \in O_{L_n}$ for some $n$. We can express $x$ by its Teichmuller expansion
		\begin{equation}
			x=\sum_{i \geq 0}[a_i]\pi_n^i, \quad a_i \in k.
		\end{equation}
		We are going to find $y \in O_{L_{n+1}}$ such that $y^p \equiv x \mod (\varpi^p)$. We can express $y$ by its Teichmuller expansion
		\begin{equation}
			y=\sum_{i \geq 0}[b_i]\pi_{n+1}^i, \quad b_i \in k.
		\end{equation}
		Now we have 
		$$y^p \equiv \sum_{i \geq 0}[b_i]^p\pi_{n+1}^{ip} \mod (\varpi^p),$$
		because $p \in (\varpi^p)$. Also, 
		$\pi_{n+1}^p \equiv \pi_n \mod (\varpi^p)$ because $\frak{a}_L \subseteq (\varpi^p)$.
		Therefore, it suffices to solve for every $i \geq 0$,
		$$[b_i]^p=[a_i].$$ 
		But these equations has solutions $b_i \in k$ because $k$ is perfect.
	\end{proof}

	\begin{proposition}
		Let $\sfT$ be any torus over $k$, $J$ a finite set, $\nu = (\nu_j)_{j \in J} \in \bbX_{\bullet}(\sfT)^J$. 
		\begin{enumerate}
			\item The $I_E^J$-action on the complex $\Psi_{f_T}(\Sat V_{\nu} \boxtimes \calL_{\chi})$ is trivial on $P_E^J$.
			\item Under the identification
			$$I_E/P_E \cong \widehat{\bbZ}^{(p)}(1) \cong I_{k_E((t))}/P_{k_E((t))},$$
			the $(I_E/P_E)^J$-action on $\Psi_{f_T}(\Sat V_{\nu} \boxtimes \calL_{\chi})$ is given by $\boxtimes_{j \in J}a^*\nu_j^*\calL_{\chi}$.
		\end{enumerate}
	\end{proposition}
	
	\begin{proof}
		(1) follows from Lemma \ref{Lemma_tame_descent}. (2) follows from (1), because once we know $\Psi_{f_T}(\Sat V_{\nu} \boxtimes \calL_{\chi})$ is trivial on $P_E$, we can compute the monodromy by passing to some totally ramified $\bbZ_p$-extension $\bE_{\infty}$. Over $(\Spd \bO_{\infty})^J$, we can define sections $z_\nu$ analogous to those in the equal-characteristic case. Then the same computation as the equal characteristic case shows that the $(I_E/P_E)^J$-action on $\Psi_{f_T}(\Sat V_{\nu} \boxtimes \calL_{\chi})$ is given by $\boxtimes_{j \in J}a^*\nu_j^*\calL_{\chi}$.
	\end{proof}

	\section{Compatibility with DeBacker--Reeder}\label{Section_DR}
	In this section, we show that for irreducible depth-zero regular supercuspidal representations, the $L$-parameters constructed in \cite{fargues2021geometrization} and \cite{debacker2009depth} agree. Contrast to the previous sections, we use $\overline{\bbQ}_{\ell}$ as the coefficient field. 
	
	Let $\pi$ be an  irreducible depth-zero regular supercuspidal representation of $G(E)$ over $\overline{\bbQ}_{\ell}$. Let $$\varphi_{\pi}^{\on{FS}}: W_E \to \widehat{G}(\overline{\bbQ}_{\ell}), \quad \varphi_{\pi}^{\on{DR}}: W_E \to \widehat{G}(\overline{\bbQ}_{\ell})$$
	be the $L$-parameters of $\pi$ over $\overline{\bbQ}_{\ell}$ constructed by \cite{fargues2021geometrization} and \cite{debacker2009depth}, respectively.

	\subsection{Compatibility on the inertia}
	
	The proof of the following lemma was communicated to the author by Sean Cotner.
	\begin{lemma}\label{Lemma_inertia}
		Assuming Assumption \ref{Assum_sm} and \ref{Assum_CT}. We have $\varphi_{\pi}^{\on{FS}}|_{I_E}=\varphi_{\pi}^{\on{DR}}|_{I_E}$.
	\end{lemma}
	
	\begin{proof}
		Fix a topological generator $\tau_E$ of $I_E/P_E$ as before. By construction of \cite{debacker2009depth}, $\varphi_{\pi}^{\on{DR}}$ is trivial on $P_E$, and $$\varphi_{\pi}^{\on{DR}}(\tau_E)=\on{LS}_{\overline{\bbQ}_{\ell}}(\pi) \in (\widehat{T} \sslash W)^{\widehat{F}}(\overline{\bbQ}_{\ell}).$$
		
		By \cite[proof of Proposition 4.16]{dat2024local}, up to an unramified twist, we may assume that there is some number field $K$ and a $G(E)$-stable $O_K[1/p]$-lattice 
		\begin{equation}
			\pi^0 \subseteq \pi
		\end{equation}
		such that 
		\begin{equation}
			\pi \cong \pi^0 \otimes_{O_K[1/p]} \overline{\bbQ}_{\ell}.
		\end{equation}
		
		By \cite[Subsection 3.10]{dat2025depthzero}, the scheme $(\widehat{T}\sslash W)^{\widehat{F}}$ can be defined over $\bbZ[1/p]$ and is finite over $\bbZ[1/p]$.
		Therefore, enlarging $K$, we may assume that
		\begin{equation}
			(\widehat{T} \sslash W)^{\widehat{F}}(O_K[1/p])=(\widehat{T} \sslash W)^{\widehat{F}}(K)=(\widehat{T} \sslash W)^{\widehat{F}}(\overline{\bbQ})=(\widehat{T} \sslash W)^{\widehat{F}}(\overline{\bbQ}_{\ell}).
		\end{equation} 
		Hence we may view $\varphi_{\pi}^{\on{FS}}(\tau_E)$  and $\on{LS}_{\overline{\bbQ}_{\ell}}(\pi)$ as elements in $(\widehat{T} \sslash W)^{\widehat{F}}(O_K[1/p])$. 
		
		Let $\ell'$ be a prime number not dividing $pm$, where $m=|\pi_0(Z(G))|$ is the cardinality of the connected components of $Z(G)$. Let $\lambda: O_K[1/p] \to \overline{\bbF}_{\ell'}$ be a ring map. Let $\overline{\pi}_{\lambda}$ be an irreducible subquotient of $\pi^0 \otimes_{O_K[1/p], \lambda} \overline{\bbF}_{\ell'}$.
		
		We know that $	\varphi_{\overline{\pi}_{\lambda}}^{\on{FS}}$ is trivial on $P_E$ and
		\begin{equation}\label{Eq_modell'}
			\varphi_{\overline{\pi}_{\lambda}}^{\on{FS}}(\tau_E)=\on{LS}_{\overline{\bbF}_{\ell'}}(\overline{\pi}_{\lambda}) \in (\widehat{T} \sslash W)^{\widehat{F}}(\overline{\bbF}_{\ell' }). 
		\end{equation} 
		
		The $L$-parameter $\varphi_{\pi}^{\on{FS}}$ is trivial on $P_E$ by \cite{dat2025depthzero}. 
		
		Let 
		\begin{equation}
			\on{sp}_{\lambda}: (\widehat{T} \sslash W)^{\widehat{F}}(K)=(\widehat{T} \sslash W)^{\widehat{F}}(O_K[1/p])\to  (\widehat{T} \sslash W)^{\widehat{F}}(\overline{\bbF}_{\ell'})
		\end{equation}
		be the specialization map.
		
		By \cite{scholze2026motivic}, $\varphi_{\pi}^{\on{FS}}$ is  compatible under specialization from $K$-points to $\overline{\bbF}_{\ell'}$-points:
		\begin{equation}\label{Eq_spphiFS}
			\varphi_{\overline{\pi}_{\lambda}}^{\on{FS}}(\tau_E)=\on{sp}_{\lambda}(	\varphi_{\pi}^{\on{FS}}(\tau_E)).
		\end{equation}
		Indeed, we have the following maps over $O_K[1/pm]$:
		\begin{equation}
			\calO(Z^1(W_F, \widehat{G})/\widehat{G}) \to \on{End}(\on{id}_{\Rep G(E)}) \to \on{End}(\pi^0).
		\end{equation}
		We also have
		\begin{equation}
			\on{End}(\pi^0) \cong O_K[1/p] \subseteq \overline{\bbQ}_{\ell} \cong \on{End}(\pi).
		\end{equation}
		Every $f \in \calO(Z^1(W_F, \widehat{G})/\widehat{G})$ acts on $\pi$ via a scalar $\chi_{\pi}(f) \in O_K[1/p] \subseteq \overline{\bbQ}_{\ell} \cong \on{End}(\pi).$
		Hence $f$ acts on $\pi^0 \otimes_{O_K[1/p], \lambda} \overline{\bbF}_{\ell'}$
		via the scalar $\lambda(\chi_{\pi}(f))$. Hence, $f$ acts on $\overline{\pi}_{\lambda}$ via the scalar $\lambda(\chi_{\pi}(f))$. Therefore, we obtain Equation \eqref{Eq_spphiFS} from the construction of the Fargues--Scholze $L$-parameters.

		By \cite[Section 3.10]{dat2025depthzero}, the map $\on{LS}$ is compatible under specialization from $K$-points to $\overline{\bbF}_{\ell'}$-points:
		\begin{equation}
			\on{LS}_{\overline{\bbF}_{\ell'}}(\overline{\pi}_{\lambda}) = \on{sp}_{\lambda}(\on{LS}_{\overline{\bbQ}_{\ell}}(\pi)).
		\end{equation}

		Therefore, Equation \eqref{Eq_modell'} implies that 
		\begin{equation}
			\on{sp}_{\lambda}(\varphi_{\pi}^{\on{FS}}(\tau_E))=\on{sp}_{\lambda}(\on{LS}_{\overline{\bbQ}_{\ell}}(\pi)) \in (\widehat{T} \sslash W)^{\widehat{F}}(\overline{\bbF}_{\ell'})
		\end{equation}
		for every ring map $\lambda: O_K[1/p] \to \overline{\bbF}_{\ell'}$ and every $\ell' \nmid pm$.
		Therefore,
		\begin{equation}
			\varphi_{\pi}^{\on{FS}}(\tau_E)=\on{LS}_{\overline{\bbQ}_{\ell}}(\pi) \in (\widehat{T} \sslash W)^{\widehat{F}}(\overline{\bbQ}_{\ell}).
		\end{equation}
	\end{proof}

	\subsection{Compatibility on the Frobenius}
	Recall that $\tau_E$ is a topological generator of $I_E/P_E$. Let $\on{Fr} \in W_E$ be a lift of Frobenius. Recall that we have
	\begin{equation}\label{Eq_Fr}
		\Fr \; \tau_E \; \Fr^{-1} =\tau_E^q \quad  \text{in $I_E/P_E$}.
	\end{equation}
	
	Recall that $G$ is an unramified connected reductive group over $E$. Let $\widehat{G}$ be its Langlands dual group over $\overline{\bbQ}_{\ell}$. The action of $W_E$ on $\widehat{G}$ is unramified. We denote by $\widehat{F}$ the automorphism of $\widehat{G}$ induced by $\on{Fr}$. 
	
	By Lemma \ref{Lemma_inertia}, we denote
	$$[\phi_{\pi}]:=\varphi_{\pi}^{\on{FS}}|_{I_E}=\varphi_{\pi}^{\on{DR}}|_{I_E} \in (Z^1(I_E/P_E, \widehat{G})\sslash \widehat{G})(\overline{\bbQ}_{\ell}).$$

	Let $Z^0$ be the identity component of the center of $G$. By \cite[Section 1.4, 2.1, and 10.1]{borel1979automorphic}, we have a natural map 
	\begin{equation}
		h: \widehat{G} \to \widehat{Z^0}.
	\end{equation}
	Moreover, we have a short exact sequence (\cite[Exposé XXI, §6.3 and Exposé XXII, Theorem 6.2.1]{demazure1970sga33})
	\begin{equation}
		1 \to (\widehat{G})_{\on{der}} \to \widehat{G} \xrightarrow{h} \widehat{Z^0} \to 1.
	\end{equation}
	If $\omega_{\pi}$ is the central character of $\pi$, let
	\begin{equation}
		\varphi_{\omega_{\pi}}: W_E \to \widehat{Z^0}
	\end{equation}
	denote the $L$-parameter corresponding to $\omega_{\pi}|_{Z^0(E)}$ under the local Langlands correspondence for the torus $Z^0$. Let $[\varphi] \in (Z^1(W_E/P_E, \widehat{G})\sslash \widehat{G})(\overline{\bbQ}_{\ell})$ be an extension of $[\phi_{\pi}]$. The extension $[\varphi]$ of $[\phi_{\pi}]$ is called \textit{compatible with the central character} if 
	\begin{equation}\label{Eq_central_char}
		[h \circ \varphi]=[\varphi_{\omega_{\pi}}] \; \text{in} \; H^1(W_E, \widehat{Z^0}).
	\end{equation}
	Both $\varphi_{\pi}^{\on{FS}}$ and 
	$\varphi_{\pi}^{\on{DR}}$ satisfy Equation \eqref{Eq_central_char}.
	
	The following lemma grows out of a discussion with Arnaud Eteve. Some of the ideas are related to \cite[Chapter 2]{fu2025categorical}.
	
	\begin{lemma}
		There is a unique extension 
		$$[\varphi] \in (Z^1(W_E/P_E, \widehat{G})\sslash \widehat{G})(\overline{\bbQ}_{\ell})$$ of $[\phi_{\pi}]$ such that $[\varphi]$ is compatible with the central character in the sense of \eqref{Eq_central_char}.
	\end{lemma}
	
	\begin{proof}
		Choose a representative $\phi_{\pi}$ of $[\phi_{\pi}]$ and set
		\begin{equation}
			s:=\phi_{\pi}(\tau_{E}), \quad \widehat{T}_s:=Z_{\widehat{G}}(s) \;\text{the centralizer of $s$.}
		\end{equation}
		The regularity of $\pi$ implies that $s$ is regular semisimple (\cite[Subsection 3.4]{debacker2009depth}). Hence $\widehat{T}_s$ is a maximal torus of $\widehat{G}$. Let
		\begin{equation}
			n:=\varphi_{\pi}^{\on{DR}}(\Fr).
		\end{equation}
		The cocycle identity and \eqref{Eq_Fr} give
		\begin{equation}\label{Eq_phi(Fr)}
			n\widehat{F}(s)n^{-1}=s^q.
		\end{equation}
		Consequently, 
		\begin{equation}
			\theta: \widehat{T}_s \to \widehat{T}_s, \quad t \mapsto n\widehat{F}(t)n^{-1}
		\end{equation}
		is an automorphism of $\widehat{T}_s$.
		
		Let $[\varphi]$ be a centrally compatible extension of $[\phi_{\pi}]$ and $\varphi$ be a representative of $[\varphi]$. After changing $\varphi$ within its conjugacy class, we may arrange simultaneously that its restriction to $I_E$ is the chosen representative $\phi_{\pi}$ and that 
		\begin{equation}
			h \circ \varphi = h \circ \varphi_{\pi}^{\on{DR}}
		\end{equation}
		as cocycles. Indeed, since $h: \widehat{G} \to \widehat{Z^0}$ is surjective, Equation \eqref{Eq_central_char} gives the required normalization after lifting a $\widehat{Z^0}$-coboundary to $\widehat{G}$; the inertia representatives can then be aligned by an element of $(\widehat{G})_{\on{der}}=\ker h$, because
		$h|_{\widehat{T}_s}: \widehat{T}_s \to \widehat{Z^0}$ is surjective.
		
		Write
		\begin{equation}
			x:=\varphi(\Fr).
		\end{equation}
		Applying the cocycle identity to $\varphi$ shows that $x$ satisfies the analogue of \eqref{Eq_phi(Fr)}. Therefore
		\begin{equation}
			x=tn \quad \text{for some $t \in \widehat{T}_s$}.
		\end{equation}
		Moreover, conjugation of the cocycle by an element $u \in \widehat{T}_s$ replaces $t$ by
		\begin{equation}\label{Eq_Lang}
			t\; L_{\theta}(u), \quad L_{\theta}(u):=u^{-1}\theta(u).
		\end{equation}
		
		The centralizer of the DeBacker--Reeder $L$-parameter is 
		\begin{equation}
			S_{\varphi_{\pi}^{\on{DR}}}=(\widehat{T}_s)^{\theta}.
		\end{equation}
		Since $\varphi_{\pi}^{\on{DR}}$ is elliptic, 
		$$S_{\varphi_{\pi}^{\on{DR}}}/Z(\widehat{G})^{\Gal(\overline{E}/E)}$$
		is finite. Define the subtorus
		\begin{equation}
			\widehat{T}_{s, \on{der}}:=\widehat{T}_s \cap (\widehat{G})_{\on{der}}=\ker(h|_{\widehat{T}_s}).
		\end{equation}
		It is stable under $\theta$, and 
		$$(\widehat{T}_{s, \on{der}})^{\theta}=S_{\varphi_{\pi}^{\on{DR}}} \cap (\widehat{G})_{\on{der}}$$
		is finite: ellipticity reduces the assertion to the finiteness of $Z(\widehat{G}) \cap (\widehat{G})_{\on{der}}$ (\cite[Exposé XXII, Theorem 6.2.1(iv)]{demazure1970sga33}). Thus the twisted Lang map
		$$L_{\theta}: \widehat{T}_{s, \on{der}} \to \widehat{T}_{s, \on{der}}, \quad u \mapsto u^{-1}\theta(u)$$
		has finite kernel. Therefore, $L_{\theta}$ is an isogeny of tori, and in particular it is surjective over $\overline{\bbQ}_{\ell}$.
		
		Finally, we have $h(x)=h(n)$ by compatibility with the central character. So the element $t$ in $x=tn$ belongs to $\widehat{T}_{s, \on{der}}$. By surjectivity of $L_{\theta}$, choose $u \in \widehat{T}_{s, \on{der}}$ such that $L_{\theta}(u)=t^{-1}$. Equation \eqref{Eq_Lang} then conjugates $x$ to $n$. Hence every centrally compatible extension is conjugate to $\varphi_{\pi}^{\on{DR}}$. Existence is supplied by $\varphi_{\pi}^{\on{DR}}$ itself, proving the lemma.

	\end{proof}

	\begin{corollary}\label{Cor_FSDR}	
		Assuming Assumption \ref{Assum_sm} and \ref{Assum_CT}. Then the Fargues--Scholze and DeBacker--Reeder $L$-parameters agree:
		\begin{equation}
			\varphi_{\pi}^{\on{FS}} = \varphi_{\pi}^{\on{DR}}.
		\end{equation}
	\end{corollary}

	\subsection{Stability of the $L$-packets}
	
	The stability of Fargues--Scholze $L$-packet for $\varphi_{\pi}^{\on{FS}}$ is proven in \cite{fu2024stability}. Therefore, by Corollary \ref{Cor_FSDR}, we recover the stability result proven in \cite{debacker2009depth} under Assumption \ref{Assum_sm} and \ref{Assum_CT}.

	\bibliographystyle{alpha}
	\bibliography{reference}
	
\end{document}